\documentclass[11pt, pdftex]{article}

\usepackage{latexsym,amssymb,color}
\usepackage{amsthm}
\usepackage{amsmath}
\usepackage{color}
\usepackage{enumerate}
\usepackage{mathrsfs}
\usepackage{graphicx}
\usepackage{float}
\usepackage{wrapfig}
\usepackage{layout}
\usepackage{cite}
\usepackage{braket}
\usepackage
[draft=false,setpagesize=false,pdfstartview=FitH,
colorlinks=true,linkcolor=blue,citecolor=magenta,pagebackref=true,bookmarks=false]
{hyperref}
\usepackage{lineno}
\newcommand*\patchAmsMathEnvironmentForLineno[1]{
	\expandafter\let\csname old#1\expandafter\endcsname\csname #1\endcsname
	\expandafter\let\csname oldend#1\expandafter\endcsname\csname end#1\endcsname
	\renewenvironment{#1}
	{\linenomath\csname old#1\endcsname}
	{\csname oldend#1\endcsname\endlinenomath}}
\newcommand*\patchBothAmsMathEnvironmentsForLineno[1]{
	\patchAmsMathEnvironmentForLineno{#1}
	\patchAmsMathEnvironmentForLineno{#1*}}
\AtBeginDocument{
	\patchBothAmsMathEnvironmentsForLineno{equation}
	\patchBothAmsMathEnvironmentsForLineno{align}
	\patchBothAmsMathEnvironmentsForLineno{flalign}
	\patchBothAmsMathEnvironmentsForLineno{alignat}
	\patchBothAmsMathEnvironmentsForLineno{gather}
	\patchBothAmsMathEnvironmentsForLineno{multline}
}

\theoremstyle{definition}
\newtheorem{Definition}{Definition}[section]
\theoremstyle{plain}
\newtheorem{Theorem}{Theorem}[section]
\newtheorem{Corollary}{Corollary}[section]
\newtheorem{Lemma}{Lemma}[section]
\newtheorem{Proposition}{Proposition}[section]
\theoremstyle{remark}
\newtheorem{Remark}{\bf {\rm{{\bf Remark}}}} [section]
\allowdisplaybreaks[1]

\newenvironment{theorem}{\begin{Theorem}$\!\!\!$}{\end{Theorem}}
\newenvironment{lemma}{\begin{Lemma}$\!\!\!$}{\end{Lemma}}
\newenvironment{proposition}{\begin{Proposition}$\!\!\!$}{\end{Proposition}}

\newenvironment{remark}{\begin{Remark}$\!\!\!$}{\end{Remark}}

\newenvironment{definition}{\begin{Definition}$\!\!\!$}{\end{Definition}}

\numberwithin{equation}{section}

\def\wh#1{\widehat{#1}}

\newcommand{\R}{\mathbb{R}}

\newcommand{\RN}{\mathbb{R}^N}
\newcommand{\HS}{\mathbb{R}^{N+1}_+}

\newcommand{\e}{\varepsilon}
\newcommand{\la}{\left\langle}
\newcommand{\ra}{\right\rangle}

\newcommand{\Hsloc}{H^s_{\mathrm{loc}}}
\newcommand{\Hloc}{H^1_{\mathrm{loc}}}
\newcommand{\dHs}{\dot{H}^s}

\DeclareMathOperator{\diver}{div}
\DeclareMathOperator{\supp}{supp}
\DeclareMathOperator{\dist}{dist}

\usepackage[truedimen,margin=30truemm, bottom=30truemm]{geometry}

\begin{document}
\title{On weak solutions to a fractional Hardy--H\'enon equation:\\ Part I: Nonexistence}
\author{
Shoichi Hasegawa\\
Department of Mathematics, 
School of Fundamental Science and Engineering,\\
Waseda University,\\
3-4-1 Okubo, Shinjuku-ku, Tokyo 169-8555, Japan
\\
\\
Norihisa Ikoma\\
Department of Mathematics,
Faculty of Science and Technology,\\
Keio University,\\
3-14-1 Hiyoshi, Kohoku-ku, Yokohama, Kanagawa 223-8522, Japan
\\
\\
Tatsuki Kawakami\\
Applied Mathematics and Informatics Course,
Faculty of Advanced Science and Technology,\\ 
Ryukoku University,\\
1-5 Yokotani, Seta Oe-cho, Otsu, Shiga 520-2194, Japan
}
\date{}
\maketitle

\begin{abstract}
This paper and \cite{HIK-20} treat the existence and nonexistence of stable (resp. outside stable) 
weak solutions 
to a fractional Hardy--H\'enon equation 
$(-\Delta)^s u = |x|^\ell |u|^{p-1} u$ in $\mathbb{R}^N$, 
where $0 < s < 1$, $\ell > -2s$, $p>1$, $N \geq 1$ and $N > 2s$. 
In this paper, the nonexistence part is proved for the Joseph--Lundgren subcritical case. 
\end{abstract}


\noindent \textbf{AMS Mathematics Subject Classification System 2020}: 35R11, 35J61, 35B33, 35B53, 35D30.
\section{Introduction}
\label{section:1}

In this paper and \cite{HIK-20}, we consider a fractional Hardy--H\'enon equation
\begin{equation}
\label{eq:1.1}
(-\Delta)^s u=|x|^\ell |u|^{p-1}u\qquad\mbox{in}\quad{\mathbb R}^N
\end{equation}
and throughout this paper, we always assume the following condition on $s,\ell,p,N$:
	\begin{equation}\label{eq:1.2}
		0 < s < 1, \quad \ell>-2s, \quad p>1,\quad N\ge1,\quad N>2s.
	\end{equation}
In \eqref{eq:1.1}, 
$(-\Delta)^s$ is the fractional Laplacian, which is defined for any $\varphi\in C^\infty_c (\mathbb R^N)$ by
	\[
		(-\Delta)^s\varphi(x)
		:= C_{N,s}\mbox{P.V.}\int_{\mathbb R^N}\frac{\varphi(x)-\varphi(y)}{|x-y|^{N+2s}}\,dy
		=C_{N,s}\lim_{\varepsilon\to0}\int_{|x-y|>\varepsilon}\frac{\varphi(x)-\varphi(y)}{|x-y|^{N+2s}}\,dy
\]
for $x\in \mathbb R^N$,  where P.V. stands for the Cauchy principal value integral, 
\begin{equation}
\label{eq:1.3}
C_{N,s}:=2^{2s}s(1-s)\pi^{-\frac{N}{2}}\frac{\Gamma(\frac{N+2s}{2})}{\Gamma(2-s)}
\end{equation}
and $\Gamma(z)$ is the gamma function.

	This paper and \cite{HIK-20} are motivated by previous work in \cite{Farina,DDG,DDW,H} and 
we shall study the existence and nonexistence of stable solutions to \eqref{eq:1.1}. 
Farina \cite{Farina} studied \eqref{eq:1.1} in the case $s=1$, $\ell = 0$ and $N \geq 2$ and he showed 
the existence ($ p \geq p_c(N) $) and nonexistence ($1 < p < p_c(N)$) of 
stable (resp. stable outside a compact set) solutions where the Joseph--Lundgren exponent $p_c(N)$ is defined by 
	\[
		p_c(N) := \left\{\begin{aligned}
			&\frac{ (N-2)^2 - 4N + 8 \sqrt{N-1} }{(N-2)(N-10)} 
			& &\text{if } N \geq 11,
			\\
			&\infty & &\text{if } 1 \leq N \leq 10.
		\end{aligned}\right.
	\]
Next, Dancer, Du and Guo \cite[Theorem 1.2]{DDG} and Wang and Ye \cite[Theorem 1.7]{WY-12} studied 
the case $\ell > -2 $ and $N \geq 2 $ and showed the existence and nonexistence of 
stable and finite Morse index solutions in $H^1_{\rm loc} (\RN) \cap L^\infty_{\rm loc}(\RN)$. 
We remark that in \cite{WY-12}, they treated the weaker class of solutions than 
those in $H^1_{\rm loc} (\RN) \cap L^\infty_{\rm loc} (\RN)$. The threshold on $p$ is given by 
	\[
		p_+(N,\ell) := \left\{\begin{aligned}
			& \frac{ (N-2)^2 - 2 (\ell + 2) (\ell + N) + 2\sqrt{ (\ell +2)^3 (\ell + 2 N - 2) } }
			{(N-2) (N-4 \ell - 10) } & &\text{if} \ N > 10 + 4 \ell ,
			\\
			& \infty & &\text{if} \ 2 \leq N \leq 10 + 4 \ell.
		\end{aligned}\right.
	\]

	On the other hand, the case $s=1/2$, $\ell = 0$ and $N \geq 2$ was treated 
in Chipot, Chleb\'{\i}k, Fila and Shafrir \cite{CCFS} as an extension problem 
and it is shown that there exists a positive radial solution to \eqref{eq:1.1} for 
$p \geq \frac{N+1}{N-1} = p_S(N,0)$ where $p_S(N,\ell)$ is defined in \eqref{eq:1.6}. 
Harada \cite{H} considered the same case $s=1/2$, $\ell = 0$ and $N \geq 2$, 
introduced the notion corresponding to the Joseph--Lundgren exponent $p_c(N)$ and 
proved the existence of a family of layered positive radial solutions 
when $p$ is the Joseph--Lundgren supercritical or critical. 
In \cite{H}, the subcritical case is also treated. 
D\'{a}vila, Dupaigne and Wei \cite{DDW} dealt with the case $\ell = 0$ and $0<s<1$, and 
proved the existence and nonexistence of stable (resp. stable outside a compact set) solutions of \eqref{eq:1.1}. 
We remark that in \cite{DDW}, they treated solutions 
$u \in C^{2\sigma} (\RN) \cap L^1( \RN , (1+|x|)^{-N-2s}dx )$ 
(the sign of the weight in \cite[Theorem 1.1]{DDW} might be a missprint)
 where $\sigma > s$ 
and 
	\[
		L^q(\RN , w(x)dx) := \Set{ u : \RN \to \R | \| u \|_{L^q(\RN , w(x)dx)} 
		:= \left( \int_{  \RN} |u(x)|^q w(x) \,dx \right)^{\frac{1}{q}} < \infty }.
	\]
However, in order to make their argument work, 
it seems appropriate to assume 
$ u \in C^{2\sigma} (\RN) \cap L^2 (\RN , (1+|x|)^{-N-2s}dx ) $. 
For this point, 
see Remark \ref{Remark:2.4} and \cite[Lemmata 2.1--2.4]{DDW}. 
Notice that $L^2(\RN, (1+|x|)^{-N-2s}dx ) \subset L^1(\RN , (1+|x|)^{-N-2s}dx)$ 
since $(1+|x|)^{-N-2s} \in L^1(\RN)$.

	We also refer to Li and Bao \cite{LB-19} for the study of positive solutions of \eqref{eq:1.1} 
with singularity at $x=0$ in the case $-2s < \ell \leq 0$ and $1 < p \leq p_S(N,\ell)$.

	The aim of this paper and \cite{HIK-20} is to extend the results of \cite{DDW,H} 
into the case $\ell \neq 0$ and established the result which is a fractional counterpart of \cite{DDG}. 
In this paper, we establish the nonexistence result. On the other hand, 
in \cite{HIK-20}, we will consider the existence result and study properties of solutions.

	After submitting this paper, we learned Barrios and Quaas \cite{BQ-20} and 
the references therein from Alexander Quaas. 
In \cite{BQ-20} and Dai and Qin \cite{DQ}, 
they studied the nonexistence of positive solution (and nonnegative nontrivial solutions) 
of \eqref{eq:1.1} for $\ell \in (-2s,\infty)$ and $0<p<p_S(N,\ell)$
On the other hand, Yang \cite{Ya-15} considered the existence of positive solution of \eqref{eq:1.1} 
via the minimizing problem for $p=p_S(N,\ell)$. 
Finally, Fazly--Wei \cite{FW-16} considered \eqref{eq:1.1} for the case $0<\ell$ and $1 < p \leq p_S(N,\ell)$. 
They proved the nonexistence of stable solutions for $1<p< p_S(N,\ell)$, 
and for $p=p_S(N,\ell)$, they obtained the same result under the finite energy condition for solutions. 
See also the comments after Theorem \ref{Theorem:1.1}.

	We first introduce the notation of solutions of \eqref{eq:1.1}. 
\begin{definition}
\label{Definition:1.1}
Suppose \eqref{eq:1.2}. We say that $u$ is a \emph{solution} of \eqref{eq:1.1} 
if $u$ satisfies $u\in \Hsloc(\mathbb R^N)\cap L^\infty_{\rm loc}(\mathbb R^N) \cap L^1(\RN, (1+|x|)^{-N-2s}dx) $ 
and 
\begin{equation}
\label{eq:1.4}
\langle u,\varphi\rangle_{\dot H^s(\mathbb R^N)}
=\int_{\mathbb R^N}|x|^\ell |u|^{p-1}u\varphi\, dx
\quad\mbox{for all}\,\,\, \varphi\in C^\infty_c(\mathbb R^N)
\end{equation}
where
\begin{equation}
\label{eq:1.5}
\langle u,\varphi\rangle_{\dot H^s(\mathbb R^N)}
:=\frac{C_{N,s}}{2}\int_{\mathbb R^N\times\mathbb R^N}\frac{(u(x)-u(y))(\varphi(x)-\varphi(y))}{|x-y|^{N+2s}}\,dx\,dy
\end{equation}
and $C_{N,s}$ is the constant defined by \eqref{eq:1.3}. 
Remark that 
$|x|^\ell |u(x)|^{p-1} u(x) \in L^1_{\rm loc} (\RN)$ 
due to $u \in L^\infty_{\rm loc}(\RN)$ 
and \eqref{eq:1.2}. For $\Omega \subset \RN$, we also set 
	\[
		\| u \|_{\dHs (\Omega) } 
		:= 
		\left( \frac{C_{N,s}}{2} \int_{ \Omega \times \Omega } 
		\frac{ |u(x) - u(y)|^2 }{|x-y|^{N+2s}} \, dx \, dy 
		 \right)^{ \frac{1}{2} }. 
	\]
\end{definition} 
\begin{remark}
\label{Remark:1.1}
In Section \ref{section:2}, we will see that 
	\begin{enumerate}
		\item 
		$\la u, \varphi \ra_{\dHs(\RN)} \in \R $ 
		for any $u \in \Hsloc(\RN)  \cap L^1(\RN, (1+|x|)^{-N-2s}dx) $ 
		and $\varphi \in C^\infty_c (\RN)$. 
		\item
		we may replace $C^\infty_{c}(\RN)$ in Definition \ref{Definition:1.1} by $C^1_c(\RN)$.
		\item
		our solution $u$ satisfies \eqref{eq:1.1} in the distribution sense, that is, 
			\[
				\int_{\RN} u \left( -\Delta \right)^s\varphi dx 
				= \la u , \varphi \ra_{\dHs(\RN)} = \int_{\RN} |x|^\ell |u|^{p-1} u \varphi \, dx 
				\quad \text{for every $\varphi \in C^\infty_c(\RN)$}. 
			\]
	\end{enumerate}
For the details, see Lemma \ref{Lemma:2.1}. 
\end{remark}

	In order to state our main result and for later use, 
following \cite{DDW}, we introduce some notation. We put 
	\[
		\begin{aligned}
			B_R &:= \Set{ x \in \RN | \, |x| <R }, & S_R &:= \partial B_R = \set{x \in \RN | \, |x| =R},
			\\
			B_R^+ &:= \Set{ (x,t) \in \HS | \, \left| (x,t) \right|<R }, &
			S_R^+ &:= \partial B_R^+ = \Set{ (x,t) \in \HS | \, \left| (x,t) \right| = R },
		\end{aligned}
	\]
and $B_R^c:=\mathbb R^N\setminus B_R$. 
For $N\ge1$, $s\in(0,1)$ and $\ell>-2s$, we write
	\begin{equation}\label{eq:1.6}
		p_S(N,\ell) := \frac{N+2s+2\ell}{N-2s} \in (1,\infty).
	\end{equation}
Note that $p_S(N,0)$ corresponds to the critical exponent of 
the fractional Sobolev inequality $H^s(\RN) \subset L^{p_S(N,0)} (\RN)$. 
Next, for $\alpha\in[0,(N-2s)/2)$, we set 
\begin{equation}
\label{eq:1.7}
\lambda(\alpha)
:=2^{2s}\frac{\Gamma(\frac{N+2s+2\alpha}{4})\,\Gamma(\frac{N+2s-2\alpha}{4})}
{\Gamma(\frac{N-2s-2\alpha}{4})\,\Gamma(\frac{N-2s+2\alpha}{4})}.
\end{equation}
It is known that 
	\begin{equation}\label{eq:1.8}
		\text{the function $\alpha\mapsto \lambda(\alpha)$ is strictly decreasing}
	\end{equation}
and $\lambda(\alpha)\to0$ as $\alpha \nearrow (N-2s)/2$ 
(see, e.g. Frank, Lieb and Seiringer \cite[Lemma 3.2]{FLS-08} and 
D\'avila, Dupaigne and Montenegro \cite[Appendix]{DDM}).

	\begin{remark}
Let $v_\alpha(x) := |x|^{ - \left( \frac{N-2s}{2} - \alpha \right) }$ for $ 0 \leq \alpha < \frac{N-2s}{2}$. 
According to Fall \cite[Lemma 4.1]{Fall}, the constant $\lambda(\alpha)$ appears in the equation  
	\[
		(-\Delta)^{s} v_\alpha = \lambda(\alpha) |x|^{-2s} v_\alpha \quad 
		\text{in} \ \RN\setminus\{0\}.
	\]
	\end{remark}

	Finally, we introduce the notation of stable, stable outside a compact set and Morse index equal to $K$: 
	\begin{definition}\label{Definition:1.2}
		Let $u \in \Hsloc (\RN) \cap L^\infty_{\rm loc}(\RN) \cap L^1(\RN, (1+|x|)^{-N-2s}dx) $ 
		be a solution of \eqref{eq:1.1}. 
		We say that $u$ is \emph{stable} if $u$ satisfies 
			\begin{equation}\label{eq:1.9}
				p \int_{\RN} |x|^\ell |u|^{p-1} \varphi^2 \, dx 
				\leq \| \varphi \|_{\dHs(\RN)}^2 \quad \text{for every $\varphi \in C^\infty_c(\RN )$}. 
			\end{equation}
		On the other hand, $u$ is called \emph{stable outside a compact set} if 
		there exists an $R_0 \geq 0$ such that 
			\begin{equation}
			\label{eq:1.10}
			p\int_{\mathbb R^N}|x|^\ell|u|^{p-1}\varphi^2\,dx\le \|\varphi\|^2_{\dot H^s(\mathbb R^N)}
			\quad \text{for every $\varphi \in C^\infty_c(\RN \setminus \overline{B_{R_0}}  )$}.
			\end{equation}
		Finally, a solution $u$ is said to have a 
		\emph{Morse index equal to $K$} provided $K$ is the maximal dimension of subspaces 
		$Z \subset C^\infty_c ( \RN)$ with 
			\[
				\| \varphi \|_{\dHs(\RN)}^2 -  p \int_{  \RN} |x|^\ell |u|^{p-1} \varphi^2 \, dx < 0
				\quad \text{for each $\varphi \in Z \setminus \{0\}$}.
			\]
	\end{definition}

	\begin{remark} \label{Remark:1.3} 
		\begin{enumerate}
			\item 
				By a density argument, in Definition \ref{Definition:1.2}, we may replace $C_c^\infty(\RN)$ 
				and $C^\infty_c(\RN \setminus \overline{ B_{R_0}} )$ 
				by $C_c^1(\RN)$ and $C^1_c( \RN \setminus \overline{ B_{R_0} }   )$. 
				In addition, \eqref{eq:1.10} remains true 
				for $\varphi \in C^1_c(\RN)$ with $\varphi \equiv 0$ on $B_{R_0}$. 
			\item 
				As in \cite[Remark 1]{Farina}, we may check that if a solution $u$ has a finite Morse index, 
				then $u$ is stable outside a compact set. 
		\end{enumerate}
	\end{remark}

	The following is the main result of this paper:

\begin{theorem}
\label{Theorem:1.1}
Suppose \eqref{eq:1.2} and let 
$u\in \Hsloc(\mathbb R^N)\cap L^\infty_{\rm loc}(\mathbb R^N) \cap L^2(\RN, (1+|x|)^{-N-2s}dx) $ 
be a solution of \eqref{eq:1.1} which is stable outside a compact set. 
\begin{itemize}
\item[\rm(i)]
If $1<p<p_S(N,\ell)$, 
then $u\equiv0$.
\item[\rm(ii)]
If $p=p_S(N,\ell)$, then $u$ has finite energy, that is
	\[
\|u\|^2_{\dot H^s(\mathbb R^N)}=\int_{\mathbb R^N}|x|^\ell|u|^{p+1}\,dx<+\infty.
	\]
Furthermore, if $u$ is stable, then $u\equiv 0$.
\item[\rm(iii)]
If $p_S(N,\ell)<p$ with
\begin{equation}
\label{eq:1.11}
p\,\lambda\bigg(\frac{N-2s}{2}-\frac{2s+\ell}{p-1}\bigg)>\lambda(0),
\end{equation}
then $u\equiv0$,
where $\lambda(\alpha)$ is the function given by \eqref{eq:1.7}.
\end{itemize}
\end{theorem}
\bigskip

As mentioned in the above, it might be necessary to suppose $u \in L^2(\RN, (1+|x|)^{-N-2s} dx)$ in \cite{DDW} 
and also in \cite{FW-16} since a similar argument was used in \cite{FW-16}. 
Taking this point into consideration, we succeed to extend the nonexistence part of \cite{DDW} and \cite{FW-16}
into the case $\ell \in (-2s, 0)$. 
Furthermore, Theorem \ref{Theorem:1.1} may be regarded as a fractional version of a part of 
\cite[Theorem 1.2]{DDG}. 
We remark that we deal with weak solutions and in \cite{DDM,FW-16}, classical solutions were studied. 
On the other hand, when $p_S(N,\ell) < p$ holds and \eqref{eq:1.11} fails to hold, 
then we will show the existence of stable solutions in \cite{HIK-20} and observe 
the properties of those solutions. 
By Remark \ref{Remark:1.3}, Theorem \ref{Theorem:1.1} asserts also that 
there is no solution of \eqref{eq:1.1} with finite Morse index when $1<p<p_S(N,\ell)$ or 
$p_S(N,\ell) < p $ with \eqref{eq:1.11}. 
When $p=p_S(N,\ell)$, we find that any solution of $u$ of \eqref{eq:1.1} with finite Morse index 
satisfies $u \in \dHs (\RN) \cap L^{p+1} (\RN , |x|^{\ell} dx )$.

	The proof of Theorem \ref{Theorem:1.1} is basically similar to \cite{DDW}. 
However, in \cite{DDW}, they use the fact that solutions are of class $C^1$ or smooth, 
for instance, see \cite[the proofs of Theorem 1.1 for $1 < p \leq p_S(n)$ and Theorem 1.4, 
and at the end of proof of Theorem 1.1 for $p_S(n) < p$]{DDW}. 
On the other hand, \eqref{eq:1.1} contains the term $|x|^\ell$ and 
especially 
in the case $ \ell < 0$, solutions of \eqref{eq:1.1} are not of class $C^1$ at the origin. 
Therefore, we need some modifications in the argument. 
In this paper, we first prove a local Pohozaev type identity as in Fall and Felli \cite{FF} 
and exploit it to show Theorem \ref{Theorem:1.1} 
for $1 < p \leq p_S(N,\ell)$. 
In addition, to show the motonicity formula (Lemma \ref{Lemma:4.2}), 
we use the idea in \cite[section 3]{FF} where they studied the Almgren type frequency.

	This paper is organized as follows. 
In subsection \ref{section:2.1}, we investigate the properties of the $s$-harmonic extension 
of functions in $\Hsloc(\RN) \cap  L^1( \RN , (1+|x|)^{-N-2s} dx )$, 
that is, functions satisfying the extension problem. 
Subsection \ref{section:2.2} is devoted to the proof of local Pohozaev identity and 
the energy estimate is done in subsection \ref{section:2.3}. 
Section \ref{section:3} contains the proof of Theorem \ref{Theorem:1.1} for $ 1 < p \leq p_S(N,\ell)$ and 
in section \ref{section:4}, we deal with the case $ p_S(N,\ell) < p$.

\section{Preliminaries}
\label{section:2}
This section is divided into three subsections.
In subsection \ref{section:2.1}, we show properties of functions which belong to 
$\Hsloc(\mathbb R^N) \cap L^1(\RN, (1+|x|)^{-N-2s}dx) $,
and give a relationship between a solution $u$ of \eqref{eq:1.1} and $s$-harmonic functions. 
In subsection \ref{section:2.2}, we recall local regularity estimates for the extension problem.
This estimate is useful to establish the Pohozaev identity in Proposition~\ref{Proposition:2.2}.
Furthermore, applying an argument similar to the one in \cite{DDW},
we also give energy estimates for solutions of \eqref{eq:1.1} in subsection \ref{section:2.3}.

Throughout this paper, by the letter $C$
we denote generic positive constants and they may have different values also within the same line.
Furthermore, we write $X : = (x,t) \in \HS$. 

\subsection{Remark on notion of weak solutions}
\label{section:2.1}

We first prove properties of functions which belong to 
$  \Hsloc(\mathbb R^N) \cap  L^1(\RN, (1+|x|)^{-N-2s}dx) $.

\begin{lemma}
\label{Lemma:2.1}
Let $u\in \Hsloc(\mathbb R^N) \cap  L^1(\RN, (1+|x|)^{-N-2s}dx)$.
Then the following hold.
\vspace{-5pt}
\begin{itemize}
\item[\rm(i)] For any $\psi \in C^\infty_c(\RN)$, $\la u , \psi \ra_{\dHs(\RN)} \in \R$ and 
	\begin{equation}\label{eq:2.1}
		\begin{aligned}
			&\left| \la u , \psi \ra_{\dHs(\RN)} \right| 
			\\
			\leq \ & C(N,s) 
			\left\{ 
			\| u \|_{H^s(B_{2R})} \| \psi \|_{H^s(B_{2R})}
			+ \| u \|_{L^1( \RN, (1+|x|)^{-N-2s} dx )} 
			\| \psi \|_{L^1(B_{2R})}
			 \right\}
		\end{aligned}
	\end{equation}
where $\supp \psi \subset B_R$ with $R \geq 1$ and $C(N,s)$ is a constant depending on $N$ and $s$. 
In addition, 
let $\varphi_1\in C^\infty_c(\mathbb R^N)$ with $\varphi_1(x)\equiv1$ in $B_1$ 
and $\varphi_1(x)\equiv0$ in $B_2^c$,
and set $\varphi_n(x):=\varphi_1(n^{-1}x)$.
Then for any $\psi\in C^\infty_c(\mathbb R^N)$,
\begin{equation*}
\langle \varphi_nu,\psi\rangle_{\dot H^s(\mathbb R^N)}\to \langle u,\psi\rangle_{\dot H^s(\mathbb R^N)}
\quad{\rm as}\quad n\to\infty.
\end{equation*}
In particular, 
	\begin{equation}\label{eq:2.2}
		\la u , \psi \ra_{\dHs(\RN)} = \int_{\RN} u (-\Delta)^s \psi \, dx \quad 
		\text{for each $\psi \in C^\infty_{c}(\RN)$}.
	\end{equation}
\item[\rm(ii)]
Put
\begin{equation}
\label{eq:2.3}
P_s (x,t):= p_{N,s}\frac{t^{2s}}{\left( |x|^2 + t^2 \right)^{\frac{N+2s}{2}}},\quad 
U(x,t) := (P_s (\cdot, t)*u ) (x)
\end{equation}
where $p_{N,s}>0$ is chosen so that $\| P_s(\cdot, t) \|_{L^1(\RN)} = 1$. 
Then
\begin{equation}
\label{eq:2.4}
		 -\diver \left( t^{1-2s} \nabla U \right)=0\quad{\rm in}\,\,\,\mathbb R^{N+1}_+, \quad 
		U(x,0)=u(x), 
		\quad
		U \in \Hloc \left( \overline{\HS} ,t^{1-2s}dX \right) 
\end{equation}
and for each $\psi\in C^\infty_c(\overline{\mathbb R^{N+1}_+})$ with $\partial_t\psi(x,0)=0$,
\begin{equation}
\label{eq:2.5}
	\begin{aligned}
		-\lim_{t\to+0}\int_{\mathbb R^N}t^{1-2s}\partial_tU(x,t)\psi(x,t)\,dx
		&=
		\kappa_s\int_{\mathbb R^N}u(x)(-\Delta)^s\psi(x,0)\,dx 
		\\
		&= \kappa_s \la u , \psi (\cdot, 0) \ra_{\dHs(\RN)}. 
	\end{aligned}
\end{equation}
Here 
	\[
		\begin{aligned}
			& \Hloc \left( \overline{\HS} ,t^{1-2s}dX \right) 
			\\
			:= \ & \Set{ V : \HS \to \R |  
			\int_{B_R^+} t^{1-2s} \left\{ |\nabla V|^2 + V^2 \right\} dX < \infty 
			\ \ \text{for all $R > 0$}}
		\end{aligned}
	\]
and
\begin{equation}
\label{eq:2.6}
\kappa_s:=\frac{\Gamma(1-s)}{2^{2s-1}\Gamma(s)}.
\end{equation}
\end{itemize}
\end{lemma}

\begin{remark}
	\label{Remark:2.1}
	\begin{enumerate}
	\item
	In Lemma \ref{Lemma:2.2} and Remark \ref{Remark:2.2}, 
	we will see that \eqref{eq:2.5} holds for every $\psi \in C^\infty_c ( \overline{ \HS} )$ 
	if we assume that $u$ is a solution of \eqref{eq:1.1}. 
	\item
		Here we collect properties on the Poisson kernel $P_s(x,t)$. 
		For the properties below, see, for instance, \cite{CS,DiNPV-12,FF,FF-15,JLX}. 
		First of all, it is known that the Fourier transform of $P_s(x,t)$ is given by 
		$\widehat{P}_s(\xi,t) = \theta_0 \left( 2\pi |\xi| t \right)$ where 
			\[
				\begin{aligned}
					& \wh{v} (\xi) := \int_{  \RN} v(x) e^{ - 2\pi i x \cdot \xi } \, dx 
					\quad \text{for $v \in L^1 (\RN)$},
					\\
					&\theta_0(t) : = \frac{2}{\Gamma(s)} \left( \frac{t}{2} \right)^s
					K_s(t), \quad \theta_0'' + \frac{1-2s}{t} \theta_0' - \theta_0 = 0 \quad 
					\text{in} \ (0,\infty), \quad 
					\theta_0 (0) = 1, 
				\end{aligned}
			\]
		$K_s(t)$ is the modified Bessel function of the second kind with order $s$ and 
			\[
				\kappa_s = \int_{0}^\infty t^{1-2s} \left\{ (\theta_0'(t))^2 + \theta_0^2(t) \right\} \,dt 
				= \frac{\Gamma(1-s)}{2^{2s-1} \Gamma(s) }.
			\]
		By these properties, it is possible to prove that for each $v \in H^s(\RN)$, 
			\[
				\int_{\HS} t^{1-2s} \left| \nabla ( P_s (\cdot, t) \ast v )(x) \right|^2 \, dX 
				= \kappa_s \int_{\RN} \left( 4\pi^2|\xi|^2 \right)^s \left| \wh{v} (\xi) \right|^2 \, d \xi 
				= \kappa_s \| v \|_{\dHs(\RN)}^2
			\]
		and for $U(x,t) = (P_s(\cdot , t) \ast u)(x)$ and $u \in \dHs (\RN)$, 
			\begin{equation}\label{eq:2.7}
				\int_{\HS} t^{1-2s} |\nabla U|^2 \, dX = \kappa_s \| u \|_{\dHs(\RN)}^2.
			\end{equation}
		Furthermore, for each $\zeta \in C^1_c(\overline{\HS})$, 
			\begin{equation}\label{eq:2.8}
				\kappa_s \| \zeta(\cdot,0) \|_{\dHs(\RN)}^2 
				= \int_{  \HS} t^{1-2s} | \nabla ( P_s (\cdot, t) \ast \zeta(\cdot, 0)  ) (x) |^2 \, dX 
				\leq \int_{  \HS} t^{1-2s} | \nabla \zeta |^2 \, dX. 
			\end{equation}
		Finally, for $\varphi \in C^\infty_{c} (\RN)$,
			\[
				- \lim_{t \to +0} t^{1-2s} \partial_t 
				\left( P_s( \cdot, t ) \ast \varphi \right) (x) = \kappa_s \left( -\Delta \right)^s \varphi (x) 
				\quad \text{for any $x \in \RN$}.
			\]
	\end{enumerate}
\end{remark}

\begin{proof}[Proof of Lemma \ref{Lemma:2.1}]
(i) 
We first show $\la u, \psi \ra_{\dHs(\RN)} \in \R$ and \eqref{eq:2.1}. 
Let $\psi \in C^\infty_{c}(\RN)$ with $\supp \psi \subset B_R$ and $R \geq 1$.  
Since $\psi \equiv 0$ on $B_R^c$ and $|x-y| \geq |y|/2$ for $x \in B_R$ and $y \in B_{2R}^c$, 
we see from $R \geq 1$ that 
	\[
		\frac{1+|y|}{|x-y|} \leq2 \frac{1+|y|}{|y|} \leq 4 \quad 
		\quad \text{for each $x \in B_R$ and $y \in B^c_{2R}$}
	\]
and 
	\[
		\begin{aligned}
			& \int_{\RN \times \RN} \frac{\left| u(x) - u(y) \right| \left| \psi(x) - \psi(y) \right|}
			{|x-y|^{N+2s}} \, dx \, d y
			\\
			= \ &
			\left( \int_{B_{2R} \times B_{2R} } + 2 \int_{B_{2R} \times B_{2R}^c}  \right)
			\frac{\left| u(x) - u(y) \right| \left| \psi(x) - \psi(y) \right|}{|x-y|^{N+2s}} \, dx \, d y 
			\\
			\leq \ & 
			\frac{2}{C_{N,s}} \| u \|_{\dHs(B_{2R})} \| \psi \|_{\dHs (B_{2R})} 
			\\
			& \qquad 
			+ 2 \int_{ B_{2R}} dx 
			\int_{|y| \geq 2R} 
			\left( |u(x)| + |u(y)| \right) |\psi(x)| \left( 1 + |y| \right)^{-N-2s} 
			\left( \frac{1+|y|}{|x-y|} \right)^{N+2s} dy 
			\\
			\leq \ & 
			\frac{2}{C_{N,s}} \| u \|_{\dHs(B_{2R})} \| \psi \|_{\dHs (B_{2R})} 
			\\
			& \qquad 
			+ C (N,s)
			\int_{ B_{2R}} dx \int_{ |y| \geq 2R} 
			\left( |u(x)| + |u(y)| \right) |\psi(x)| (1+|y|)^{-N-2s} \, dy
			\\
			\leq \ & 
			\frac{2}{C_{N,s}} \| u \|_{\dHs(B_{2R})} \| \psi \|_{\dHs (B_{2R})} 
			\\
			&\quad 
			+ C (N,s) 
			\left\{ \| u \|_{L^2(B_{2R})} \| \psi \|_{L^2(B_{2R})} 
			+ \| \psi \|_{L^1(B_{2R})} \| u \|_{L^1(\RN, (1+|x|)^{-N-2s} dx )}
			\right\}
		\end{aligned}
	\]
where $C_{N,s}$ is the constant given by \eqref{eq:1.3}. 
Since $\| u \|_{H^s(B_{2R})} < \infty$ due to $u \in \Hsloc (\RN)$, we observe that 
$\la u, \psi \ra_{\dHs(\RN)} \in \R$ and \eqref{eq:2.1} holds.

	The assertion $\la \varphi_n u , \psi \ra_{\dHs(\RN)} \to \la u , \psi \ra_{\dHs(\RN)}$ as $n \to \infty$ 
follows from \eqref{eq:2.1} and $\| \varphi_n u - u \|_{L^1(\RN , (1+|x|)^{-N-2s} dx )} \to 0$.

	Finally we prove \eqref{eq:2.2}. We remark that due to Fall and Weth \cite[Lemma 2.1]{FW}
and $\supp \psi \subset B_R$, there exists a $C_R>0$ such that 
	\[
		\left| \int_{|x-y| > \e } 
		\frac{\psi(x) - \psi(y) }{|x-y|^{N+2s}} dx dy 
		 \right| \leq C_R \| \psi \|_{C^2(\RN)} \left( 1 + |x| \right)^{-N-2s} 
		 \quad \text{for all $x \in \RN$ and $\e \in (0,1)$}.
	\]
Therefore, $\int_{\RN} u (-\Delta )^s \psi \, dx \in \R$ by $u \in L^1(\RN, (1+|x|)^{-N-2s}dx)$. 
Moreover, by the dominated convergence theorem, 
	\[
		\begin{aligned}
			\la u, \psi \ra 
			&= \frac{C_{N,s}}{2} \lim_{\e \to 0} 
			\int_{ |x-y|> \e } 
			\frac{ \left( u(x) - u(y) \right) \left( \psi(x) - \psi (y) \right) }{|x-y|^{N+2s}} dx dy
			\\
			&= \lim_{\e \to 0} C_{N,s} \int_{\RN} 
			\left[ u(x) \int_{|x-y| > \e} \frac{\psi(x) - \psi (y) }{|x-y|^{N+2s}} dy \right] dx 
			= \int_{  \RN} u(x) \left( - \Delta \right)^s \psi (x) dx.
		\end{aligned}
	\]
Hence, (i) holds. 

\vspace{5pt}

(ii) Notice that $U$ is well-defined thanks to $u \in L^1(\RN, (1+|x|)^{-N-2s}dx)$. 
We prove \eqref{eq:2.4}.
The assertion $ -\diver (t^{1-2s} \nabla U) = 0$ in $\mathbb R^{N+1}_+$ 
follows from the fact $ -\diver (t^{1-2s} \nabla P_s ) = 0 $ in $\mathbb R^{N+1}_+$. 
It is also easily seen that $P_s (x, t) \to \delta_0$ as $t \to +0$, hence, $U(x,0) = u(x)$ 
for $u \in C^\infty_{c} (\RN)$. 
For general $u \in \Hsloc (\RN) \cap  L^1(\RN,(1+|x|)^{-N-2s}dx) $, 
the assertion follows from $U \in \Hloc ( \overline{ \HS }, t^{1-2s}dX ) $ (this will be proved below) 
and the existence of the trace operator $\Hloc (\overline{ \HS}, t^{1-2s}dX ) \to \Hsloc (\RN)$ 
(see, for instance, Lions \cite{L-59}, and Demengel and Demengel \cite{DD-12}).

In order to prove $U \in \Hloc ( \overline{\HS}, t^{1-2s}dX )$, 
we will show $U \in H^1( B_R \times (0,R),t^{1-2s}dX )$ for each $R > 0$. 
To this end, let $\varphi_1$ be as in (i), $n_0 > 2R$ and write 
	\[
		U(x,t) = (P_s(\cdot, t)*u ) (x) 
		= ( P_s(\cdot,t)*( \varphi_{n_0} u + (1-\varphi_{n_0}) u ) ) (x) 
		=: U_1(x,t) + U_2(x,t).
	\]

	For $U_1$, we have $\nabla U_1 \in L^2(\HS,t^{1-2s}dX)$ due to $\varphi_{n_0} u \in H^s(\RN)$ and 
Remark \ref{Remark:2.1}. 
In addition, Young's inequality and the fact $\| P_s (\cdot , t) \|_{L^1(\RN)} = 1$ imply 
	\[
		\begin{aligned}
			\int_{0}^R dt\int_{\RN}   t^{1-2s}\left(P_s (\cdot , t) \ast (\varphi_{n_0} u )(x)\right)^2 \, dx 
			&\leq \int_{0}^R t^{1-2s} \| P_s (\cdot ,t ) \|_{L^1(\RN)}^2 \| \varphi_{n_0} u \|_{L^2(\RN)}^2 
			\, dt 
			\\
			& = C_{R,s} \| \varphi_{n_0} u \|_{L^2(\RN)}^2. 
		\end{aligned}
	\]
Hence, $U_1 \in H^1(B_R \times (0,R),t^{1-2s}dX)$ for every $R > 0$.

	On the other hand, for $U_2$, thanks to $n_0 > 2R$, we may write as  
	\[
U_2(x,t) 
=p_{N,s} \int_{|y|\ge{2R}} \frac{t^{2s}}{\left( |x-y|^2 + t^2  \right)^{\frac{N+2s}{2}}}(1-\varphi_{n_0}(y)) u(y)\,dy.
	\]
Noting $|x-y|\ge |y|/2$ for all $|y| \ge n_0$ and $|x| \le R$, 
we see that 
	\[
		\left| U_2(x,t) \right| \leq 
		C t^{2s} \int_{ |y| \geq n_0} |u(y)| (1+|y|)^{-N-2s} \left( \frac{1+|y|}{|x-y|} \right)^{N+2s} dy 
		\leq C t^{2s} \| u \|_{L^1(\RN,(1+|x|)^{-N-2s}dx)}
	\]
and that 
\begin{align*}
|\nabla U_2(x,t)|
&
\le C \int_{|y| \geq n_0} 
\left( \frac{t^{2s-1}}{|x-y|^{N+2s}} + \frac{t^{2s}}{|x-y|^{N+1+2s}} \right) |u(y)| \, dy 
\\
&
\le C \| u \|_{L^1(\RN,(1+|x|)^{-N-2s}dx)} \left( t^{2s-1} + t^{2s} \right)
\end{align*}
for all $(x,t) \in B_R \times (0,R)$. 
This yields $U_2 \in H^1(B_R \times (0,R),t^{1-2s}dX)$, which implies 
$U=U_1 +U_2 \in \Hloc \left( \overline{\HS},t^{1-2s}dX \right)$. 

Next we prove \eqref{eq:2.5}.
Let $\psi \in C^\infty_c(\overline{\mathbb R^{N+1}_+})$ with $\mathrm{supp}\, \psi \subset B_{R/2} \times [0,R] $ 
and $\partial_t \psi(x,0) = 0$. 
Then, by \eqref{eq:2.3} and Fubini's theorem, we have
\begin{equation}
\label{eq:2.9}
\begin{aligned}
\hspace{-7pt}
-\int_{\mathbb R^N}t^{1-2s}\partial_t U(x,t)\psi(x,t)\,dx
&=
-t^{1-2s}\int_{\mathbb R^N}\partial_t 
\left( \int_{\mathbb R^N} \frac{ p_{N,s}t^{2s} u(y) }{(|x-y|^2+t^2)^{\frac{N+2s}{2}}}\,dy \right)\,\psi(x,t)\,dx
\\
&
=
- t^{1-2s}\int_{\mathbb R^N}
\partial_t  \left[ \int_{\mathbb R^N} \frac{p_{N,s} t^{2s} \psi(x,t) }{(|x-y|^{2}+t^2)^{\frac{N+2s}{2}}}\,dx \right]u(y)\,dy
\\
&
\hspace{1cm}
+ t^{1-2s} \int_{\mathbb R^N}\left[\int_{\mathbb R^N} \frac{p_{N,s} t^{2s}\partial_t \psi(x,t) }{ (|x-y|^2 + t^2)^{\frac{N+2s}{2}}}\,dx\right]
u(y)\,dy
\\
&
=\int_{\mathbb R^N}\bigg(I_1(y,t)+I_2(y,t)\bigg)u(y)\, dy
\end{aligned}
\end{equation}
where
\begin{align*}
I_1(y,t)&:=- t^{1-2s}\partial_t  \left[ \int_{\mathbb R^N} \frac{p_{N,s} t^{2s} \psi(x,t) }{(|x-y|^{2}+t^2)^{\frac{N+2s}{2}}}\,dx \right],
\\
I_2(y,t)&:=t^{1-2s}\int_{\mathbb R^N} \frac{p_{N,s} t^{2s}\partial_t \psi(x,t) }{ (|x-y|^2 + t^2)^{\frac{N+2s}{2}}}\,dx.
\end{align*}
From $\psi (x,0) \in C^\infty_c(\RN)$ and Remark \ref{Remark:2.1}, we observe that 
	\begin{equation}\label{eq:2.10}
		-\lim_{t\to +0}t^{1-2s}\partial_t  (P_s(\cdot,t) \ast \psi(\cdot,0) ) (x)=\kappa_s(-\Delta)^s\psi(x,0) 
		\quad \text{for $x \in \RN$}.
	\end{equation}
Notice also that 
\begin{equation}
\label{eq:2.11}
\begin{aligned}
&
I_1(y,t)
\\
= \ &-
t^{1-2s}
\left\{\partial_t  \left[ \int_{\mathbb R^N} \frac{p_{N,s} t^{2s} \left(\psi(x,t) - \psi(x,0) \right) }
{(|x-y|^2+t^2)^{\frac{N+2s}{2}}}\,dx \right]\right.
\\
&
\hspace{6cm}
+ 
\left.
\partial_t  \left[ \int_{\mathbb R^N} \frac{p_{N,s} t^{2s} \psi(x,0) }{(|x-y|^{2}+t^2)^{\frac{N+2s}{2}}}\,dx \right] 
\right\}
\\
= \ & - I_3(y,t) - t^{1-2s}\partial_t  (P_s(\cdot,t) \ast \psi(\cdot,0) ) (y)
\end{aligned}
\end{equation}
where
	\[
I_3(y,t):=-t^{1-2s}\partial_t  \left[ \int_{\mathbb R^N} \frac{p_{N,s} t^{2s} \left(\psi(x,t) - \psi(x,0) \right) }
{(|x-y|^{2}+t^2)^{\frac{N+2s}{2}}}\,dx \right].
	\]
Thus, if we may prove that for all $y \in \RN$
	\begin{align}
		\label{align:2.12}
		&\lim_{t\to+0}\bigg(\left|I_2(y,t)\right|+ \left|I_3(y,t)\right|\bigg)=0 \quad 
		\text{strongly in $L^\infty(\RN)$}, 
		\\
		\label{align:2.13}
		&
		\begin{aligned} 
		&|I_2(y,t)|+|I_3(y,t)|+|t^{1-2s}\partial_t  (P_s(\cdot,t) \ast \psi(\cdot,0) ) (y)|\le C|y|^{-N-2s} 
		\\
		&\hspace{6cm}\text{for each $|y| \geq 2R$ and $t \in (0,1]$},
		\end{aligned}
	\end{align}
then we have \eqref{eq:2.5} by applying the dominated convergence theorem with 
$u \in \Hsloc (\RN) \cap L^1(\RN,(1+|x|)^{-N-2s}dx) $, 
and \eqref{eq:2.10}--\eqref{align:2.13} to \eqref{eq:2.9}.

	We first deal with \eqref{align:2.12}. 
Recall that $ {\rm supp}\, \psi \subset B_{R}$ and $\partial_t \psi (x,0) \equiv 0$. 
Then, 
	\[
		\begin{aligned}
		&
		\frac{\psi(x,t) - \psi(x,0)}{t} 
		= \int_{0}^{1} \partial_t \psi(x,t\theta) \,d\theta 
		= \int_0^1 \partial_t \psi (x,t \theta) - \partial_t \psi (x,0) \, d \theta 
		=:\Psi (x,t) \in C^\infty(\overline{\mathbb R^{N+1}_+}), 
		\\
		&
		\Psi(x,0)=0, \quad |\Psi(x,t)| \le \varphi_R(x) t, 
		\quad 
		\left| \partial_t \Psi(x,t) \right| \le \varphi_R(x) 
		\quad \text{for each $(x,t) \in \mathbb R^N \times [0,1]$}
		\end{aligned}
	\]
where $\varphi_R \in C_c(B_{3R/2})$. 
Hence, we observe from $(|x-y|^2 + t^2)^{1/2} \geq t$ that 
\begin{equation}
\label{eq:2.14}
\begin{aligned}
&
 |I_3(y,t)|
\\
= \ &
 \left|t^{1-2s} \partial_t 
\left[ \int_{\mathbb R^N} \frac{p_{N,s} t^{1+2s} \Psi(x,t) }{(|x-y|^2+t^2)^{\frac{N+2s}{2}}} \,dx \right] 
\right|
\\
\le \ &
C \int_{\mathbb R^N} t \left| \frac{ \Psi(x,t) }{(|x-y|^2+t^2)^{\frac{N+2s}{2}}} \right| \,dx 
+ t^2 \int_{\mathbb R^N} \left| \frac{p_{N,s} \partial_t\Psi(x,t) }{(|x-y|^2+t^2)^{\frac{N+2s}{2}}} \right| \,dx
\\
& + C t^2 \int_{  \RN} 
\left| \frac{\Psi(x,t)}{ (|x-y|^2 + t^2)^{ \frac{N+2s}{2} + \frac{1}{2} }  } \right| \, dx
\\
\le \ &
C \int_{\mathbb R^N} t^2 \frac{\varphi_R(x) }{(|x-y|^2+t^2)^{\frac{N+2s}{2}}} \,dx 
= t^{2-2s} \int_{\mathbb R^N} \frac{\varphi_R(y-tz) }{(|z|^2 +1)^{\frac{N+2s}{2}}} \,dz 
\leq C \| \varphi_R \|_{L^\infty(\RN)} t^{2-2s}
\end{aligned}
\end{equation}
where $C> 0$ is independent of $y \in \RN$. 
Hence, $I_3(\cdot,t) \to 0$ in $L^\infty(\RN)$ as $ t \to 0$.

	Since $|\partial_t \psi(x,t) | \le t\varphi_R(x)$ for all $(x,t) \in \mathbb R^N \times [0,1]$ 
due to $\partial_t\psi(x,0)=0$, in a similar way, we can check that 
$I_2(\cdot,t)  \to 0$ in $L^\infty(\RN)$ as $ t \to 0$
and \eqref{align:2.12} holds.

	Next, we treat \eqref{align:2.13}. 
Let $|y| \geq 2R$ and consider $I_3$.
Noting $|x-y|\ge|y|-|x|\ge |y|/4 $ for all $|y|\ge 2R$ and $|x|\le 3R/2$,
by \eqref{eq:2.14} we see that 
\begin{equation}
\label{eq:2.15}
\begin{aligned}
|I_3(y,t)|
&
\le 
C\int_{\mathbb R^N} \frac{t^2 \varphi_R(x) }{( |x-y|^2 + t^2 )^{\frac{N+2s}{2}}} \,dx 
\\
&
\le
C\int_{|x|\le{3R/2}} t^2  \| \varphi_R \|_{L^\infty(\RN)} |x- y|^{-N-2s} \, dx
\le C t^2 |y|^{-N-2s}.
\end{aligned}
\end{equation}
In a similar way, we may prove 
\begin{equation}
\label{eq:2.16}
|I_2(y,t)|\le C t^2 |y|^{-N-2s}.
\end{equation}
Furthermore,
applying the change of variables and the integration by parts 
and noting $\supp \psi \subset B_R$, we obtain
\begin{equation}
\label{eq:2.17}
\begin{aligned}
t^{1-2s}\partial_t  (P_s(\cdot,t) \ast \psi(\cdot,0) ) (y)
&
=
t^{1-2s} \partial_t \left[ \int_{\mathbb R^N} \frac{p_{N,s} \psi(y-tz,0) }{(|z|^2+1)^{\frac{N+2s}{2}}} \,dz \right] 
\\
&
=
-t^{1-2s} \int_{\mathbb R^N} \frac{p_{N,s} \nabla_y \psi(y-tz,0) \cdot z }{(|z|^2+1)^{\frac{N+2s}{2}}} \,dz
\\
&
= 
\int_{|x|\le R}\frac{p_{N,s} \nabla_x \psi(x,0) \cdot (x-y) }{(|x-y|^2+t^2)^{\frac{N+2s}{2}}} \,dx
\\
&
=-\int_{|x|\le R} 
p_{N,s} \psi(x,0) \mathrm{div}_x \left[ \frac{x-y}{(|x-y|^2+t^2)^{\frac{N+2s}{2}}} \right] \,dx.
\end{aligned}
\end{equation}
Since it is easily seen that for $|y| \ge 2R$
	\[
\int_{|x|\le R} |\psi(x,0) |
\left|  \mathrm{div}_x \left[ \frac{x-y}{(|x-y|^2+t^2)^{\frac{N+2s}{2}}} \right]  \right| \,dx
\le C\int_{|x|\le R} \left|  x - y \right|^{ -N-2s } \,dx \le C |y|^{-N-2s},
	\]
by \eqref{eq:2.17}, we have
	\[
	\left| t^{1-2s} \partial_t \left( P_s(\cdot , t) \ast \psi (\cdot , 0) \right) (y) \right| \leq  C |y|^{-N-2s}.
	\]
This together with \eqref{eq:2.15} and \eqref{eq:2.16} implies \eqref{align:2.13}, 
which completes the proof of Lemma~\ref{Lemma:2.1}. 
\end{proof}
\par
Applying Lemma~\ref{Lemma:2.1}, we have the following.
\begin{lemma}
\label{Lemma:2.2}
Let $u \in \Hsloc (\RN) \cap L^\infty_{\rm loc} (\RN) \cap L^1(\RN, (1+|x|)^{-N-2s} dx )$ 
be a solution of \eqref{eq:1.1} under \eqref{eq:1.2}, and $U$ be the function given in \eqref{eq:2.3}.
Then
\begin{equation}
\label{eq:2.18}
	\begin{aligned}
		\int_{\mathbb R^{N+1}_+}t^{1-2s}\nabla U \cdot\nabla\psi\,dX
		&=\kappa_s\int_{\mathbb R^N}|x|^\ell|u|^{p-1}u\psi(x,0)\, dx 
		= \kappa_s \la u , \psi (\cdot,0) \ra_{\dHs(\RN)}
	\end{aligned}
\end{equation}
for every $ \psi \in C^1_c ( \overline{\R^{N+1}_+} )$ 
where $\kappa_s$ is the constant given in \eqref{eq:2.6}.
\end{lemma}

	\begin{remark}\label{Remark:2.2}
		Since $U \in C^\infty(\HS)$, by \eqref{eq:2.4}, \eqref{eq:2.18} and the integration by parts, 
		for every $\psi \in C^\infty_{c}( \overline{\HS} )$, we have 
			\[
				\begin{aligned}
				- \lim_{\tau  \to +0} \int_{\RN} \tau^{1-2s} \partial_t U (x,\tau) \psi(x,\tau) \, dx 
				&= \lim_{\tau \to +0} \int_{ \RN \times (\tau,\infty)} t^{1-2s} \nabla U \cdot \nabla \psi \, dX 
				\\
				&= \kappa_s \int_{\RN} |x|^\ell |u|^{p-1} u \psi(x,0) \, dx 
				= \kappa_s \la u , \psi (\cdot, 0) \ra_{\dHs(\RN)}.
				\end{aligned}
			\]
		Hence, for solutions $u \in \dHs (\RN) \cap L^\infty_{\rm loc} (\RN) \cap L^1(\RN,(1+|x|)^{-N-2s}dx) $ 
		of \eqref{eq:1.1}, the corresponding extension problem is 
			\[
				\left\{\begin{aligned}
					-\diver \left( t^{1-2s} \nabla U \right) &=0 & &\text{in} \ \HS, \\
					U(x,0) &= u(x) & &\text{on} \ \RN,\\
					- \lim_{t \to +0} t^{1-2s} \partial_t U(x,t) &= \kappa_s |x|^\ell |u|^{p-1} u 
					& &\text{on} \ \RN. 
				\end{aligned}\right.
			\]
	\end{remark}

	\begin{proof}[Proof of Lemma \ref{Lemma:2.2}]
We first prove \eqref{eq:2.18} for $\psi \in C^\infty_c ( \overline{\R^{N+1}_+} )$ with 
$\partial_t \psi(x,0) \equiv 0$ on $\RN$. 
Let $\e>0$, $u$ be a solution of \eqref{eq:1.1} 
and $\psi \in C^\infty_c ( \overline{ \R^{N+1}_+ } )$ with $\partial_t \psi(x,0) = 0$.
Then we have 
\begin{align*}
0 
&
= 
\int_{\mathbb R^N \times (\varepsilon,\infty)} - \diver \left(t^{1-2s} \nabla U \right) \psi (x,t) \,dX 
\\
&
= \int_{\mathbb R^N \times (\varepsilon,\infty)} t^{1-2s} \nabla U \cdot \nabla \psi \,dX 
+ \int_{\mathbb R^N} \varepsilon^{1-2s} \partial_t U(\varepsilon,x) \psi(x,\varepsilon) \,dx.
\end{align*}
By letting $\varepsilon \to 0$ and \eqref{eq:2.5}, it follows that 
\[
\int_{\mathbb R^{N+1}_+} t^{1-2s} \nabla U \cdot \nabla \psi \,dX 
= \kappa_s \int_{\mathbb R^N} u(x) (-\Delta)^s \psi(x,0) \,dx. 
\]
Since $u$ is a solution of \eqref{eq:1.1}, Lemma \ref{Lemma:2.1} yields 
	\[
		\int_{\mathbb R^{N+1}_+} t^{1-2s} \nabla U \cdot \nabla \psi \,dX 
		= \kappa_s \la u , \psi (\cdot, 0) \ra_{\dHs(\RN)} 
		= \kappa_s \int_{\RN} |x|^\ell |u|^{p-1} u \psi(x,0) \, dx
	\]
for every $\psi \in C^\infty_c (\overline{\HS} )$ with $\partial_t \psi (x,0) \equiv 0$.

	Next, we show \eqref{eq:2.18} for $\psi \in C^\infty_c ( \overline{\R^{N+1}_+} )$. 
Let $\psi \in C^\infty_c ( \overline{\mathbb R^{N+1}_+} )$  
and $(\rho_\varepsilon(t))_\varepsilon$ be a mollifier in $t$ with $\rho_\e(-t) = \rho_\e(t)$. 
Set 
	\[
\Psi(x,t) := 
\left\{
\begin{aligned}
& \psi(x,t) & &\text{if $t\geq 0$},
\\
& \psi(x,-t) & &\text{if $t < 0$}.
\end{aligned}
\right.
	\]
Then $\Psi \in C^\infty( \mathbb R^{N+1} \setminus \{ t=0 \} ) \cap W^{1,\infty} (\mathbb R^{N+1})$
and $\Psi(\cdot,t) \in C^\infty_c(\mathbb R^N)$ for every $t\in \R$. 
Define $\Psi_\varepsilon$ by 
	\[
\Psi_\varepsilon(x,t)
:= \int_{\mathbb R} \rho_\varepsilon (t-\tau) \Psi(x,\tau) \,d\tau 
= \int_{\mathbb R} \rho_\varepsilon (\tau) \Psi(x, t - \tau) \,d\tau.
	\]
It is easily seen that $\Psi_\varepsilon \in C^\infty_c ( \mathbb R^{N+1} )$ 
with $\partial_t \Psi_\varepsilon(x,0) = 0$ thanks to the symmetry of $\Psi$ and $\rho_\e$ in $t$. 
Since it holds that for any $k\in\mathbb N$ and $t_1>0$
	\begin{equation}\label{eq:2.19}
		\lim_{\varepsilon\to 0}\bigg(\| \Psi_\varepsilon (\cdot, 0) - \psi(\cdot,0) \|_{C^k(\mathbb R^N)} 
		+ \sup_{ t \ge t_1 >0 } \| \Psi_\varepsilon (\cdot, t) - \psi(\cdot,t) \|_{C^k(\mathbb R^N)}\bigg)=0
	\end{equation}
and $\| \Psi_\varepsilon \|_{W^{1,\infty}(\mathbb R^{N+1}_+)} \le M_0$ for some $M_0>0$, 
we deduce that
\begin{equation}
\label{eq:2.20}
\Psi_\varepsilon \to \psi \quad \text{strongly in $H^1(\mathbb R^{N+1}_+, t^{1-2s}dX)$}.
\end{equation}
Therefore, from
	\[
\int_{\mathbb R^{N+1}_+} t^{1-2s} \nabla U \cdot \nabla \Psi_\varepsilon \,dX 
= \kappa_s \int_{\mathbb R^N} |x|^\ell |u|^{p-1}u \Psi_\varepsilon(x,0) \,dx 
	\]
with \eqref{eq:2.19}, \eqref{eq:2.20} and $u \in L^\infty_{\rm loc} (\RN)$, as $\e \to 0$, 
we have \eqref{eq:2.18} for all $\psi \in C^\infty_c ( \overline{\HS} )$. 
Finally, since we may approximate functions in $C^\infty_c  ( \overline{ \HS} )$ 
by functions in $C^1_c ( \overline{ \HS} )$ 
in the $C^1( \overline{ \HS} )$ sense, \eqref{eq:2.18} holds for every $C^1_c (\overline{ \HS} )$ 
and we complete the proof. 
\end{proof}
%
\subsection{Local Regularity and the Pohozaev identity}
\label{section:2.2}

In this subsection we recall local regularity estimates for the extension problem in Remark \ref{Remark:2.2},
which are taken from \cite[Section~3.1]{FF} (see also \cite{CaSi,JLX}).
Furthermore, we prove the Pohozaev identity for solutions to the extension problem.
\vspace{5pt}
\par 
We first recall local regularity estimates.
\begin{proposition}
\label{Proposition:2.1}
{\rm(Fall and Felli \cite[Proposition~3.2, Lemma~3.3]{FF}, Jin, Li and Xiong \cite[Proposition~2.6]{JLX}, 
Cabr\`e and Sire \cite[Lemma~4.5]{CaSi})}
Let $R_0>0$, $x_0 \in \RN$, $g(x,u) : B_{4R_0}(x_0) \times \R \to \R $ 
and $W \in H^1(B^+_{4R_0} (x,0), t^{1-2s}dX)$ be a weak solution to 
	\[
		\left\{\begin{aligned}
			-\diver \left( t^{1-2s} \nabla W \right) &= 0 
			& &\mathrm{in} \ B_{4R_0}^+(x_0,0),
			\\
			- \lim_{t \to +0} t^{1-2s} \partial_t W (x,t) 
			&= \kappa_s g(x, W(x,0) ) & &\mathrm{on} \ B_{4R_0}(x_0),
		\end{aligned}\right.
	\]
that is, for all $\varphi \in C^\infty_{c}( B_{4R_0}^+(x_0,0) \cup B_{4R_0} (x_0) )$, 
	\[
		\int_{B_{4R_0}^+(x_0)} t^{1-2s} \nabla W \cdot \nabla \varphi \,  d X 
		= \kappa_s \int_{B_{4R_0}(x_0)} g \left( x,W(x,0) \right) \varphi(x,0) \, dx.
	\]
\begin{enumerate}
	\item[{\rm (i)}] 
	Suppose that $g(x,u) := a(x) u + b(x)$ with $a,b \in L^q(B_{4R_0}(x_0))$ for some $q > N/(2s)$. 
	For any $\mu > 0$ there exists a $C=C(N,s,q, \mu , \| a \|_{L^q(B_{4R_0} (x_0) )} )$ such that 
		\[
			\left\| W \right\|_{L^\infty ( B_{2R_0}^+(x_0,0) ) }
			\leq C \left[ \| W \|_{L^{\mu} ( B_{3R_0}^+(x_0,0) ) } + \| b \|_{L^{q}(B_{4R_0} (x_0) )} \right]. 
		\]
	In addition, there exists an $\alpha \in (0,1)$ such that $W \in C^{\alpha} ( \overline{B_{R_0}^+(x_0,0)} )$ 
	and 
		\[
			\| W \|_{C^{\alpha} ( \overline{B_{R_0}^+(x_0,0)  )} } 
			\leq C \left[ \| W \|_{L^\infty (B_{2R_0}^+(x_0,0)) } + \| b \|_{L^q(B_{4R_0}(x_0))}  \right].
		\]
	\item[{\rm (ii)}] 
	Suppose that $W \in C^\alpha ( \overline{B_{2R_0}^+ (x_0,0) } )$ and 
	$g(x,u ) \in C^1( \overline{B_{4R_0} (x_0)} \times \R )$ for some $\alpha \in (0,1)$. 
	Then there exist $\beta \in (0,1)$ and 
	$C = C(N,s, \| g \|_{C^1( \overline{B_{2R_0}^+(x_0,0)} \times [ -A , A ] )})$ 
	where $A := \| W \|_{C^\alpha( \overline{B_{2R_0}^+ (x_0,0)} ) }$ such that 
	$\nabla_x W \in C^\beta ( \overline{B_{R_0}^+(x_0,0)})$ and 
	$ t^{1-2s} \partial_t W \in  C^\beta ( \overline{B_{R_0}^+(x_0,0)})$ with 
		\[
			\left\| \nabla_x W \right\|_{C^\beta ( \overline{B_{R_0}^+(x_0,0)})} 
			+ \left\| t^{1-2s} \partial_t W \right\|_{C^\beta ( \overline{B_{R_0}^+(x_0,0)})} 
			\leq C.
		\]
\end{enumerate}
\end{proposition}

For solutions of \eqref{eq:1.1}, we have 
\begin{lemma}
\label{Lemma:2.3}
Let $u \in \Hsloc (\RN) \cap L^\infty_{\rm loc} (\RN) \cap L^1(\RN, (1+|x|)^{-N-2s}dx)  $ 
be a solution of \eqref{eq:1.1} under \eqref{eq:1.2} and $U$ be the function given in \eqref{eq:2.3}. 
Then for each $R>1$ there exists an $\alpha_R \in (0,1)$ such that 
$U \in C^{\alpha_R} ( \overline{B_R^+} )$ and 
$\nabla_x U , t^{1-2s} \partial_t U \in C^{\alpha_R} ( \overline{(B_R \setminus B_{1/R}) \times (0,R)} )$. 
As a consequence with Remark \ref{Remark:2.2}, 
	\[
		- \lim_{t\to+0}t^{1-2s} \partial_t U(x,t) 
		=\kappa_s |x|^\ell |u(x)|^{p-1} u(x)\quad\mbox{in}\quad C_{\rm loc} (\RN \setminus \{0\}).
	\] 
\end{lemma}

	\begin{proof}
By $u \in L^\infty_{\rm loc}  (\RN)$ and \eqref{eq:1.2}, we find some $q > N/(2s)$ such that 
$a(x) := |x|^\ell |u(x)|^{p-1} \in L^q(B_{4R})$ for each $R > 0$.  
Thus, we may apply Proposition \ref{Proposition:2.1} (i) for $U$ with $a(x)$ and 
there exists an $\alpha_R \in (0,1)$ so that $U \in C^\alpha( \overline{ B_R^+} )$.

	Next, notice that $g(x , u ) := |x|^\ell | u |^{p-1} u \in 
C^1( \overline{ B_R\setminus B_{1/R}} \times \R  )$. 
Therefore, we may apply Proposition \ref{Proposition:2.1} (ii) and obtain the desired result. 
	\end{proof}

Next we prove the following Pohozaev identity.
\begin{proposition}
\label{Proposition:2.2}
Let $u \in \Hsloc (\RN) \cap L^\infty_{\rm loc} (\RN) \cap L^1(\RN, (1+|x|)^{-N-2s}dx)  $ 
be a solution of \eqref{eq:1.1} with \eqref{eq:1.2}, and $U$ be the function given in \eqref{eq:2.3}. 
Then for all $R>0$, there holds
\begin{equation}
\label{eq:2.21}
\begin{aligned}
&
-\frac{N-2s}{2}\left[\int_{B^+_R}t^{1-2s}|\nabla U|^2\,dX
-\frac{2\kappa_s}{N-2s}\frac{N+\ell}{p+1}\int_{B_R}|x|^\ell|u|^{p+1}\,dx\right]
\\
&
\quad
+\frac{R}{2}\left[\int_{S^+_R}t^{1-2s}|\nabla U|^2\,dS
-\frac{2\kappa_s}{p+1}\int_{S_R}|x|^\ell|u|^{p+1}\,d \omega \right]
=R\int_{S^+_R}t^{1-2s}\bigg|\frac{\partial U}{\partial \nu}\bigg|^2\,dS
\end{aligned}
\end{equation}
and
\begin{equation}
\label{eq:2.22}
\int_{B^+_R}t^{1-2s}|\nabla U|^2\,dX
-\kappa_s\int_{B_R}|x|^\ell|u|^{p+1}\,dx
=\int_{S^+_R}t^{1-2s}\frac{\partial U}{\partial \nu}U\,dS.
\end{equation}
Here $\nu = X/|X|$ is the unit outer normal vector of $S_R^+$ at $X$ and 
$\kappa_s$ the constant given in \eqref{eq:2.6}.
\end{proposition}

	\begin{proof}
We follow the argument in \cite[Proof of Theorem~3.7]{FF}. 
Let $u$ be a solution of \eqref{eq:1.1} and $U$ the function given in \eqref{eq:2.3}, and 
we take any $R>0$. 
Then, by \eqref{eq:2.4} we have
\begin{equation}
\label{eq:2.23}
\frac{N-2s}{2}t^{1-2s}|\nabla U|^2
= \diver \left( \frac{1}{2}t^{1-2s}|\nabla U|^2X-t^{1-2s}(X\cdot\nabla U)\nabla U \right)
\end{equation}
for $X\in B^+_R$.
Let $\rho<R$.
Then, integrating \eqref{eq:2.23} over the set
	\[
\mathcal O_\delta:= (B^+_R\setminus\overline{B^+_\rho})\cap
\Set{X = (x,t) \in \mathbb R^{N+1}_+ | t>\delta} \quad 
	\]
with $\delta \in (0,\rho)$ and 
writing $B_{R,\rho, \delta} := B_{ \sqrt{R^2-\delta^2}  } \setminus B_{ \sqrt{\rho^2 - \delta^2} } \subset \RN$, 
we have
\begin{equation}
\label{eq:2.24}
\begin{aligned}
\frac{N-2s}{2}\int_{\mathcal O_\delta}t^{1-2s}|\nabla U|^2\,dX
&=  \int_{\mathcal O_\delta} \diver \left( \frac{1}{2}t^{1-2s}|\nabla U|^2X-t^{1-2s}
(X\cdot\nabla U)\nabla U \right)\,dX
\\
&= -\frac{1}{2}\delta^{2-2s}\int_{ B_{R,\rho,\delta} }
|\nabla U(x,\delta)|^2\,dx
+\delta^{2-2s}\int_{ B_{R,\rho,\delta} }  |U_t(x,\delta)|^2\,dx
\\
& \quad 
+\delta^{1-2s}\int_{ B_{R,\rho,\delta} }  
(x\cdot\nabla_xU(x,\delta))U_t(x,\delta)\,dx
\\
& \quad +\frac{R}{2}\int_{S^+_R\cap [t>\delta] }t^{1-2s}|\nabla U|^2\,dS
-R\int_{S^+_R\cap [t>\delta ]}t^{1-2s}\bigg|\frac{\partial U}{\partial\nu}\bigg|^2\,dS
\\
&
\quad 
-\frac{\rho}{2}\int_{S^+_\rho\cap [t>\delta] }t^{1-2s}|\nabla U|^2\,dS
+\rho\int_{S^+_\rho\cap [t>\delta] }t^{1-2s}\bigg|\frac{\partial U}{\partial\nu}\bigg|^2\,dS.
\end{aligned}
\end{equation}
Now we claim that there exists a sequence $\delta_n\to0$ such that
\begin{equation}
\label{eq:2.25}
\lim_{n\to\infty}\bigg[\frac{1}{2}\delta_n^{2-2s}\int_{B_R}|\nabla U(x,\delta_n)|^2\,dx
+\delta_n^{2-2s}\int_{B_R}|U_t(x,\delta_n)|^2 \,dx \bigg]=0.
\end{equation}
In fact, if there is no such sequence, 
then there exists a $C>0$ such that
	\[
\liminf_{\delta\to0}\bigg[\frac{1}{2}\delta^{2-2s}\int_{B_R}|\nabla U(x,\delta)|^2\,dx
+\delta^{2-2s}\int_{B_R}|U_t(x,\delta)|^2 \, dx \bigg]\ge C
	\]
and thus there exists a $\delta_0>0$ such that for all $\delta \in (0,\delta_0)$, 
	\[
\frac{1}{2}\delta^{1-2s}\int_{B_R}|\nabla U(x,\delta)|^2\,dx
+\delta^{1-2s}\int_{B_R}|U_t(x,\delta)|^2 \, dx \ge\frac{C}{2\delta}. 
	\]
Since $U\in \Hloc ( \overline{\mathbb R^{N+1}_+} ,t^{1-2s}dX)$, 
integrating the above inequality in $\delta$ over $(0,\delta_0)$, we have a contradiction 
and \eqref{eq:2.25} holds.

On the other hand, by Lemma \ref{Lemma:2.3}, we see that
\begin{equation}
\label{eq:2.26}
\begin{aligned}
\lim_{\delta\to0}\delta^{1-2s}\int_{ B_{R,\rho,\delta} }
(x\cdot\nabla_xU(x,\delta))U_t(x,\delta)\,dx
=-\kappa_s\int_{B_R\setminus B_\rho}(x\cdot\nabla_xu)|x|^\ell|u|^{p-1}u\,dx.
\end{aligned}
\end{equation}
By \eqref{eq:2.24}--\eqref{eq:2.26} 
and replacing $\mathcal O_\delta$ with $\mathcal O_{\delta_n}$ for a sequence $\delta_n\to 0$, 
we conclude that 
\begin{equation}
\label{eq:2.27}
\begin{aligned}
\hspace{-5pt}
\frac{N-2s}{2}\int_{B^+_R\setminus B^+_\rho}t^{1-2s}|\nabla U|^2\,dX
&
=
-\kappa_s\int_{B_R\setminus B_\rho}(x\cdot\nabla_xu)|x|^\ell|u|^{p-1}u\,dx
\\
&
\qquad
+\frac{R}{2}\int_{S^+_R}t^{1-2s}|\nabla U|^2\,dS
-R\int_{S^+_R}t^{1-2s}\left|\frac{\partial U}{\partial\nu}\right|^2\,dS
\\
&
\qquad\quad
-\frac{\rho}{2}\int_{S^+_\rho}t^{1-2s}|\nabla U|^2\,dS
+\rho\int_{S^+_\rho}t^{1-2s}\left|\frac{\partial U}{\partial\nu}\right|^2\,dS.
\end{aligned}
\end{equation}
Furthermore, integration by parts yields
\begin{equation}
\label{eq:2.28}
\begin{aligned}
&
- \kappa_s \int_{B_R\setminus B_\rho}(x\cdot\nabla_xu)|x|^\ell|u|^{p-1}u\,dx
\\
= \ & - \kappa_s  \int_{B_R\setminus B_\rho}|x|^\ell x\cdot\nabla_x \left(\frac{|u|^{p+1}}{p+1} \right)\,dx
\\
= \ & \kappa_s \frac{N+\ell}{p+1}\int_{B_R\setminus B_\rho}|x|^\ell|u|^{p+1}\,dx
\\
&\hspace{3cm}
-\frac{\kappa_s}{p+1}\int_{S_R}|x|^{\ell+1}|u|^{p+1}\, d \omega
+\frac{\kappa_s}{p+1}\int_{S_\rho}|x|^{\ell+1}|u|^{p+1}\,d \omega.
\end{aligned}
\end{equation}
Similarly to \eqref{eq:2.25},
since $U\in \Hloc ( \overline{\mathbb R^{N+1}_+} ,t^{1-2s}dX)$ and $u\in L^\infty_{\rm loc}(\mathbb R^N)$,
by \eqref{eq:1.2},  
we may prove that there exists a sequence $\rho_n\to0$ such that
	\[
\lim_{n\to\infty}\left[\rho_n\int_{S^+_{\rho_n}}t^{1-2s}|\nabla U|^2\,dS
+\rho_n\int_{S^+_{\rho_n}}t^{1-2s}\left|\frac{\partial U}{\partial\nu}\right|^2\,dS
+\int_{S_{\rho_n}}|x|^{\ell+1}|u|^{p+1}\,d \omega \right]=0.
	\]
Therefore, taking $\rho=\rho_n$ and letting $n\to\infty$ in \eqref{eq:2.27} and \eqref{eq:2.28},
we have \eqref{eq:2.21}.

	On the other hand, since 
	\[
		- \mathrm{div} \left( t^{1-2s} \nabla U \right) \varphi = 0 \quad 
		\text{in} \ \widetilde{\mathcal{O}}_\delta := B_R^+ \cap \Set{ X \in \HS | t > \delta }
	\]
for any $\varphi \in C^\infty ( \overline{ B_R^+ } )$, integration by parts gives 
	\[
		\int_{  \widetilde{\mathcal{O}}_\delta } 
		t^{1-2s} \nabla U \cdot \nabla \varphi \, dX 
		= \int_{S_R^+ \cap [ t > \delta ] } t^{1-2s} \frac{\partial U}{\partial \nu} \varphi \, d S 
		- \int_{B_{\sqrt{R^2-\delta^2}}} \delta^{1-2s} \frac{\partial U}{\partial t} (x,\delta) 
		\varphi (x,\delta) \, d x.
	\]
For the last term, by decomposing $\varphi$ into 
$\varphi (x,t) = \zeta(x,t) \varphi (x,t) + (1-\zeta(x,t)) \varphi (x,t)$ 
where $\zeta \in C^\infty_c( \overline{B_{R/2}^+} )$ with $\zeta \equiv 1$ on $B_{R/4}^+$ and 
noting $\zeta(x,t) \varphi(x,t) \in C^\infty_{c} ( \overline{ \HS } )$, 
Remark \ref{Remark:2.2} and Lemma \ref{Lemma:2.3} yield 
	\[
		- \int_{B_{\sqrt{R^2-\delta^2}}} \delta^{1-2s} \frac{\partial U}{\partial t} (x,\delta) 
		\varphi (x,\delta) \, d x \to 
		\kappa_s \int_{B_R} |x|^\ell |u|^{p-1} u \varphi (x,0) \, dx 
		\quad \text{as $\delta \to 0$}.
	\]
Therefore, for every $\varphi\in C^\infty(\overline{B^+_R})$,
	\[
\int_{B^+_R}t^{1-2s}\nabla U\cdot\nabla\varphi\,dX
=\int_{S^+_R}t^{1-2s}\frac{\partial U}{\partial\nu}\varphi\,dS
+\kappa_s\int_{B_R}|x|^\ell|u|^{p-1}u \varphi(x,0) \,dx.
	\]
From the fact that $U \in \Hloc (\overline{\HS},t^{1-2s}dX)$ can be approximated by functions of 
$C^\infty(\overline{B^+_R})$ in the $H^1(B^+_R,t^{1-2s}dX)$ sense, 
by setting $\varphi = U$ in the above, 
we obtain \eqref{eq:2.22} and Proposition~\ref{Proposition:2.2} follows.
\end{proof}
%

\subsection{Energy estimates}
\label{section:2.3}

We first show several lemmata by following \cite[Lemma~2.2 and Corollary~2.3]{DDW}. 
\begin{lemma}
\label{Lemma:2.4}
For $\zeta \in W^{1,\infty}(\RN) \cap H^1(\RN) $, define 
	\begin{equation}
	\label{eq:2.29}
	\rho(\zeta; x):=\int_{\mathbb R^N}\frac{(\zeta(x)-\zeta(y))^2}{|x-y|^{N+2s}}\, dy.
	\end{equation}
Then there exists a $C = C(N,s)>0$ such that
	\begin{equation}\label{eq:2.30}
		\begin{aligned}
			\rho \left( \zeta ; x \right) 
			&\leq 
			C 
			\left\{ \left( 1 + |x| \right)^2 \| \nabla \zeta \|_{L^{\infty} (\Omega(|x|)) }^2 
			+ \| \zeta \|_{L^\infty( \Omega (|x|) )}^2 
			\right.
			\\
			& \hspace{3cm} 
			\left.+ \left( 1 + |x| \right)^{-N} \| \zeta \|_{L^2(\RN)}^2
			\right\}  \left( 1 + |x| \right)^{-2s} 
			\qquad \text{for all $|x| \geq 1$}, 
		\end{aligned}
	\end{equation}
and 
	\begin{equation}\label{eq:2.31}
		\rho \left( \zeta ; x\right) \leq 
		C \left( 1 + \| \zeta \|_{W^{1,\infty} (\RN) }^2 \right) 
		\quad \text{for all $|x| \leq 1$}
	\end{equation}
where 
	\[
		\Omega (|x|) := \Set{ y \in \RN | \, |y| \geq \frac{|x|}{2} }.
	\]
\end{lemma}

	\begin{proof}
The proof is basically the same as the one in \cite[Lemma 2.2]{DDW}. 
We treat the case $|x| \geq 1$ and put 
	\[
		\begin{aligned}
			D_1 &:= \Set{ y \in \RN | \, |y-x| \leq \frac{|x|}{2} }, 
			& D_2 &:= \Set{ y \in \RN | \, \frac{|x|}{2} \leq |y-x| \leq 2 |x| },
			\\
			D_3 &:= \Set{ y \in \RN | \, 2|x| \leq |y-x| }. 
		\end{aligned}
	\]
For $D_1$, 
notice that $D_{1}$ is convex and $D_1 \subset \Omega (|x|)$. 
Since it follows from $\zeta \in W^{1,\infty} (\RN)$ that 
\[
|\zeta(x)-\zeta(y)|\le \| \nabla \zeta \|_{L^\infty( D_1 )} |x-y| 
\leq \| \nabla \zeta \|_{L^\infty (\Omega (|x|))  } |x-y| ,
\]
we have
\begin{equation}
\label{eq:2.32}
\begin{aligned}
\int_{D_1}\frac{(\zeta(x)-\zeta(y))^2}{|x-y|^{N+2s}}\,dy
&
\le C \| \nabla \zeta \|_{L^\infty (\Omega (|x|) )}^2 \int_{D_1} |x-y|^{2-N-2s}\,dy 
\\
&
\leq C \| \nabla \zeta \|_{L^\infty (\Omega(|x|))}^2 (1+|x|)^{2-2s}. 
\end{aligned}
\end{equation}
For $y\in D_2$, by $|x| \geq 1$,  it holds that
\begin{equation}
\label{eq:2.33}
\begin{aligned}
\int_{D_2}\frac{(\zeta(x)-\zeta(y))^2}{|x-y|^{N+2s}}\,dy
&
\le  
C|x|^{-N-2s}\int_{|y|\le 3|x|}(\zeta(x)^2+\zeta(y)^2)\, dy
\\
&
\leq 
C \zeta(x)^2 |x|^{-2s} + C |x|^{-N-2s} \| \zeta \|_{L^2(\RN)}^2 
\\
&\leq 
C \| \zeta \|_{L^\infty (\Omega(|x|))}^2 (1+|x|)^{-2s} + C \| \zeta \|_{L^2(\RN)}^2 (1+|x|)^{-N-2s}.
\end{aligned}
\end{equation}
For $y \in D_3$, since $|y| \geq |x|$, we have
\begin{equation}\label{eq:2.34}
\begin{aligned}
\int_{D_3}\frac{(\zeta(x)-\zeta(y))^2}{|x-y|^{N+2s}}\,dy
&
\le C \| \zeta \|_{L^\infty (\Omega (|x|))}^2 \int_{D_3}|x-y|^{-N-2s}\, dy 
\\
&
\leq C \| \zeta \|_{L^\infty (\Omega (|x|))}^2 (1+|x|)^{-2s}.
\end{aligned}
\end{equation}
Putting \eqref{eq:2.32}--\eqref{eq:2.34} together, we have \eqref{eq:2.30} for $|x| \geq 1$.

	On the other hand, for $|x| \leq 1$, 
by dividing $\RN$ into $B_2$ and $B_2^c$, and arguing as in \eqref{eq:2.32} and \eqref{eq:2.34}, 
we obtain \eqref{eq:2.31}. We omit the details and Lemma \ref{Lemma:2.4} holds. 
\end{proof}

\begin{lemma}
\label{Lemma:2.5}
For $m>N/2$, set 
	\begin{equation}
		\label{eq:2.35}
		\eta(x) :=(1+|x|^2)^{-\frac{m}{2}}. 
	\end{equation}
Let $R\ge R_0\ge1$ and 
$\psi\in C^\infty(\mathbb R^N)$ such that $0\le \psi\le 1$, $\psi\equiv 0$ on $B_1$ 
and $\psi\equiv 1$ on $B_2^c$. 
Define $\eta_R$ and $\rho_R$ by 
\begin{equation}
\label{eq:2.36}
\eta_R(x) :=\eta\bigg(\frac{x}{R}\bigg)\psi\bigg(\frac{x}{R_0}\bigg), \quad 
\rho_R(x) := \rho \left( \eta_R; x \right). 
\end{equation}
Then there exists a constant $C=C(N,s,m,R_0)>0$ such that
\begin{equation}
\label{eq:2.37}
\rho_R(x) \leq \left\{\begin{aligned} 
	& C\eta \left(\frac{x}{R} \right)^2|x|^{-N-2s}+2R^{-2s}\rho \left(\eta; \frac{x}{R}\right) 
	& & \text{if} \ |x|\ge 3R_0,\\
	& C  + 2 R^{-2s} \rho \left( \eta ; \frac{x}{R} \right)
	& & \text{if} \ |x| \leq 3R_0.
	\end{aligned}\right.
\end{equation}
\end{lemma}
\begin{remark}
\label{Remark:2.3}
	 By Lemma \ref{Lemma:2.4} and \cite[Lemma 2.2]{DDW}, there exist two constants $c,C>0$ such that 
	 	\begin{equation*}
	 		c \left( 1 + |x| \right)^{-N-2s} \leq 
	 		\rho (\eta ;x) \leq C \left( 1 + |x| \right)^{-N-2s} \quad \text{for all $x \in \RN$}.
	 	\end{equation*}
	 In addition, later (see \eqref{eq:2.58}) we shall also prove that 
	 for all sufficiently large $R>0$ and $x \in \RN$, 
	 	\[
	 		 0 < c_R (1+|x|)^{-N-2s} \leq \rho_R(x).
	 	\]
\end{remark}

	\begin{proof}[Proof of Lemma \ref{Lemma:2.5}]
Applying Young's inequality with the definition of $\eta_R$,
we have
	\begin{equation}\label{eq:2.38}
	\begin{aligned}
		&\rho_R(x)
		 \leq 2 \eta\left( \frac{x}{R} \right)^2 
		\int_{\RN} \frac{\left( \psi \left( \frac{x}{R_0} \right) - \psi \left( \frac{y}{R_0} \right) 
			\right)^2}{|x-y|^{N+2s}} \, dy 
			\\
			&\hspace{4cm}
		+ 2 \int_{\RN} \psi \left( \frac{y}{R_0} \right)^2 
		\frac{ \left( \eta \left( \frac{x}{R} \right) - \eta \left( \frac{y}{R} \right) \right)^2  }{|x-y|^{N+2s}} \, dy.
		\end{aligned}
	\end{equation}
For the first term, if $|x|\ge 3R_0$, then $|x-y| \geq |x|/3$ for any $ y \in B_{2R_0}$ and 
	\[
		\int_{\RN} \frac{\left( \psi \left( \frac{x}{R_0} \right) - \psi \left( \frac{y}{R_0} \right) 
			\right)^2}{|x-y|^{N+2s}} \, dy 
		\leq \int_{B_{2R_0}} |x-y|^{-N-2s} \, dy \leq C_{R_0} |x|^{-N-2s}
	\]
and if $|x| \leq 3R_0$, then we see 
	\[
		\begin{aligned}
			\int_{\RN} \frac{\left( \psi \left( \frac{x}{R_0} \right) - \psi \left( \frac{y}{R_0} \right) 
			\right)^2}{|x-y|^{N+2s}} \, dy 
			&\leq 
			\left( \int_{B_{R_0} (x) } + \int_{B_{R_0}^c (x)} \right)
			\frac{\left( \psi \left( \frac{x}{R_0} \right) - \psi \left( \frac{y}{R_0} \right) 
			\right)^2}{|x-y|^{N+2s}} \, dy 
			\\
			&\leq R_0^{-2} \| \psi \|_{C^1(\RN)}^2 \int_{|z| \leq R_0}  |z|^{-N-2s+2} \, dz 
			 + \int_{|z| \geq R_0} |z|^{-N-2s} \, dz.
 		\end{aligned}
	\]
Since 
	\[
		\int_{\RN} \psi \left( \frac{y}{R_0} \right)^2 
		\frac{ \left( \eta \left( \frac{x}{R} \right) - \eta \left( \frac{y}{R} \right) \right)^2  }{|x-y|^{N+2s}} \, dy
		\leq 
		\int_{\RN}
		\frac{ \left( \eta \left( \frac{x}{R} \right) - \eta \left( \frac{y}{R} \right) \right)^2  }{|x-y|^{N+2s}} \, dy
		= R^{-2s} \rho \left( \eta ; \frac{x}{R} \right),
	\]
by \eqref{eq:2.38}, we have \eqref{eq:2.37}. 
	\end{proof}

\begin{lemma}
\label{Lemma:2.6}
Let $u \in \Hsloc (\RN) \cap L^\infty_{\rm loc} (\RN) \cap L^2(\RN,(1+|x|)^{-N-2s}dx) $ 
be a solution of \eqref{eq:1.1} under \eqref{eq:1.2}. 
Assume that $u$ is stable outside $B_{R_0}$. 
Let $\zeta\in C^1(\RN)$ satisfy $\zeta \equiv 0 $ on $B_{R_0}$ and 
	\begin{equation}\label{eq:2.39}
		|x| \left| \nabla \zeta (x) \right| + |\zeta (x)| \leq C\left( 1+|x| \right)^{- m } 
		\quad \text{for all $x \in \RN$}
	\end{equation}
for some $m > N/2$. Then
\begin{equation}
\label{eq:2.40}
\int_{\mathbb R^N}|x|^{\ell}|u|^{p+1}\zeta^2\, dx
+\frac{1}{p}\|u\zeta\|_{\dot H^s(\mathbb R^N)}^2 
\le \frac{C_{N,s}}{p-1}\int_{\mathbb R^N}u (x)^2\rho(\zeta;x)\, dx
\end{equation}
where $C_{N,s}$ is the constant given by \eqref{eq:1.3}. 
\end{lemma}
	\begin{remark}\label{Remark:2.4}
		Later, we shall use Lemma \ref{Lemma:2.6} for $\zeta (x) = \eta_R (x)$ 
		and we require the right-hand side of \eqref{eq:2.40} to be finite. 
		For example, see Lemma \ref{Lemma:2.7} and the end of proof of it. 
		Therefore, we need the condition $u \in L^2(\RN,(1+|x|)^{-N-2s}dx)$ 
		by Remark \ref{Remark:2.3}. 
	\end{remark}

	\begin{proof}
We follow the argument in \cite[Lemma~2.1]{DDW}. 
Remark that the right-hand side in \eqref{eq:2.40} is finite due to \eqref{eq:2.30}, \eqref{eq:2.31}, \eqref{eq:2.39} 
and $u \in L^2(\RN,(1+|x|)^{-N-2s}dx)  $. 
We first treat the case $\zeta \in C^1_c ( \RN)$ where $\zeta \equiv 0$ on $B_{R_0}$. 
Since $u$ is a solution of \eqref{eq:1.1}, by Lemma~\ref{Lemma:2.3}, 
we see that $u\in C^1(\mathbb R^N\setminus\{0\})$.
Since $u\zeta^2\in C^1_c(\mathbb R^N)$, 
by Remark~\ref{Remark:1.1} we can take $\varphi=u\zeta^2$ as a test function in \eqref{eq:1.4},
and by \eqref{eq:1.5} we have
\begin{equation*}
\begin{aligned}
\int_{\mathbb R^N}|x|^\ell|u|^{p+1}\zeta^2\,dx
&= \la u , u \zeta^2 \ra_{\dHs(\RN)} 
\\
&
=\frac{C_{N,s}}{2}
\int_{\mathbb R^N \times \RN} \frac{(u(x)-u(y))(u(x)\zeta(x)^2-u(y)\zeta(y)^2)}{|x-y|^{N+2s}}\,dx\,dy
\\
&
=\frac{C_{N,s}}{2}\int_{\mathbb R^N \times \RN }
\frac{u(x)^2\zeta(x)^2-u(x)u(y)(\zeta(x)^2+\zeta(y)^2)+u(y)^2\zeta(y)^2}{|x-y|^{N+2s}}\,dx\,dy
\\
&
=\frac{C_{N,s}}{2}\int_{\mathbb R^N \times \RN}
\frac{(u(x)\zeta(x)-u(y)\zeta(y))^2-(\zeta(x)-\zeta(y))^2u(x)u(y)}{|x-y|^{N+2s}}\,dx\,dy
\\
&
=\|u\zeta\|^2_{\dot H^s(\mathbb R^N)}-\frac{C_{N,s}}{2}\int_{\mathbb R^N\times \RN}
\frac{(\zeta(x)-\zeta(y))^2u(x)u(y)}{|x-y|^{N+2s}}\,dx\,dy.
\end{aligned}
\end{equation*}
Applying the fundamental inequality $2ab\le a^2+b^2$ with \eqref{eq:2.29},
we deduce that
\begin{equation}
\label{eq:2.41}
\begin{aligned}
&
\|u\zeta\|^2_{\dot H^s(\mathbb R^N)}-\int_{\mathbb R^N}|x|^\ell|u|^{p+1}\zeta^2\,dx
\\
\le \ &
\frac{C_{N,s}}{4} \left(
\int_{\mathbb R^N \times \RN} 
\frac{(\zeta(x)-\zeta(y))^2}{|x-y|^{N+2s}}u(x)^2\,dx\,dy
+\int_{\mathbb R^N \times \RN} 
\frac{(\zeta(x)-\zeta(y))^2}{|x-y|^{N+2s}}u(y)^2\,dx\,dy\right)
\\
= \ & \frac{C_{N,s}}{2}\int_{\mathbb R^N}u(x)^2\rho(\zeta;x)\,dx.
\end{aligned}
\end{equation}
Since $u$ is stable outside $B_{R_0}$, 
by \eqref{eq:1.10} with $\varphi=u\zeta $ (see Remark \ref{Remark:1.3}) and \eqref{eq:2.41}, we have
	\begin{equation}\label{eq:2.42}
		(p-1)\int_{\mathbb R^N}|x|^\ell|u|^{p+1}\zeta^2\,dx 
		\leq
		\| u \zeta \|_{\dHs(\RN)}^2 - \int_{\RN} |x|^\ell |u|^{p+1} \zeta^2 \, dx
		\leq
		\frac{C_{N,s}}{2}\int_{\mathbb R^N}u(x)^2\rho(\zeta;x)\,dx.
	\end{equation}
Hence, by \eqref{eq:2.41} and \eqref{eq:2.42}, 
	\[
		\begin{aligned}
			\frac{1}{p} \| u \zeta \|_{\dHs(\RN)}^2 
			& \leq \frac{1}{p} \int_{  \RN} |x|^\ell |u|^{p+1} \zeta^2 \, dx 
			+ \frac{C_{N,s}}{2p} \int_{  \RN} u(x)^2 \rho (\zeta;x) \, dx 
			\\
			& \leq \frac{C_{N,s}}{2(p-1)} \int_{  \RN} u^2 \rho (\zeta;x) \, dx.
		\end{aligned}
	\]
This together with \eqref{eq:2.42} implies \eqref{eq:2.40} for 
$\zeta \in C^1_{c} (\RN)$ with  $ \zeta \equiv 0$ on $B_{R_0}$.

	Next, let $\zeta \in C^1(\RN)$ satisfy $\zeta \equiv 0$ on $B_{R_0}$ and \eqref{eq:2.39}. 
Let $(\varphi_n)_{n}$ be a sequence of cut-off functions in Lemma \ref{Lemma:2.1} and 
set $\zeta_n := \varphi_n \zeta \in C^1_c(\RN)$. It is easily seen that $(\zeta_n)_n$ satisfies \eqref{eq:2.39} 
uniformly with respect to $n$, namely, the constant $C$ in \eqref{eq:2.39} is independent of $n$. 
Exploiting this fact with \eqref{eq:2.30} and \eqref{eq:2.31}, 
we observe that there exists a $C>0$ such that for all $x \in \RN$ and $n$,
	\[
		\rho(\zeta_n;x) \to \rho(\zeta;x), \quad 
		\left| \rho(\zeta_n;x) \right| \leq C \left( 1 + |x| \right)^{-N-2s}.
	\]
Therefore, from \eqref{eq:2.40} with $\zeta_n$, we find that 
$( \zeta_n u )_n$ is bounded in $\dHs(\RN)$ and it is not difficult to see 
$ |u(x)|^{p+1} \zeta_n (x)^2 \nearrow |u(x)|^{p+1} \zeta(x)^2$ for each $x \in \RN$ and 
$\zeta_n u \rightharpoonup \zeta u$ weakly in $\dHs(\RN)$. 
Hence, from the monotone convergence theorem, the weak lower semicontinuity of norm, 
the fact $u \in L^2(\RN,(1+|x|)^{-N-2s}dx)$, \eqref{eq:2.40} with $\zeta_n$ and 
the dominated convergence theorem, 
it follows that \eqref{eq:2.40} holds for each $\zeta \in C^1(\RN)$ with \eqref{eq:2.39} and 
$\zeta \equiv 0$ on $B_{R_0}$. This completes the proof. 
	\end{proof}

By using Lemmata~\ref{Lemma:2.5} and ~\ref{Lemma:2.6}, we have the following.
\begin{lemma}
\label{Lemma:2.7}
Let $u \in \Hsloc (\RN) \cap L^\infty_{\rm loc} (\RN) \cap L^2(\RN,(1+|x|)^{-N-2s}dx) $ 
be a solution of \eqref{eq:1.1} with \eqref{eq:1.2}, which is stable outside $B_{R_0}$,
and let $\rho_R$ be the function given in Lemma~$\ref{Lemma:2.5}$ with
\begin{equation}
\label{eq:2.43}
m\in\left(\frac{N}{2} , \,  \frac{N+s(p+1) + \ell}{2} \right).
\end{equation}
Then there exists a constant $C=C(N,s,\ell,m,p,R_0)>0$ such that
\begin{equation}
\label{eq:2.44}
\int_{\mathbb R^N}u^2\rho_R\, dx\le C\left( 
\int_{B_{3R_0}} u^2 \rho_R \, dx + 
R^{N-\frac{2}{p-1}(s(p+1)+\ell)}+1\right)
\end{equation}
for all $R\ge 3R_0$.
\end{lemma}
\begin{remark}
\label{Remark:2.5}
Due to \eqref{eq:1.2}, $ 0 < s (p+1) + \ell$ holds and 
we may choose an $m$ satisfying \eqref{eq:2.43}. 
\end{remark}

	\begin{proof}[Proof of Lemma \ref{Lemma:2.7}]
We basically follow the idea in \cite[Lemma 2.4]{DDW} and let $R \geq 3R_0$. 
First, by H\"older's inequality, 
	\begin{equation}\label{eq:2.45}
		\begin{aligned}
			&\int_{\RN} u^2 \rho_R \, dx 
			\\
			=\ & \int_{B_{3R_0}} u^2 \rho_R \, dx 
			+ \int_{B_{3R_0}^c} 
			u^2 \rho_R 
			\left( |x|^{\frac{\ell}{2}} \eta_R \right)^{\frac{4}{p+1}}
			\left( |x|^{\frac{\ell}{2}} \eta_R \right)^{ - \frac{4}{p+1} }  \, dx
			\\
			\leq \ & \int_{B_{3R_0}} u^2 \rho_R \, dx 
			+ \left( \int_{B_{3R_0}^c} |x|^\ell |u|^{p+1} \eta_R^2 \, dx \right)^{\frac{2}{p+1}} 
			\left( \int_{B_{3R_0}^c} |x|^{ -\frac{2\ell}{p-1} } \eta_R^{ -\frac{4}{p-1} } \rho_R^{ \frac{p+1}{p-1} } 
			\, dx \right)^{ \frac{p-1}{p+1} }.
		\end{aligned}
	\end{equation}

For the case $3R_0\le|x|\le R$, 
by Lemma~\ref{Lemma:2.4} with \eqref{eq:2.35}, \eqref{eq:2.36} and Remark \ref{Remark:2.3}, we have
\begin{equation}
\label{eq:2.46}
2^{-\frac{m}{2}}=\eta(1)\le \eta_R(x)\le \eta\bigg(\frac{3R_0}{R}\bigg)\le 1
\qquad\mbox{and}\qquad
\rho\bigg(\eta;\frac{x}{R}\bigg)\le C. 
\end{equation}
Then, by \eqref{eq:2.37} we obtain
\[
\rho_R(x)\le C(|x|^{-N-2s}+R^{-2s}) 
\quad \text{for all $3R_0 \leq |x| \leq R$}.
\]
This together with \eqref{eq:2.46} yields
\begin{equation}
\label{eq:2.47}
\begin{aligned}
&
\int_{B_R\setminus B_{3R_0}}|x|^{-\frac{2\ell}{p-1}}\rho_R^{\frac{p+1}{p-1}}\eta_R^{-\frac{4}{p-1}}\,dx
\\
\le\ & C\int_{3R_0}^Rr^{-\frac{2\ell}{p-1}}\bigg(r^{-(N+2s)}+R^{-2s}\bigg)^{\frac{p+1}{p-1}}r^{N-1}\, dr
\\
\le \ & C\int_{3R_0}^Rr^{N-1-\frac{2\ell}{p-1}-(N+2s)\frac{p+1}{p-1}}\, dr+C R^{-2s\frac{p+1}{p-1}}\int_{3R_0}^Rr^{N-1-\frac{2\ell}{p-1}}\, dr
\\
\le \ & C\int_{3R_0}^\infty r^{-1-\frac{2}{p-1}(N+\ell+s(p+1))}\, dr+C R^{-2s\frac{p+1}{p-1}}\int_{3R_0}^Rr^{N-1-\frac{2\ell}{p-1}}\, dr
\\
\le \ & C+C R^{-2s\frac{p+1}{p-1}}\int_{3R_0}^Rr^{N-1-\frac{2\ell}{p-1}}\, dr.
\end{aligned}
\end{equation}
Since
\[
	\begin{aligned}
		R^{-2s\frac{p+1}{p-1}}\int_{3R_0}^Rr^{N-1-\frac{2\ell}{p-1}}\, dr
		&\leq 
		\left\{\begin{aligned}
			& C  R^{-2s \frac{p+1}{p-1} } & &
			\text{if} \ N - \frac{2\ell}{p-1} < 0,\\
			& C R^{ - 2 s \frac{p+1}{p-1} } \left( 1 + \log R \right) & &
			\text{if} \ N - \frac{2\ell}{p-1} = 0,\\
			& C R^{ N - \frac{2}{p-1} \left( s(p+1) + \ell \right) } & &
			\text{if} \ N - \frac{2\ell}{p-1} > 0,
		\end{aligned}\right.
		\\
		&\le 
		\left\{
		\begin{aligned}
		& C & &\mbox{if}\quad N-\frac{2\ell}{p-1}\le0,
		\\
		&C R^{N-\frac{2}{p-1}(s(p+1)+\ell)} 
		& &\mbox{if}\quad N-\frac{2\ell}{p-1}>0,
		\end{aligned}
		\right.
	\end{aligned}
\]
by \eqref{eq:2.47}, we see that
\begin{equation}
\label{eq:2.48}
\int_{B_R\setminus B_{3R_0}}|x|^{-\frac{2\ell}{p-1}}\rho_R^{\frac{p+1}{p-1}}\eta_R^{-\frac{4}{p-1}}\,dx
\le C+CR^{N-\frac{2}{p-1}(s(p+1)+\ell)}.
\end{equation}

	On the other hand, for the case $|x|\ge R\ge 3R_0$,
by \eqref{eq:2.35} and $R^2+|x|^2 \leq 2|x|^2$, we obtain
\begin{align*}
\eta\bigg(\frac{x}{R}\bigg)^2|x|^{-N-2s}
\le C\eta(1)^2(R^2+|x|^2)^{-\frac{N}{2}-s}
&\le CR^{-N-2s}\bigg(1+\frac{|x|^2}{R^2}\bigg)^{-\frac{N}{2}-s}
\\
&\le CR^{-2s}\bigg(1+\frac{|x|^2}{R^2}\bigg)^{-\frac{N}{2}-s}.
\end{align*}
This together with \eqref{eq:2.30}, \eqref{eq:2.37} and Remark \ref{Remark:2.3} implies that
	\[
		\rho_R(x)\le CR^{-2s}\bigg(1+\frac{|x|^2}{R^2}\bigg)^{-\frac{N}{2}-s} 
		\quad \mbox{for each} \ |x| \geq R.
	\]
Thus, \eqref{eq:2.35}, \eqref{eq:2.36} and the fact $|x| \geq R$ give 
\begin{equation}
\label{eq:2.49}
	\begin{aligned}
		&|x|^{-\frac{2\ell}{p-1}}\rho_R(x)^{\frac{p+1}{p-1}}\eta_R(x)^{-\frac{4}{p-1}}
		\\
		\leq \ & 
		C \left( R^2+|x|^2 \right)^{-\frac{\ell}{p-1}} 
		\left\{  R^{-2s} \left( 1 + \frac{|x|^2}{R^2} \right)^{- \frac{N+2s}{2} } \right\}^{ \frac{p+1}{p-1} } 
		\left( 1 + \frac{|x|^2}{R^2} \right)^{ \frac{m}{2} \frac{4}{p-1}  }
		\\
		\leq \ & CR^{-2s\frac{p+1}{p-1}-\frac{2\ell}{p-1}} 
		\left(1+\frac{|x|^2}{R^2}\right)^{-\frac{N+2s}{2}\frac{p+1}{p-1}+\frac{2m}{p-1} - \frac{\ell}{p-1} }.
	\end{aligned}
\end{equation}
Since it follows from \eqref{eq:2.43} that
\[
\alpha_{ N,s,p,m,\ell } := -\frac{N+2s}{2}\frac{p+1}{p-1}+\frac{2m}{p-1} - \frac{\ell}{p-1} < -\frac{N}{2},
\]
by \eqref{eq:2.49} we have
\begin{equation}
\label{eq:2.50}
\begin{aligned}
&
\int_{B_R^c}|x|^{-\frac{2\ell}{p-1}}\rho_R^{\frac{p+1}{p-1}}\eta_R^{-\frac{4}{p-1}}\,dx
\\
\le \ &
CR^{-2s\frac{p+1}{p-1}-\frac{2\ell}{p-1}}\int_R^\infty
 \left(1+\frac{r^2}{R^2}\right)^{\alpha_{ N,s,p,m,\ell } }r^{N-1}\, dr
\\
= \ & 
C R^{- 2s \frac{p+1}{p-1} - \frac{2\ell}{p-1} + N-1 } 
\int_{R}^{\infty} \left( 1 + \frac{r^2}{R^2} \right)^{ \alpha_{ N,s,p,m,\ell } } 
\left( \frac{r}{R} \right)^{N-1} \, dr
\\
\le \ &
CR^{N-\frac{2}{p-1}(s(p+1)+\ell)}.
\end{aligned}
\end{equation}
Combining \eqref{eq:2.48} and \eqref{eq:2.50},  we obtain
\begin{equation}
\label{eq:2.51}
\int_{B_{3R_0}^c} |x|^{-\frac{2\ell}{p-1}}\rho_R^{\frac{p+1}{p-1}}\eta_R^{-\frac{4}{p-1}}\,dx
\le C\left(R^{N-\frac{2}{p-1}(s(p+1)+\ell)}+1\right).
\end{equation}

	Now we substitute \eqref{eq:2.51} into \eqref{eq:2.45} and 
infer from \eqref{eq:2.40} with $\zeta= \eta_R$ that 
	\[
		\begin{aligned}
			&\int_{\RN} u^2 \rho_R \, dx
			\\
			\leq \ & 
			\int_{B_{3R_0}} u^2 \rho_R \, dx 
			+ C \left( \int_{ B_{3R_0^c} } |x|^\ell |u|^{p+1} \eta_R^2 \, dx \right)^{\frac{2}{p+1}} 
			\left(R^{N-\frac{2}{p-1}(s(p+1)+\ell)}+1\right)^{\frac{p-1}{p+1}} 
			\\
			\leq \ &
			\int_{B_{3R_0}} u^2 \rho_R \, dx 
			+ C \left( \int_{\RN } u^2 \rho_R \, dx \right)^{\frac{2}{p+1}} 
			\left(R^{N-\frac{2}{p-1}(s(p+1)+\ell)}+1\right)^{\frac{p-1}{p+1}} 
		\end{aligned} 
	\]
for $R\ge 3R_0$. Dividing the both sides 
by $ \left( \int_{\RN } u^2 \rho_R \, dx \right)^{\frac{2}{p+1}} < \infty $ 
and noting 
	\[
		\left( \int_{ B_{3R_0} } u^2 \rho_R \, dx \right) 
		\left( \int_{  \RN} u^2 \rho_R \, dx \right)^{ - \frac{2}{p+1} } 
		\leq \left( \int_{ B_{3R_0} } u^2 \rho_R \, dx \right)^{\frac{p-1}{p+1}}, 
	\]
we have \eqref{eq:2.44} and Lemma~\ref{Lemma:2.7} follows.
\end{proof}

For the supercritical case,
we have the following energy estimates for the function $U$, which is given in \eqref{eq:2.3}.
\begin{lemma}
\label{Lemma:2.8}
Assume $p_S(N,\ell) < p $ and \eqref{eq:1.2}. 
Let $u \in \Hsloc (\RN) \cap L^\infty_{\rm loc} (\RN) \cap L^2(\RN, (1+|x|)^{-N-2s}dx) $ 
be a solution of \eqref{eq:1.1} which is stable outside $B_{R_0}$
and $U$ be the function given in \eqref{eq:2.3}.
Then there exists a $C=C(N,p,s,\ell,R_0,u)>0$ such that
\begin{equation}
\label{eq:2.52}
\int_{B^+_R}t^{1-2s}U^2\, dX\le CR^{N+2(1-s)-\frac{2(2s+\ell)}{p-1}}
\end{equation}
for all $R\ge 3R_0$.
\end{lemma}

	\begin{remark}
		If $ p_S(N,\ell) < p $, then 	
			\begin{equation}\label{eq:2.53}
				N - \frac{2}{p-1} \left( s ( p+ 1) + \ell \right) > 0.
			\end{equation}
	\end{remark}

	\begin{proof}[Proof of Lemma \ref{Lemma:2.8}]
Let $\zeta_{R_0} \in C^\infty_c(\RN)$ with $0 \leq \zeta_{R_0} \leq 1$, $\zeta_{R_0} \equiv 1$ on $B_{3R_0}$ and 
$\zeta_{R_0} \equiv 0$ on $B_{4R_0}^c$. 
We decompose $u$ as
	\[
		u(x) = \zeta_{R_0} (x) u(x) + (1 - \zeta_{R_0} (x)) u(x) = : v(x) + w(x).
	\]
Notice that $v \in H^s(\RN)$ with $\supp v \subset \overline{B_{4R_0}}$ and 
$w \in H^s_{\rm loc} (\RN)$. 
Recalling \eqref{eq:2.3}, we also decompose $U$ as 
	\[
		U(x,t) = (P_s(\cdot, t) \ast u) (x) 
		= (P_s (\cdot , t) \ast v ) (x) + (P_s (\cdot , t) \ast w) (x) 
		= : V(x,t) + W(x,t).
	\]

	We first estimate $V(x,t)$. By Young's inequality and $\| P_s (\cdot , t) \|_{L^1(\RN)} = 1$, 
it follows that 
	\[
		\left\| V(\cdot, t) \right\|_{L^2(\RN)} \leq \| P_s (\cdot , t) \|_{L^1(\RN)} \| v \|_{L^2(\RN)} 
		= \| v \|_{L^2(\RN)} \quad \text{for each $t \in (0,\infty)$}.
	\]
Therefore, 
	\[
		\begin{aligned}
		\int_{B_R^+} t^{1-2s} |V|^2 \, d X 
		& \leq \int_{0}^{R} dt \int_{\RN} t^{1-2s} |V(x,t)|^2 \, dx 
		\leq \int_0^R t^{1-2s} \| v \|_{L^2(\RN)}^2 dt 
		= C R^{2-2s}.
		\end{aligned}
	\]
From 
	\[
		2-2s \leq N + 2 (1-s) - \frac{2(2s+\ell)}{p-1},
	\]
we infer that 
	\begin{equation}\label{eq:2.54}
		\int_{B^+_R} t^{1-2s} |V|^2 \, dX \leq C R^{N + 2 (1-s) - \frac{2(2s+\ell)}{p-1}}. 
	\end{equation}

	Next, we consider $W(x,t)$. By H\"older's inequality, 
	\begin{equation}\label{eq:2.55}
		\begin{aligned}
			& \int_{B_R^+} t^{1-2s} |W|^2 \, dX 
			\\
			= \ &
			\int_{ B_R^+} t^{1-2s} \left( \int_{\RN} (P_s(x-y,t))^{1/2} (P_s(x-y,t))^{1/2} w(y) \, dy  \right)^2 \, dX 
			\\
			\leq \ & C \int_{ B_R^+} dX  t^{1-2s} \int_{\mathbb R^N}w(y)^2\frac{t^{2s}}{(|x-y|^2+t^2)^{\frac{N+2s}{2}}}\, dy
			\\
			\leq \ & C \int_{ B_R^+} dX t^{1-2s} 
			\left( \int_{|x-y| \leq 3R} + \int_{ |x-y|> 3R} \right) 
			 w(y)^2\frac{t^{2s}}{(|x-y|^2+t^2)^{\frac{N+2s}{2}}}\, dy.
		\end{aligned}
	\end{equation}
For $y \in \RN$ with $|x-y| \leq 3R$, since it follows from \eqref{eq:1.2} that 
	\[
		- \frac{N+2s}{2} + 1 < 0 \quad \text{if} \ N \geq 2, \quad 
		- \frac{N+2s}{2} + 1 = \frac{1-2s}{2} > 0 \quad \text{if} \ N = 1, 
	\]
we see that 
	\[
		\begin{aligned}
			&\int_{ B_R^+} dX t^{1-2s} \int_{ |x-y| \leq 3R} 
			w(y)^2 \frac{t^{2s}}{(|x-y|^2+t^2)^{\frac{N+2s}{2}}}\, dy
			\\
			\leq & \ 
			\int_0^R dt \int_{B_R} dx 
			\int_{|x-y| \leq 3 R}
			w(y)^2\frac{t}{(|x-y|^2+t^2)^{\frac{N+2s}{2}}}\, dy
			\\
			= & \ 
			\int_{ B_R} dx \int_{ |x-y| \leq 3R} dy 
			\int_{0}^R w(y)^2 
			\frac{1}{2-N-2s} \frac{\partial}{\partial t} 
			\left( |x-y|^2 + t^2 \right)^{ \frac{2 - (N+2s)}{2} }
			\, dt 
			\\
			= \ & \frac{1}{2 - N - 2s} \int_{ B_R} dx \int_{ |x-y| \leq 3R} 
			w(y)^2 
			\left\{ \left( |x-y|^2 + R^2 \right)^{ \frac{2-N-2s}{2}   } - |x-y|^{ 2 - N -2s } \right\} \, dy
			\\
			\leq \ & 
			\left\{\begin{aligned}
				& \frac{1}{1 - 2s } \int_{ B_R} dx \int_{ |x-y| \leq 3R} 
				w(y)^2 \left( |x-y|^2 + R^2 \right)^{ \frac{1-2s}{2} }  \, dy 
				& & \text{when $N = 1$},
				\\
				& \frac{1}{N+2s -2 }\int_{ B_R} dx \int_{ |x-y| \leq 3R} 
				w(y)^2 |x-y|^{ 2-N-2s } \, dy 
				& & \text{when $N \geq 2$}.
			\end{aligned}\right.
		\end{aligned}
	\]
When $N = 1$, by $\set{y \in \RN | |x-y| \leq 3R} \subset B_{4R}$ for each $x \in B_R$, we have 
	\[
		\begin{aligned}
			 \int_{ B_R} dx \int_{ |x-y| \leq 3R} 
			w(y)^2 \left( |x-y|^2 + R^2 \right)^{ \frac{1-2s}{2} }  \, dy 
			& \leq C R^{1-2s} \int_{ B_R} dx \int_{ |x-y| \leq 3R} w(y)^2 \, dy
			\\
			& \leq C R^{2-2s} \int_{  B_{4R}} w(y)^2 \, dy.
		\end{aligned}
	\]
On the other hand, when $N \geq 2$, since 
$B_R(y) \subset B_{5R}$ for each $y \in B_{4R}$, we have 
	\[
		\begin{aligned}
			\int_{ B_R} dx \int_{ |x-y| \leq 3R} 
			w(y)^2 |x-y|^{ 2-N-2s } \, dy 
			& \leq C \int_{ B_R} dx \int_{  B_{4R}} w(y)^2 |x-y|^{2-N-2s } \, dy
			\\
			& = C \int_{  B_{4R}} dy \int_{ B_R (y) } w(y)^2 |z|^{2-N-2s} \, dz
			\\
			& \leq C \int_{  B_{4R}} dy \int_{  B_{5R}} w(y)^2 |z|^{2-N-2s} \, d z
			\\
			& = C R^{2 - 2s} \int_{  B_{4R}} w(y)^2 \,dy,
		\end{aligned}
	\]
hence, for $N \geq 1$ and $R \geq 3R_0$, 
	\[
		\int_{ B_R^+} dX t^{1-2s} \int_{ |x-y| \leq 3R} 
		w(y)^2 \frac{t^{2s}}{(|x-y|^2+t^2)^{\frac{N+2s}{2}}}\, dy 
		\leq C R^{2-2s} \int_{  B_{4R}} w(y)^2 \, dy.
	\]
Notice that $w \equiv 0$ on $B_{3R_0}$, $|w| \leq |u|$ 
and $0<c \leq \eta_R(x)$ for any $3R_0 \leq |x| \leq 4R$. 
By Lemma \ref{Lemma:2.6} and $N - 2\ell / (p-1) > 0$ due to $p > p_S(N,\ell)$,  
we may argue as in \eqref{eq:2.45} and \eqref{eq:2.47} to obtain 
	\[
		\begin{aligned}
			& \int_{ B_R^+} dX t^{1-2s} \int_{ |x-y| \leq 3R} 
			w(y)^2 \frac{t^{2s}}{(|x-y|^2+t^2)^{\frac{N+2s}{2}}}\, dy 
			\\
			\leq \ & C R^{2-2s} \int_{ 3R_0 \leq |x| \leq 4R} 
			u^2 \left( |x|^{\frac{\ell}{2}} \eta_R \right)^{\frac{4}{p+1}} 
			\left( |x|^{\frac{\ell}{2}} \eta_R \right)^{-\frac{4}{p+1}} \, dx
			\\
			\leq \ & CR^{2-2s} 
			\left( \int_{\RN} |x|^\ell |u|^{p+1} \eta_R^2 \, dx  \right)^{\frac{2}{p+1}} 
			\left( \int_{3R_0 \leq |x| \leq 4R} |x|^{-\frac{2\ell}{p-1}} \,dx
			 \right)^{\frac{p-1}{p+1}}
			 \\
			 \leq \ & C R^{ 2(1-s) + (N-\frac{2\ell}{p-1})\frac{p-1}{p+1} } 
			 \left( \int_{  \RN} u^2 \rho_R \, dx \right)^{\frac{2}{p+1}}.
		\end{aligned}
	\]
Furthermore, by \eqref{eq:2.53} and Lemma~\ref{Lemma:2.7}, enlarging $C$ if necessary, we obtain
\begin{equation}\label{eq:2.56}
\int_{\mathbb R^N}u^2\rho_R\,dx\le CR^{N-\frac{2}{p-1}(s(p+1)+\ell)},
\end{equation}
which yields 
\begin{equation}
\label{eq:2.57}
\int_{ B_R^+ } dX t^{1-2s} \int_{ |x-y| \leq 3R} 
w(y)^2 \frac{t^{2s}}{(|x-y|^2+t^2)^{\frac{N+2s}{2}}}\, dy 
\le CR^{N+2(1-s)-\frac{2(2s+\ell)}{p-1}}. 
\end{equation}

	Next, we consider the second term in \eqref{eq:2.55}, namely the case $|x-y| > 3R$. 
Since $|x-y|\ge |y|-|x| \ge |y|-R\ge |y|/2$ and $B_{3R}^c(x) \subset B_{2R}^c$ 
for $x\in B_R$ and $y\in B^c_{2R}$, we have
	\[
		\begin{aligned}
			&\int_{ B_R^+} dX t^{1-2s} \int_{ |x-y| > 3R} 
			w(y)^2 \frac{t^{2s}}{(|x-y|^2+t^2)^{\frac{N+2s}{2}}}\, dy
			\\
			\leq \ & \int_{ B_R} dx \int_{ |x-y| > 3R} dy 
			\int_0^R w(y)^2  t |x-y|^{-N-2s} \, dt
			\\
			\leq \ & R^2 \int_{ B_R} dx \int_{ B_{3R}^c(x)} 
			w(y)^2 |x-y|^{-N-2s} \, dy
			\\
			\leq \  & C R^2 \int_{ B_R} dx \int_{ B_{3R}^c(x)} 
			w(y)^2 |y|^{-N-2s} \, dy
			\\
			\leq \ & C R^{2+N} \int_{ |z| \geq 2R} w(z)^2 |z|^{-N-2s} \, dy.
		\end{aligned}
	\]
On the other hand, from the definition of $\eta_R$ and $\rho_R$, it follows that 
for $|x| \geq 2 R \geq 6 R_0$ and $\boldsymbol{e}_1 = (1,0,\ldots,0) \in \RN$, 
	\begin{equation}\label{eq:2.58}
		\begin{aligned}
			\rho_R(x) = \int_{  \RN} \frac{ ( \eta_R(x) - \eta_R (y) )^2 }{|x-y|^{N+2s}} \, dy 
			&\geq \int_{ |y| \geq 2 R_0 } 
			\frac{ \left( \eta\left( \frac{x}{R} \right) - \eta \left(\frac{y}{R}\right) \right)^2 }{|x-y|^{N+2s}} \, dy
			\\
			&= R^{-2s} \int_{ |z| \geq 2 R^{-1} R_0 } 
			\frac{ \left( \eta\left(\frac{x}{R}\right) - \eta(z) \right)^2 }{|R^{-1} x - z|^{N+2s}  } \, dz
			\\
			& \geq R^{-2s} \int_{ |z- \boldsymbol{e}_1 | < \frac{1}{3}}
			\frac{ \left( \eta\left(\frac{x}{R}\right) - \eta(z) \right)^2 }{|R^{-1} x - z|^{N+2s}  } \, dz
			\\
			&\geq C_0 R^{-2s} \left| R^{-1} x \right|^{-N-2s} 
			= C_0 R^{N} |x|^{-N-2s}
		\end{aligned}
	\end{equation}
for some $C_0>0$. Thus, noting $u \equiv w$ on $|y| \geq 2R$ and \eqref{eq:2.56}, we obtain 
	\[
		\begin{aligned}
			 R^{2+N} \int_{ |y| \geq 2R} w(y)^2 |y|^{-N-2s} \, dy
			 & \leq C R^2 \int_{ |y| \geq 2R} u^2 \rho_R \, dy 
			 \leq  CR^{N+ 2 (1-s) - \frac{2(2s+\ell)}{p-1} },
		\end{aligned}
	\]
which implies 
	\begin{equation}\label{eq:2.59}
		\int_{ B_R^+} dX t^{1-2s} \int_{ |x-y| > 3R} 
		w(y)^2 \frac{t^{2s}}{(|x-y|^2+t^2)^{\frac{N+2s}{2}}}\, dy
		\leq CR^{N+ 2 (1-s) - \frac{2(2s+\ell)}{p-1} }.
	\end{equation}
Substituting \eqref{eq:2.57} and \eqref{eq:2.59} into \eqref{eq:2.55}, we see that 
	\[
		\int_{ B_R^+} t^{1-2s} W^2 \, dX \leq CR^{N+ 2 (1-s) - \frac{2(2s+\ell)}{p-1} }.
	\]
This with \eqref{eq:2.54} and $U=V+W$ completes the proof of Lemma \ref{Lemma:2.8}. 
	\end{proof}

\begin{lemma}
\label{Lemma:2.9}
Assume the same conditions as in Lemma~$\ref{Lemma:2.8}$. 
Then there exists a $C=C(N,p,s,\ell,R_0,u)>0$ such that
\begin{equation}
\label{eq:2.60}
\int_{B^+_R}t^{1-2s}|\nabla U|^2\, dX
+\int_{B_R}|x|^\ell|u|^{p+1}\,dx
\le CR^{N-\frac{2}{p-1}(s(p+1)+\ell)}
\end{equation}
for all $R\ge 3R_0$.
\end{lemma}

	\begin{proof}
We first prove the weighted $L^{p+1}$ estimate for $u$. 
Let $\eta_R$ and $\rho_R$ be the functions given in Lemma~\ref{Lemma:2.5}.
Then, since $u\in L^\infty_{\rm loc}(\mathbb R^N)$ with \eqref{eq:1.2}, 
it holds that
\begin{equation}
\label{eq:2.61}
\int_{B_{2R_0}}|x|^\ell|u|^{p+1}\, dx\le C.
\end{equation}
Furthermore, since $p_S(N,\ell) < p$ and \eqref{eq:2.53}, 
applying Lemmata~\ref{Lemma:2.6} and \ref{Lemma:2.7},  
and noting $\eta_R \geq c_0>0$ on $B_R \setminus B_{2R_0}$, 
we see that for all $R \geq 3R_0$, 
\begin{align*}
\int_{B_R\setminus B_{2R_0}}|x|^\ell|u|^{p+1}\, dx
& \le 
C\int_{\mathbb R^N}|x|^\ell|u|^{p+1}\eta_R^2 \,dx
\\
&
\le C\int_{\mathbb R^N}u^2\rho_R\,dx
\le C\bigg(1+R^{N-\frac{2}{p-1}(s(p+1)+\ell)}\bigg)
\le CR^{N-\frac{2}{p-1}(s(p+1)+\ell)}.
\end{align*}
This together with \eqref{eq:2.61} yields
\begin{equation}
\label{eq:2.62}
\int_{B_R}|x|^\ell|u|^{p+1}\,dx
\le CR^{N-\frac{2}{p-1}(s(p+1)+\ell)} 
\quad \text{for each $R \geq 3R_0$}. 
\end{equation}

	Next we take a cut-off function 
$\zeta\in C^\infty_c (\overline{\mathbb R^{N+1}_+} \setminus \overline{B^+_{R_0}} )$ such that 
	\begin{equation}\label{eq:2.63}
		\zeta\equiv
		\left\{
		\begin{array}{ll}
		1& \mbox{on}\quad B^+_R\setminus B^+_{2R_0},
		\vspace{5pt}\\
		0& \mbox{on}\quad B^+_{R_0}\cup( \mathbb R^{N+1}_+ \setminus B^+_{2R}),
		\end{array}
		\right.
		\qquad
		|\nabla \zeta|\le CR^{-1}
		\quad\mbox{on}\quad B^+_{2R}\setminus B^+_{R}.
	\end{equation}
Then, taking $\psi = U\zeta^2 \in C^1_c ( \overline{ \HS} \setminus \overline{ B^+_{R_0}} )$ 
as a test function in \eqref{eq:2.18}, 
we obtain 
\begin{equation}
\label{eq:2.64}
\begin{aligned}
\kappa_s\int_{ \mathbb R^N }|x|^\ell|u|^{p+1}\zeta(x,0)^2\,dx
&
=\int_{\mathbb R^{N+1}_+}t^{1-2s}\nabla U\cdot\nabla(U\zeta^2)\,dX
\\
&
=\int_{\mathbb R^{N+1}_+}t^{1-2s}\{|\nabla(U\zeta)|^2-U^2|\nabla\zeta|^2\}\,dX.
\end{aligned}
\end{equation}
Since $u$ is stable outside $B_{R_0}$ and $U=u$ on $\partial\mathbb R^{N+1}_+$,
we see from \eqref{eq:2.8} that $U$ is stable outside $B^+_{R_0}$, that is, 
for any $\psi \in C^1_c( \overline{ \HS } \setminus \overline{B^+_{R_0}} )$,
\begin{equation}
\label{eq:2.65}
\begin{aligned}
p\kappa_s\int_{\RN}|x|^\ell|U(x,0)|^{p-1}\psi(x,0)^2\,dx
&=
p\kappa_s\int_{\RN}|x|^\ell|u|^{p-1}\psi(x,0)^2\,dx
\\
& \leq \kappa_s \| \psi (\cdot, 0) \|_{\dHs(\RN)}^2 
\le\int_{\mathbb R^{N+1}_+}t^{1-2s}|\nabla \psi|^2\,dX. 
\end{aligned}
\end{equation}
By \eqref{eq:2.64} and \eqref{eq:2.65} with 
$\psi=U\zeta \in C^1_c( \overline{ \HS} \setminus \overline{B^+_{R_0}} )$, we have 
\[
\int_{\mathbb R^{N+1}_+}t^{1-2s}\{|\nabla(U\zeta)|^2-U^2|\nabla\zeta|^2\}\,dX
\leq 
\frac{1}{p}\int_{\mathbb R^{N+1}_+}t^{1-2s}|\nabla(U\zeta)|^2\, dX,
\]
which implies 
\begin{equation}
\label{eq:2.66}
\int_{\mathbb R^{N+1}_+}t^{1-2s}|\nabla(U\zeta)|^2\,dX
\le
\frac{p}{p-1}\int_{\mathbb R^{N+1}_+}t^{1-2s}U^2|\nabla\zeta|^2\,dX.
\end{equation}
By \eqref{eq:2.66}, \eqref{eq:2.63} and \eqref{eq:2.52}, we see that
\begin{equation}
\label{eq:2.67}
\begin{aligned}
\int_{B^+_R\setminus B^+_{2R_0}}t^{1-2s}|\nabla U|^2\,dX
&\le 
\int_{\mathbb R^{N+1}_+}t^{1-2s}|\nabla(U\zeta)|^2\,dX
\\
&
\le C\int_{\mathbb R^{N+1}_+}t^{1-2s}U^2|\nabla\zeta|^2\,dX
\\
&
\le C \left( \int_{B^+_{2R_0} \setminus B_{R_0}^+ } t^{1-2s} U^2 \, dX 
+  R^{-2}\int_{ B^+_{2R} \setminus B_R^+ }t^{1-2s}U^2\,dX \right)
\\
&
\le 
C \int_{B^+_{2R_0} \setminus B_{R_0}^+ } t^{1-2s} U^2 \, dX 
+ CR^{N-\frac{2}{p-1}(s(p+1)+\ell)}
\end{aligned}
\end{equation}
for all $R\ge3R_0$.

	On the other hand,
it follows from $U\in \Hloc ( \overline{\HS},t^{1-2s}dX)$ due to Lemma \ref{Lemma:2.1} that 
\[
\int_{B^+_{2R_0}}t^{1-2s} \left( |\nabla U|^2 + U^2 \right)\,dX\le C.
\]
This together with \eqref{eq:2.62} and \eqref{eq:2.67} yields \eqref{eq:2.60},
thus Lemma~\ref{Lemma:2.9} follows.
	\end{proof}
%

\section{The subcritical and critical case}
\label{section:3}
%
In this section,
we prove Theorem~\ref{Theorem:1.1} for the subcritical and critical case, 
that is, $1 < p \leq p_S(N,\ell)$.

	\begin{proof}[Proof of Theorem~\ref{Theorem:1.1} for $1<p\le p_S(N,\ell)$]
Let $u \in \Hsloc (\RN) \cap L^\infty_{\rm loc} (\RN) \cap L^2( \RN , (1+|x|)^{-N-2s} dx ) $ 
be a solution of \eqref{eq:1.1} which is stable outside $B_{R_0}$. 
As $R \to \infty$, $\eta_R(x) \to \psi(R_0^{-1} x)$ for each $x \in \RN$. 
Then by Lemma \ref{Lemma:2.5}, $(\rho_R)_{R\ge R_0}$ is bounded in $L^\infty(\RN)$ and 
we may check 
	\[
		\rho_R(x) = \int_{  \RN} \frac{ \left( \eta_R(x) - \eta_R(y) \right)^2 }{|x-y|^{N+2s}} \, dy 
		\to \int_{  \RN} \frac{ \left( \psi(R_0^{-1} x) - \psi ( R_0^{-1} y )  \right)^2 }{|x-y|^{N+2s}} \, dy 
		=: \rho_\infty (x).
	\]

	Next, from Lemmata \ref{Lemma:2.6} and \ref{Lemma:2.7} and 
the assumption $1 < p \leq p_S(N,\ell)$, it follows that 
$( u \eta_R^{2/(p+1)} )_{R \geq 3R_0}$ is bounded in $L^{p+1} (\RN,|x|^\ell dx)$ and 
$(u \eta_R)_{R \geq R_0}$ is bounded in $\dHs(\RN)$. 
Since 
	\[
		u(x) \eta_R(x)^{ \frac{2}{p+1} } \to u(x) \psi \left( R_0^{-1} x \right)^{\frac{2}{p+1}}, \quad 
		u(x) \eta_R(x) \to u(x) \psi \left( R_0^{-1} x \right) 
		\quad \text{for each $x \in \RN$}, 
	\]
we infer that 
	\[
		\begin{aligned}
			&u \eta_R^{\frac{2}{p+1}} \rightharpoonup u \psi \left( R_0^{-1} \cdot \right)^{\frac{2}{p+1}} 
			& &\text{weakly in} \ L^{p+1} (\RN, |x|^\ell dx),
			\\
			& u \eta_R \rightharpoonup u \psi \left( R_0^{-1} \cdot \right) 
			& &\text{weakly in} \ \dHs(\RN).
		\end{aligned}
	\]
In particular, we deduce that $u \in \dHs(\RN) \cap L^{p+1}(\RN,|x|^\ell dx)$.

	Since $\varphi_n u \to u$ strongly in $\dHs(\RN)$ where $(\varphi_n)_n$ appears in Lemma \ref{Lemma:2.1} 
and $u \in L^\infty_{\rm loc}(\RN)$, 
we may use $\varphi_n u$ as a test function in \eqref{eq:1.4}: 
	\[
		\int_{  \RN} |x|^\ell |u|^{p+1} \varphi_n \, dx 
		= \la u , \varphi_n u \ra_{\dHs(\RN)}.
	\]
Letting $n \to \infty$, we obtain 
\begin{equation}
\label{eq:3.1}
\int_{\mathbb R^N}|x|^\ell|u|^{p+1}\, dx=\|u\|_{\dot H^s(\mathbb R^N)}^2.
\end{equation}
Thus, the former assertion of Theorem \ref{Theorem:1.1}(ii) is proved.

	For the latter assertion of Theorem \ref{Theorem:1.1} (ii), 
assume that $p = p_S(N,\ell)$ and $u$ is stable. 
By the same argument as above, 
we can apply the stability inequality \eqref{eq:1.9} with the test function $\varphi=u$:
\[
p\int_{\mathbb R^N}|x|^\ell|u|^{p+1}\, dx\le\|u\|_{\dot H^s(\mathbb R^N)}^2.
\]
This contradicts \eqref{eq:3.1} unless $u\equiv0$.
So it remains to prove the subcritical case.

	Since $u\in\dot H^s(\mathbb R^N)  $, notice that 
$\nabla U\in L^2(\mathbb R^{N+1}_+,t^{1-2s}dX)$ thanks to Remark \ref{Remark:2.1}. 
Then, similarly to \eqref{eq:2.25} with $u \in L^{p+1}(\RN, |x|^\ell dx)$, 
we claim that there exists a sequence $R_n\to\infty$ such that
\begin{equation}
\label{eq:3.2}
\lim_{n\to\infty}R_n \left[
\int_{S^+_{R_n}}t^{1-2s}|\nabla U|^2\,dS
+\int_{S_{R_n}}|x|^\ell|u|^{p+1}\, d\omega
+\int_{S^+_{R_n}}t^{1-2s}\left|\frac{\partial U}{\partial \nu}\right|^2\,dS\right]=0.
\end{equation}
By \eqref{eq:2.21}, \eqref{eq:3.2} and replacing $R$ with $R_n$ for a sequence $R_n\to\infty$, 
we conclude that
\[
\int_{\mathbb R^{N+1}_+}t^{1-2s}|\nabla U|^2\,dX
=\frac{2\kappa_s}{N-2s}\frac{N+\ell}{p+1}\int_{\RN}|x|^\ell|u|^{p+1}\,dx.
\]
This together with \eqref{eq:2.7} yields the following Pohozaev identity
\[
\frac{N+\ell}{p+1}\int_{\mathbb R^N}|x|^\ell|u|^{p+1}\, dx=\frac{N-2s}{2}\|u\|_{\dot H^s(\mathbb R^N)}^2.
\]
Combining this identity with \eqref{eq:3.1}, we observe that $u\equiv0$ for $p<p_S(N,\ell)$,
and the proof of Theorem~\ref{Theorem:1.1} is completed for $ 1 < p \leq p_S(N,\ell)$. 
	\end{proof}

%
\section{The supercritical case}
\label{section:4}
%

In this section, 
we follow the argument in \cite{DDW} basically. 
However, due to the regularity issue of $U$ 
around $0$ in $\overline{ \HS }$, 
we prove the monotonicity formula (Lemma \ref{Lemma:4.2}) 
via the argument in \cite[section 3]{FF} and prove Theorem \ref{Theorem:1.1} for $p_S(N,\ell) < p$.

For $X\in \mathbb R^{N+1}_+$, we use the following notation:
	\[
		r := |X|, \quad \sigma := \frac{X}{|X|} \in S^+_1, \quad 
		\sigma_{N+1} := \frac{t}{|X|}.
	\]
Let $u \in \Hsloc (\RN) \cap L^\infty_{\rm loc} (\RN) \cap L^1(\RN, (1+|x|)^{-N-2s}dx) $ 
be a solution of \eqref{eq:1.1} and $U$ be the function given in \eqref{eq:2.3}.
For every $\lambda>0$, we define
\begin{equation}
\label{eq:4.1}
D(U;\lambda):=\lambda^{-(N-2s)}\left[\frac{1}{2}\int_{B^+_\lambda}t^{1-2s}|\nabla U|^2\,dX
-\frac{\kappa_s}{p+1}\int_{B_\lambda}|x|^\ell|u|^{p+1}\,dx\right]
\end{equation}
and
\begin{equation}
\label{eq:4.2}
H(U;\lambda):=\lambda^{-(N+1-2s)}\int_{S^+_\lambda}t^{1-2s}U^2\,dS
=\int_{S^+_1} \sigma_{N+1}^{1-2s}U(\lambda \sigma )^2\,dS. 
\end{equation}

\begin{lemma}
\label{Lemma:4.1}
As a function of $\lambda$, $D,H \in C^1((0,\infty))$ and 
	\[
		\partial_{\lambda} D(U;\lambda)
		=\lambda^{-(N-2s)}\int_{S^+_\lambda}t^{1-2s} \left|\frac{\partial U}{\partial\nu} \right|^2\,dS
		-\lambda^{-(N+1-2s)}\kappa_s\frac{2s+\ell}{p+1}\int_{B_\lambda}|x|^\ell|u|^{p+1}\,dx
	\]
and 
	\[
		\begin{aligned}
			\partial_{\lambda} H (U;\lambda)
			&=2\lambda^{-(N+1-2s)}\int_{S^+_\lambda}t^{1-2s}U\frac{\partial U}{\partial \nu}\,dS 
			\\
			&= 2\lambda^{-(N+1-2s)} \left[ \int_{B^+_\lambda}t^{1-2s}|\nabla U|^2\,dX
			-\kappa_s\int_{B_\lambda}|x|^\ell|u|^{p+1}\,dx \right]. 
		\end{aligned}
	\]
\end{lemma}

	\begin{proof}
We first remark that by \eqref{eq:2.3} and Lemma \ref{Lemma:2.3}, 
$U \in C( \overline{ \HS} )$, 
$\nabla_x U \in C( \overline{ \HS } \setminus \{0\} )$ 
and $V := t^{1-2s} \partial_t U \in C ( \overline{ \HS } \setminus \{0\} )$. 
Hence, as a function of $\lambda$, 
	\[
		 \lambda \mapsto \int_{ B_\lambda} |x|^\ell |u|^{p+1} \, dx \in C^1( (0,\infty) ).
	\]
On the other hand, since 
	\[
		\begin{aligned}
			t^{1-2s} \left| \nabla U \right|^2 
			= t^{1-2s} \left| \nabla_x U \right|^2 + t^{2s-1} V^2
		\end{aligned}
	\]
and $t^{1-2s} \in L^1_{\rm loc} ( \overline{ \HS} )$, 
it is easy to see that 
	\[
		\lambda \mapsto \int_{ S_\lambda^+} t^{1-2s} \left| \nabla U \right|^2 \, dS 
		\in C((0,\infty)), \quad 
		\lambda \mapsto \int_{ B_\lambda^+ } t^{1-2s} \left| \nabla U \right|^2 \, dX 
		\in C^1( (0,\infty) ),
	\]
which yields $D(U;\lambda) \in C^1((0,\infty))$.

	On the other hand, for any $0< \lambda_1 < \lambda_2 < \infty$, 
there exists a $C_{\lambda_1,\lambda_2}>0$ such that 
for every $\lambda \in [\lambda_1,\lambda_2]$ and $\sigma \in S^+_1$
	\[
		\begin{aligned}
			\sigma_{N+1}^{1-2s}\left| \partial_{\lambda} ( U(\lambda \sigma ) )^2 \right| 
			\leq 2 \sigma_{N+1}^{1-2s} \left| U(\lambda \sigma) \right| \left| \nabla U(\lambda \sigma ) \right| 
			&\leq 2 \left| U(\lambda \sigma ) \right| 
			\left( \sigma_{N+1}^{1-2s} 
			\left| \nabla_x U(\lambda \sigma ) \right| + \left| V (\lambda \sigma) \right| \right)
			\\
			& \leq C_{\lambda_1,\lambda_2} (1 + \sigma_{N+1}^{1-2s}).
		\end{aligned}
	\]
Hence, the dominated convergence theorem gives 
$H (U;\lambda) \in C^1((0,\infty))$.

	Next we compute the derivative of $D$ and $H$. 
Direct computations and \eqref{eq:2.21} give 
	\[
		\begin{aligned}
			&\partial_{\lambda} D(U;\lambda) 
			\\
			= \ & - (N-2s) \lambda^{-(N-2s)-1} \left[ 
			\frac{1}{2} \int_{ B_\lambda^+ } t^{1-2s} | \nabla U |^2 \, dX 
			- \frac{\kappa_s}{p+1} \int_{ B_\lambda} |x|^\ell |u|^{p+1} \, dx 
			\right]
			\\
			& \quad + \lambda^{-(N-2s)-1} \lambda  
			\left[ \frac{1}{2} \int_{ S_\lambda^+} t^{1-2s} |\nabla U|^2 \, dS 
			- \frac{\kappa_s}{p+1} \int_{ S_\lambda} |x|^\ell |u|^{p+1} \, d\omega \right]
			\\
			=\ &  - (N-2s) \lambda^{-(N-2s)-1} \left[ 
			\frac{1}{2} \int_{ B_\lambda^+ } t^{1-2s} | \nabla U |^2 \, dX 
			- \frac{\kappa_s}{p+1} \int_{ B_\lambda} |x|^\ell |u|^{p+1} \, dx 
			\right]
			\\
			& \quad + \lambda^{-(N-2s)-1} 
			\left[ \frac{N-2s}{2} \int_{ B_\lambda^+} t^{1-2s} |\nabla U|^2 \, dX 
			- \kappa_s \frac{N+\ell}{p+1} \int_{ B_\lambda} |x|^\ell |u|^{p+1} \, dx 
			\right.
			\\
			& \hspace{3cm} \left.
			+ \lambda \int_{ S_\lambda^+} t^{1-2s} \left| \frac{\partial U}{\partial \nu} \right|^2 \, dS \right]
			\\
			= \ & \lambda^{-(N-2s)} \int_{ S_\lambda^+} t^{1-2s} \left| \frac{\partial U}{\partial \nu} \right|^2 \, dS 
			- \lambda^{-(N+1-2s)}\kappa_s\frac{2s+\ell}{p+1}\int_{B_\lambda}|x|^\ell|u|^{p+1}\,dx.
		\end{aligned}
	\]
For $H$, we compute similarly by using \eqref{eq:2.22} and 
$\nabla U(X) \cdot (X/|X|) = \partial U / \partial \nu$: 
	\[
		\begin{aligned}
			\partial_{\lambda} H(U;\lambda) 
			&= \int_{ S^+_1} 
			\sigma_{N+1}^{1-2s} 2 U(\lambda \sigma)   \nabla U  (\lambda \sigma) \cdot \sigma \, d S 
			\\
			&= 2 \lambda^{-N-1+2s} \int_{ S_\lambda^+} 
			t^{1-2s} U \frac{\partial U}{\partial \nu} \, dS 
			\\
			&= 2 \lambda^{-N-1+2s} 
			\left[ \int_{ B_\lambda^+ } t^{1-2s} |\nabla U|^2 \, dX 
			- \kappa_s \int_{ B_\lambda} |x|^\ell |u|^{p+1} \, dx \right].
		\end{aligned}
	\]
Hence, we complete the proof. 
	\end{proof}

Applying Lemma~\ref{Lemma:4.1}, we prove the following monotonicity formula 
(cf. \cite[Theorem 1.4]{DDW}).
\begin{lemma}
\label{Lemma:4.2}
For $\lambda>0$, define $E(U;\lambda)$ by 
\begin{equation}
\label{eq:4.3}
E(U;\lambda):=\lambda^{\frac{2(2s+\ell)}{p-1}} \left(D(U;\lambda)+\frac{2s+\ell}{2(p-1)}H(U;\lambda) \right).
\end{equation}
Then it holds that
\begin{equation}
\label{eq:4.4}
\partial_{\lambda} E(U;\lambda)
=\lambda^{\frac{2}{p-1}(s(p+1)+\ell)-N} 
\int_{S^+_\lambda} t^{1-2s} \left|\frac{2s+\ell}{p-1}\frac{U}{\lambda}+\frac{\partial U}{\partial r}\right|^2\,dS.
\end{equation}
\end{lemma}

	\begin{proof}
Put
\begin{equation}
\label{eq:4.5}
\gamma:=\frac{2(2s+\ell)}{p-1}.
\end{equation}
By \eqref{eq:4.3} and \eqref{eq:4.5}, we have 
\begin{equation}
\label{eq:4.6}
\begin{aligned}
\partial_{\lambda} E(U;\lambda)
&
=\gamma \lambda^{\gamma-1}\left(D(U;\lambda)+\frac{\gamma}{4}H(U;\lambda)\right)
+\lambda^{\gamma}\left( \partial_{\lambda} D(U;\lambda)+\frac{\gamma}{4} \partial_{\lambda} H(U;\lambda)\right)
\\
&
=\lambda^{\gamma-1}
\left(\gamma D(U;\lambda)+\frac{\gamma^2}{4}H(U;\lambda)+\lambda \partial_{\lambda} D(U;\lambda) 
+\frac{\gamma\lambda}{4} \partial_{\lambda}H(U;\lambda) \right).
\end{aligned}
\end{equation}
Since it follows from \eqref{eq:4.5} that
\[
\bigg(\frac{1}{2}-\frac{1}{p+1}\bigg)\gamma-\frac{2s+\ell}{p+1}
=\frac{p-1}{2(p+1)}\frac{2(2s+\ell)}{p-1}-\frac{2s+\ell}{p+1}=0,
\]
by Lemma~\ref{Lemma:4.1} and \eqref{eq:2.22}, 
we see that
\begin{align*}
&
\lambda^{N-2s}
\left( \gamma D(U;\lambda)+\frac{\gamma^2}{4}H(U;\lambda)
+\lambda \partial_{\lambda} D(U;\lambda) 
+\frac{\gamma \lambda}{4} \partial_{\lambda}H(U;\lambda)\right)
\\
= \ & 
\gamma \left[\frac{1}{2}\int_{B^+_\lambda}t^{1-2s}|\nabla U|^2\,dX
-\frac{\kappa_s}{p+1}\int_{B_\lambda}|x|^\ell|u|^{p+1}\,dx \right]
+\frac{\gamma^2}{4}\lambda^{-1}\int_{S^+_\lambda}t^{1-2s}U^2\,dS
\\
&\quad
+\lambda\int_{S^+_\lambda}t^{1-2s} \left|\frac{\partial U}{\partial\nu} \right|^2\,dS
-\kappa_s\frac{2s+\ell}{p+1}\int_{B_\lambda}|x|^\ell|u|^{p+1}\,dx
+\frac{\gamma}{2}\int_{S^+_\lambda}t^{1-2s}U\frac{\partial U}{\partial\nu}\,dS
\\
= \ & \lambda \left[ \int_{ S_\lambda^+} t^{1-2s} 
\left\{ \frac{\gamma^2}{4} \left( \frac{U}{\lambda} \right)^2 
+ \gamma \frac{U}{\lambda} \frac{\partial U}{\partial \nu}
+ \left( \frac{\partial U}{\partial \nu} \right)^2
 \right\} \, dS \right]
 \\
 &\quad +\kappa_s\bigg\{\bigg(\frac{1}{2}-\frac{1}{p+1}\bigg) \gamma-\frac{2s+\ell}{p+1}\bigg\}
 \int_{B_\lambda}|x|^\ell|u|^{p+1}\,dx
\\
= \ &\lambda\int_{S^+_\lambda}t^{1-2s}\bigg|\frac{\gamma}{2}\frac{U}{\lambda}+\frac{\partial U}{\partial\nu}\bigg|^2\,dS.
\end{align*}
This together with \eqref{eq:4.6} and $\partial U/\partial\nu=\partial U/\partial r$ on $S^+_\lambda$ 
implies \eqref{eq:4.4}. 
	\end{proof}

Similar to \cite[Theorem~5.1]{DDW}, we prove the nonexistence result of 
solutions which have a special form and are stable outside $B_{R_0}^+$.

\begin{lemma}
\label{Lemma:4.3}
Let $R_0>0$ and $p_S(N,\ell)<p$. Suppose \eqref{eq:1.11} and that 
$W$ satisfies the following: 
	\begin{equation}\label{eq:4.7}
		\left\{\begin{aligned}
			& W(X) = r^{ - \frac{2s+\ell}{p-1} } \psi (\sigma) \in \Hloc ( \overline{ \HS} ,t^{1-2s} dX ) , 
			\\
			&\psi(\omega,0):=\psi|_{ \partial S^{+}_1}  \in L^{p+1} ( \partial S^{+}_1 ), 
			\\
			&  \int_{  \HS} t^{1-2s} \nabla W \cdot \nabla \Phi \, dX 
			= \kappa_s \int_{  \RN} |x|^\ell |W(x,0)|^{p-1} W(x,0) \Phi(x,0) \, dx
			\\
			&\hspace{5cm} \text{for each $\Phi \in C^1_c( \overline{ \HS } )$},
			\\
			& \kappa_s p \int_{\mathbb R^N} |x|^\ell |W(x,0)|^{p-1} \Phi(x,0)^2 \, dx 
			\leq \int_{  \HS} t^{1-2s} \left| \nabla \Phi \right|^2 \, dX 
			\\
			&\hspace{5cm} \text{for each $\Phi \in C^1_c( \overline{ \HS } \setminus \overline{ B^+_{R_0}} )$}.
		\end{aligned}\right.
	\end{equation}
Then $W \equiv 0$. 
\end{lemma}
\begin{remark}\label{Remark:4.1}
	\begin{enumerate}
	\item 
	The function $W$ in Lemma \ref{Lemma:4.4} is not necessarily defined through the form 
	$W=P_s(\cdot, t) \ast u$ where $u$ is a solution of \eqref{eq:1.1}. 
	\item 
	By $p_S(N,\ell)<p$, we have
		\[
			\ell - \frac{p}{p-1} (2s+\ell) = - \frac{2sp + \ell }{p-1} > -N, 
			\quad |x|^\ell |W(x,0)|^{p-1} W(x,0) \in L^1_{\rm loc} (\RN).
		\]
	\item 
	Set 
		\[
			\begin{aligned}
				H^1( S^+_1 , \sigma_{N+1}^{1-2s}dS ) 
				&:= \overline{ C^1 ( \overline{ S^+_1 }  ) }^{ \| \cdot \|_{H^1(S^+_1 , \sigma_{N+1}^{1-2s}dS)} }, 
				\\
				\| u \|_{H^1(S^+_1, \sigma_{N+1}^{1-2s}dS)}^2 
				&:= \int_{ S^+_1 } \sigma_{N+1}^{1-2s} 
				\left[ \left| \nabla_{S^+_1} u \right|^2 + u^2 \right] dS
			\end{aligned}
		\]
	where $\nabla_{ S^+_1 }$ stands for the standard gradient on the unit sphere in $\R^{N+1}$. 
	From \cite[Lemma 2.2]{FF}, there exists the trace operator 
	$H^1(S_1^+,\sigma_{N+1}^{1-2s}dS) \to L^2( \partial S^1_+ )$. 
	\item 
	Since 
		\[
			- \diver \left( t^{1-2s} \nabla W \right) = 0 \quad \text{in} \ \HS, \quad 
			W \in \Hloc \left( \overline{ \HS },t^{1-2s}dX \right),
		\]
	elliptic regularity yields $W = r^{ - \frac{2s+\ell}{p-1} } \psi(\sigma) \in C^\infty (\HS)$. 
	In addition, from $W \in \Hloc ( \overline{ \HS},t^{1-2s}dX)$, 
	we see $\psi \in H^1(S^+_1,\sigma_{N+1}^{1-2s}dS)  $. Next, for $k \geq 1$, consider 
		\[
			W_k(X) := \max \left\{ -k , \, \min \left\{ |X|^{ \frac{2s+\ell}{p-1} } W(X), k  \right\} \right\}.
		\]
	Then 
		\[
			W_k \in H^1 ( \overline{ B^+_2 } \setminus \overline{ B^+_{1/2} },t^{1-2s}dX) \cap 
			L^\infty ( \overline{ B^+_2 } \setminus \overline{ B^+_{1/2} } ), \quad 
			\left| W_k(X) \right| \leq \left| X \right|^{ \frac{2s+\ell}{p-1} } \left| W(X) \right|.  
		\]
	
	From this fact, we may find $( \psi_k )_k$ satisfying 
		\begin{equation}\label{eq:4.8}
			\begin{aligned}
				& \psi_k \in H^1( S^+_1,\sigma_{N+1}^{1-2s}dS) \cap L^\infty(S^+_1), \quad 
				\left| \psi_k(\sigma) \right| \leq \left| \psi(\sigma) \right| 
				\quad \text{for any $\sigma \in S^+_1$}, 
				\\
				&
				\left\| \psi_k - \psi \right\|_{H^1(S^+_1,\sigma_{N+1}^{1-2s}dS)} \to 0, \quad 
				\psi_k (\omega,0) \to \psi(\omega,0) \quad 
				\text{strongly in } L^{p+1} (\partial S^+_1).
			\end{aligned}
		\end{equation}
	\item 
		If $\varphi \in H^1(S^+_1,\sigma_{N+1}^{1-2s}dS) \cap L^\infty(S_1^+)$, then 
		we may find $(\varphi_k) \subset C^1( \overline{S^+_1} )$ such that 
			\begin{equation}\label{eq:4.9}
				\sup_{k \geq 1} \| \varphi_k \|_{L^\infty( S^+_1 )} < \infty, \quad  
				\| \varphi_k - \varphi \|_{H^1( S^+_1,\sigma_{N+1}^{1-2s}dS )} \to 0. 
			\end{equation}
		From the trace operator, we also have $\varphi_k(\omega,0) \to \varphi (\omega,0)$ 
		in $L^2(\partial S^+_1)$. 
	\end{enumerate}

\end{remark}

Even though a proof of Lemma~\ref{Lemma:4.3} is similar to the proof of \cite[Theorem~5.1]{DDW},
for the sake of completeness, we give the proof here.
Before a proof of Lemma \ref{Lemma:4.3}, we recall \cite[Lemma 2.1]{FF}:

	\begin{lemma}\label{Lemma:4.4}
		For $v (X) = f(r) \psi(\sigma) \in C^\infty (\HS)$, 
		\[
		- \diver \left( t^{1-2s} \nabla v \right) = 
		-r^{-N} \left( r^{N+1-2s} f_r(r) \right)_r \sigma_{N+1}^{1-2s} \psi(\sigma)
		- r^{-1-2s} f(r) \diver_{S^+_1} \left( \sigma_{N+1}^{1-2s} \nabla_{S^+_1} \psi \right)
		\]
		where $\diver_{S^+_1}$ is the standard 
		divergence operator on the unit sphere in $\R^{N+1}$. 
	\end{lemma}

	\begin{proof}[Proof of Lemma \ref{Lemma:4.3}]
Let $W = r^{ - \frac{2s+\ell}{p-1} } \psi (\sigma) $ be as in the statement. 
We divide our arguments into several steps. 

\medskip

\noindent
\textbf{Step 1:} \textsl{$\psi$ satisfies 
	\[
		\left\{\begin{aligned}
			&- \diver_{S^+_1} \left( \sigma_{N+1}^{1-2s} \nabla_{S^+_1} \psi \right) 
			+ \beta \sigma_{N+1}^{1-2s} \psi = 0 & &\mathrm{in} \ S^+_1,
			\\
			&- \lim_{ \sigma_{N+1} \to +0} \sigma_{N+1}^{1-2s} \partial_{\sigma_{N+1}} \psi 
			= \kappa_s |\psi |^{p-1} \psi  & &\mathrm{on} \ \partial S^+_1=S_1
		\end{aligned}\right.
	\]
where $\beta := \frac{2s+\ell}{p-1} \left( N  - \frac{2sp+\ell}{p-1} \right)$, namely, 
for each  $\varphi \in H^1( S^+_1,\sigma_{N+1}^{1-2s}dS) \cap L^\infty(S^+_1) $,
	\begin{equation}\label{eq:4.10}
		\int_{ S^+_1} \sigma_{N+1}^{1-2s} \nabla_{S^+_1} \psi \cdot \nabla_{S^+_1} \varphi 
		+ \beta \sigma_{N+1}^{1-2s} \psi \varphi \, d S
		= \kappa_s \int_{ \partial S^+_1} |\psi|^{p-1} \psi  \varphi
		\, d \omega.
	\end{equation}
Furthermore, we may also choose $\varphi = \psi$ in \eqref{eq:4.10} and obtain 
	\begin{equation}
	\label{eq:4.11}
	\int_{S^+_1} \sigma_{N+1}^{1-2s}|\nabla_{S^+_1}\psi|^2
	+\beta\sigma_{N+1}^{1-2s}\psi^2  \,dS
	=\kappa_s\int_{ \partial S^+_1 }|\psi|^{p+1}\,d\omega.
	\end{equation}
}

\begin{proof}
For \eqref{eq:4.10}, 
by \eqref{eq:4.9}, $\psi(\omega,0) \in L^{p+1}(\partial S^+_1)$ and the dominated convergence theorem, 
it is enough to prove it for $\varphi \in C^1( \overline{S^+_1} )$.

For $V(X) = V(r \sigma) \in C^1_c( \overline{ \HS } )$, notice that 
	\begin{equation}\label{eq:4.12}
	\begin{aligned}
		\nabla V(X)& = \partial_r V( r \sigma ) \sigma + r^{-1} \nabla_{S^+_1} V, 
		\\
		\nabla W(X)& = r^{ - \frac{2s+\ell}{p-1} - 1 } 
		\left[ - \frac{2s+\ell}{p-1} \psi(\sigma) \sigma + \nabla_{S^+_1} \psi(\sigma) \right].
		\end{aligned}
	\end{equation}
We see from \eqref{eq:4.7}, \eqref{eq:4.12} and $\sigma \cdot \nabla_{S^+_1} h (\sigma) = 0$ 
for functions $h$ on $S^+_1$ that 
	\begin{equation}\label{eq:4.13}
		\begin{aligned}
			&\kappa_s \int_{  \RN} |x|^{ - \frac{2sp + \ell}{p-1} } |\psi(x/|x|,0)|^{p-1} \psi(x/|x|,0) V(x,0) \, dx 
			\\
			= \ &
			\int_{  \HS} t^{1-2s} \nabla W \cdot \nabla V \, dX 
			\\
			= \ & 
			\int_{  \HS} t^{1-2s} r^{ - \frac{2s+\ell}{p-1} - 1 } 
			\left[ - \frac{2s+\ell}{p-1}  \partial_rV (r \sigma) \psi(\sigma ) 
			+ r^{-1} \nabla_{S^+_1} V \cdot \nabla_{S^+_1} \psi 
			\right] \, dX.
		\end{aligned}
	\end{equation}
If we choose $V$ as $V(X) = \eta(r) \varphi(\sigma)$ where $\eta \in C^1_{c} ([0,\infty))$ 
and $\varphi \in C^1 ( \overline{ S^+_1} )$, then \eqref{eq:4.13} is rewritten as 
	\begin{equation}\label{eq:4.14}
		\begin{aligned}
			&\kappa_s \left(  \int_0^\infty r^{ - \frac{2sp+\ell}{p-1} + N - 1 } \eta (r) \,dr \right)
			\left( \int_{ \partial S^+_1} |\psi( \omega , 0 )|^{p-1} \psi (\omega, 0) \varphi (\omega,0) \, d \omega
			\right)
			\\
			= \ & - \frac{2s+\ell}{p-1} \left( \int_0^\infty r^{ N - \frac{2sp+\ell}{p-1}} 
			\eta'(r) \, d r \right) 
			\left( \int_{ S^+_1} \sigma_{N+1}^{1-2s} \varphi(\sigma) \psi(\sigma) \, d S \right)
			\\
			& \quad + \left( \int_0^\infty r^{ - \frac{2sp+\ell}{p-1} + N -1 } \eta(r) \, dr \right)
			\left( \int_{ S^+_1} \sigma_{N+1}^{1-2s} \nabla_{S^+_1} \varphi \cdot \nabla_{S^+_1} \psi \, d S \right). 
		\end{aligned}
	\end{equation}
Since $p_S(N,\ell) < p$ yields $ - \frac{2sp+\ell}{p-1} + N > 0$, it follows from the integration by parts that 
	\[
		-\frac{2s+\ell}{p-1} \int_0^\infty r^{N-\frac{2sp+\ell}{p-1}} \eta'(r) \, dr
		= \beta \int_0^\infty r^{N-\frac{2sp+\ell}{p-1}-1} \eta(r) \, dr.
	\]
Thus, by choosing $\eta \geq 0$ with $\eta \not\equiv 0$, \eqref{eq:4.14} implies 
	\[
		\kappa_s \int_{ \partial S^+_1} |\psi(\omega,0)|^{p-1} \psi(\omega,0) \varphi (\omega,0) \, d\omega 
		= \int_{ S^+_1}\sigma_{N+1}^{1-2s} \nabla_{S^+_1} \psi \cdot \nabla_{S^+_1} \varphi 
		+ \beta \sigma_{N+1}^{1-2s} \psi \varphi \, d S
	\]
for every $\varphi \in C^1( \overline{ S^+_1} )$. Hence, \eqref{eq:4.10} holds.

	For \eqref{eq:4.11}, take $(\psi_k)_k$ satisfying \eqref{eq:4.8}. 
Then \eqref{eq:4.10} holds for $\varphi = \psi_k$. 
Thanks to $\psi(\omega,0) \in L^{p+1} (\partial S_1^+)$ and the dominated convergence theorem, 
letting $k \to \infty$, we obtain \eqref{eq:4.11}. 
\end{proof}

\medskip

\noindent
\textbf{Step 2:} 
\textsl{For every $\varphi \in H^1( S^+_1,\sigma_{N+1}^{1-2s}dS) \cap L^\infty(S^+_1)$, 
	\begin{equation}
	\label{eq:4.15}
	\kappa_sp\int_{ \partial S^+_1 }|\psi|^{p-1}\varphi^2\,d \omega
	\le \int_{ S^+_1 } 
	\sigma_{N+1}^{1-2s} |\nabla_{S^+_1}\varphi|^2\,dS 
	+\left(\frac{N-2s}{2}\right)^2\int_{ S^+_1} \sigma_{N+1}^{1-2s}\varphi^2\,dS.
	\end{equation}
}

\begin{proof}
It is enough to treat the case $\varphi \in C^1( \overline{ S^+_1} )$ due to \eqref{eq:4.9} as in Step 1. 
We recall the stability in \eqref{eq:4.7}: 
for any $\phi\in C^1_c(\overline{\mathbb R^{N+1}_+}\setminus \overline{B^+_{R_0}})$, 
\begin{equation}
\label{eq:4.16}
\kappa_sp\int_{\partial\mathbb R^{N+1}_+}|x|^\ell|W|^{p-1} \phi(x,0)^2\,dx
\le \int_{\mathbb R^{N+1}_+}t^{1-2s}|\nabla \phi|^2\,dX.
\end{equation}
For $0< \varepsilon \ll 1$, we choose $\tau_\e$ and a standard cutoff function 
$\eta_\varepsilon\in C^1_c( (0,\infty) )$ such that
\begin{equation}
\label{eq:4.17}
\tau_\e := \frac{1}{\sqrt{-\log \e}} \to 0, \qquad 
\begin{aligned}
&\chi_{(R_0+2 \tau_\e , \, R_0+ \e^{-1} )}(r)\le \eta_\varepsilon(r) 
\le \chi_{(R_0 + \tau_\e , \, R_0+2 \e^{-1} )} (r), 
\\
&
|\eta_\varepsilon'(r)|\le C \tau_\e^{-1}
\quad \mbox{for}\quad r\in(R_0+ \tau_\e , \, R_0+ 2 \tau_\e),
\\
&
|\eta_\varepsilon'(r)|\le C\varepsilon
\quad \mbox{for}\quad r\in(R_0+ \e^{-1} , \, R_0+2 \e^{-1})
\end{aligned}
\end{equation}
where $\chi_A(r)$ is the characteristic function of $A \subset (0,\infty)$.
For $\varphi\in C^1( \overline{S^+_1} )$, we put
\begin{equation}
\label{eq:4.18}
\phi(X)=r^{-\frac{N-2s}{2}}\eta_\epsilon(r)\varphi(\sigma )\qquad\mbox{for}\quad 
X = r \sigma \in\mathbb R^{N+1}_+.
\end{equation}
Since $W=r^{ - \frac{2s+\ell}{p-1} } \psi (\sigma)$, we have 
\begin{equation}
\label{eq:4.19}
\int_{\partial\mathbb R^{N+1}_+}|x|^\ell|W|^{p-1}\phi^2\,dx
=\left(\int_0^\infty r^{-1}\eta_\varepsilon^2\,dr\right)
\left(\int_{ \partial S^+_1 }|\psi|^{p-1}\varphi^2\,d\omega\right).
\end{equation}

	On the other hand, by \eqref{eq:4.18}, we see that
\begin{equation}
\label{eq:4.20}
\begin{aligned}
|\nabla\phi (X)|^2
&
= \left( \left(r^{-\frac{N-2s}{2}}\eta_\varepsilon \right)' \right)^2\varphi^2
+r^{-2}\left(r^{-\frac{N-2s}{2}}\eta_\varepsilon\right)^2 |\nabla_{S^+_1}\varphi|^2
\\
&
=\left[ \left(\frac{N-2s}{2} \right)^2\varphi^2+ | \nabla_{S^+_1}\varphi |^2\right] 
r^{-2-(N-2s)}\eta_\varepsilon^2
\\
&\qquad
+r^{-(N-2s)}(\eta_\varepsilon')^2\varphi^2-(N-2s)r^{-1-(N-2s)}\eta_\varepsilon\eta_\varepsilon'\varphi^2.
\end{aligned}
\end{equation}
Since it follows from \eqref{eq:4.17} that
\begin{equation*}
	\begin{aligned}
		&\int_0^\infty r^{N+1-2s}r^{-(N-2s)}(\eta_\varepsilon')^2\,dr
		\leq 
		C \tau_\e^{-2} \int_{R_0+\tau_\e}^{R_0+2\tau_\e} r \, dr 
		+C\varepsilon^2\int_{R_0 + 1/\e}^{R_0 + 2/\varepsilon} r\,dr
		\leq C \left( \tau_\e^{-1} + 1 \right),
		\\
		&
		\left| \int_0^\infty r^{N+1-2s}r^{-1-(N-2s)}\eta_\varepsilon\eta_\varepsilon'\,dr \right|
		\leq C \tau_\e^{-1} \int_{R_0 + \tau_\e }^{R_0 + 2 \tau_\e } \,dr
		+C\varepsilon\int_{R_0 + 1/\e}^{R_0 + 2 / \e} \,dr
		\le C,
	\end{aligned}
\end{equation*}
by \eqref{eq:4.20} we have
\begin{equation}
\label{eq:4.21}
\begin{aligned}
&
\int_{\mathbb R^{N+1}_+}t^{1-2s}|\nabla \phi|^2\,dX
\\
= \ & \int_0^\infty\, dr\, \int_{ S^+_1 } r^N ( r \sigma_{N+1})^{1-2s}
\left\{ \left[ \left(\frac{N-2s}{2} \right)^2\varphi^2 
+|\nabla_{S^+_1}\varphi|^2\right]r^{-2-(N-2s)}\eta_\varepsilon^2 \right. 
\\
&\hspace{4cm} 
+ r^{-(N-2s)}(\eta_\varepsilon')^2\varphi^2-(N-2s)r^{-1-(N-2s)}\eta_\varepsilon\eta_\varepsilon'\varphi^2\Bigg\}
\,dS
\\
\le \ & \left(\int_0^\infty r^{-1}\eta_\varepsilon^2\,dr\right)
\left(\int_{S^+_1} \sigma_{N+1}^{1-2s}
\left[ \left( \frac{N-2s}{2} \right)^2\varphi^2+|\nabla_{S^+_1}\varphi|^2\right]\,dS\right)
+C \left(  \tau_\e^{-1} + 1 \right).
\end{aligned}
\end{equation}
Finally, remark that 
	\[
		\int_0^\infty r^{-1} \eta_\e^2 \, dr 
		\geq \int_{R_0 + 2 \tau_\e}^{R_0+\e^{-1}} r^{-1} \, dr 
		= \log \left( R_0 + \e^{-1} \right) - \log \left( R_0 + 2 \tau_\e \right) 
		\geq \frac{\tau_\e^{-2}}{2} 
	\]
holds for sufficiently small $\e$. Therefore, 
substituting \eqref{eq:4.19} and \eqref{eq:4.21} to \eqref{eq:4.16}, 
dividing by $\int_0^\infty r^{-1} \eta_\e^2 \, dr $ and 
taking $\varepsilon\to0$,
we obtain \eqref{eq:4.15}. 
\end{proof}

\medskip

\noindent
\textbf{Step 3:} 
\textsl{For $\alpha\in [0 , \frac{N-2s}{2} )$, set 
	\[
		v_\alpha(x) := |x|^{ - \left( \frac{N-2s}{2} - \alpha \right) }, \quad 
		V_\alpha (X) := ( P_s (\cdot, t) \ast v_\alpha ) (x).
	\]
Then for each $X \in \HS$ and $\lambda > 0$, 
\begin{equation}
\label{eq:4.22}
V_\alpha(\lambda X)=\lambda^{-\frac{N-2s}{2}+\alpha}V_\alpha(X)
\end{equation}
and $\phi_{\alpha}(\sigma) := V_\alpha ( \sigma) = V_\alpha (r^{-1} X) 
\in C( \overline{ S^+_1 } ) \cap C^1( S^+_1  ) \cap H^1( S^+_1,\sigma_{N+1}^{1-2s}dS) $. 
Moreover, $\phi_{\alpha}$ also satisfies 
$\phi_{\alpha} > 0$ in $\overline{S^+_1}$, 
for any $\varphi \in H^1(S^+_1,\sigma_{N+1}^{1-2s}dS) \cap L^\infty(S^+_1)$, 
\begin{equation}
\label{eq:4.23}
	\begin{aligned}
		&
		\int_{S_1^+} \sigma_{N+1}^{1-2s} |\nabla_{S^+_1}\varphi  |^2\,dS 
		+ \left[ \left(\frac{N-2s}{2}\right)^2-\alpha^2\right] 
		\int_{S^+_1} \sigma_{N+1}^{1-2s}\varphi^2\,dS
		\\
		= \ &
		 \kappa_s \lambda(\alpha) \int_{\partial S^+_1} \varphi^2\, d\omega
		+\int_{S^+_1} \sigma_{N+1}^{1-2s}\phi_\alpha^2
		\left|\nabla_{S^+_1} \left( \frac{\varphi}{\phi_\alpha} \right) \right|^2\,dS
	\end{aligned}
\end{equation}
and 
	\begin{equation}
		\label{eq:4.24}
		0 \leq \alpha_1 \leq \alpha_2 < \frac{N-2s}{2} \quad \Rightarrow\quad 
		\phi_{\alpha_1} \le \phi_{\alpha_2} \qquad \mathrm{in} \quad S^+_1.
	\end{equation}
}

	\begin{proof}
By direct computation, we may check \eqref{eq:4.22}. 
For the assertion 
$\phi_{\alpha} \in C( \overline{ S^+_1 } ) \cap C^1( S^+_1  ) \cap H^1( S^+_1,\sigma_{N+1}^{1-2s}dS)$, 
we remark $V_\alpha(X)
 = (P_s(\cdot,t) \ast v_\alpha )(x) 
\in C^\infty(\HS)$.
By $\phi_\alpha=V_\alpha|_{ S^+_1}$,
to show  
$\phi_{\alpha} \in C( \overline{ S^+_1 } ) \cap C^1( S^+_1  ) \cap H^1( S^+_1,\sigma_{N+1}^{1-2s}dS)$, 
it suffices to prove 
	\begin{equation}\label{eq:4.25}
		V_\alpha, \  t^{1-2s} \partial_tV_\alpha, \ \nabla_x V_\alpha \in C 
		\left( \overline{ N^+_{1/4} (\partial S^+_1)} \right), \quad 
		N^+_r(A) := \Set{ X \in \HS | \, \dist \left( X, A \right) < r  }.
	\end{equation}
To this end, decompose $v_\alpha = v_{\alpha,1} + v_{\alpha,2}$ where 
$ \supp v_{\alpha,1} \subset B_{2} \setminus B_{1/4}$ and $v_{\alpha,1} \equiv v_\alpha$ 
on $B_{3/2} \setminus B_{1/2}$, and set $V_{\alpha,i} (X) := (P_s(\cdot, t) \ast v_{\alpha,i})(x)$. 
Since $v_{\alpha,1} \in C^\infty_c(\RN)$ and 
$N_{1/4} (\partial S^+_1 ) \cap \supp v_{\alpha,2} = \emptyset$ where 
$N_r(A) := \set{x \in \RN |\, \dist (x,A) < r }$, 
it is not difficult to show \eqref{eq:4.25} and 
$\phi_{\alpha} \in C( \overline{ S^+_1 } ) \cap C^1( S^+_1  ) \cap H^1( S^+_1,\sigma_{N+1}^{1-2s}dS)$.

	The assertion $\phi_{\alpha} > 0$ in $\overline{ S^+_1}$ follows from 
$v_\alpha > 0$ in $\RN \setminus \{0\}$ and the definition of $V_\alpha$.

	For \eqref{eq:4.23} and \eqref{eq:4.24}, remark that 
in \cite[Lemma~4.1]{Fall}, it is proved that $(-\Delta)^s v_\alpha = \lambda(\alpha) |x|^{-2s} v_\alpha$
in $\RN$,
where $\lambda(\alpha)$ appears in \eqref{eq:1.7}. 
Hence, by the property of $P_s(x,t)$ and $v_\alpha$, we may check that 
for each $\varphi \in C^\infty_c(\RN \setminus \set{0} )$, 
	\begin{equation}\label{eq:4.26}
		- \lim_{t \to +0} \int_{  \RN} t^{1-2s} \partial_t V_\alpha (X) \varphi (x) \, d x 
		= \kappa_s \int_{  \RN} v_\alpha (-\Delta)^s \varphi \, d x 
		= \kappa_s \int_{  \RN} \lambda (\alpha) |x|^{-2s} v_\alpha \varphi \, dx.
	\end{equation}
	For any fixed $\omega \in \partial S^+_1$, 
we consider a curve 
	\[
		\gamma_{\omega} (\tau ) := 
		\begin{pmatrix}
			\sqrt{1-\tau^{2}} \, \omega \\ \tau
		\end{pmatrix}
		 \in S^+_1.
	\]
Then $\phi_\alpha ( \gamma_{\omega} (\tau) ) = V_\alpha (\gamma_{\omega} (\tau))$ and 
	\[
		\frac{d}{d \tau} V_\alpha ( \gamma_{\omega} (\tau ) ) 
		= \nabla V_\alpha ( \gamma_{\omega} (\tau) ) \cdot 
		\begin{pmatrix}
			- \frac{\tau}{\sqrt{1-\tau^2}} \omega \\ 1
		\end{pmatrix}
		= - \frac{\tau}{\sqrt{1-\tau^2}} \nabla_x V_\alpha( \gamma_{\omega} (\tau) ) \cdot \omega 
		+ \partial_t V_\alpha( \gamma_{\omega} (\tau) ).
	\]
Combining this fact with \eqref{eq:4.25}, $v_\alpha(1) = 1$ and \eqref{eq:4.26}, we deduce that 
	\begin{equation}\label{eq:4.27}
		-\lim_{\sigma_{N+1}\to0} \sigma_{N+1}^{1-2s} \partial_{\sigma_{N+1}} \phi_\alpha (\sigma)
		=
		\kappa_s\lambda(\alpha)\qquad\mbox{on}\quad \partial S^+_1.
	\end{equation}
Due to \eqref{eq:4.22}, we notice that 
\[
V_\alpha(X)=r^{-\frac{N-2s}{2}+\alpha}\phi_\alpha(\sigma).
\]
Furthermore, since $0=- \diver ( t^{1-2s} \nabla V_\alpha)$ in $\HS$ and 
$V_\alpha=v_\alpha$ on $\partial\mathbb R^{N+1}_+\setminus\{0\}$,
we have $\phi_\alpha=v_\alpha=1$ on $\partial S^+_1$,
and by Lemma \ref{Lemma:4.4} with $f(r) = r^{ - \frac{N-2s}{2} + \alpha }$ and 
$\psi (\sigma) = \phi_{\alpha}(\sigma)$, $\phi_{\alpha}$ is a solution of 
\begin{equation}
\label{eq:4.28}
\left\{\begin{aligned}
 - \diver_{S^+_1} \left( \sigma_{N+1}^{1-2s} \nabla_{S^+_1}  \phi_\alpha \right) 
+ \left[ \left( \frac{N-2s}{2}\right)^2-\alpha^2\right] \sigma_{N+1}^{1-2s}\phi_\alpha 
&= 0
& &\mbox{in}\quad S^+_1,
\\
\phi_\alpha&=1
& &
\mbox{on}\quad \partial S^+_1.
\end{aligned}
\right.
\end{equation}

	Now we  prove \eqref{eq:4.23}. 
Since $\phi_{\alpha} \in C( \overline{ S^+_1} )$ and $\phi_{\alpha} > 0$ in $ \overline{ S^+_1}$, 
for every $\varphi \in H^1(S^+_1,\sigma_{N+1}^{1-2s}dS) \cap L^\infty (S^+_1)$ and 
$(\varphi_k)_k$ with \eqref{eq:4.9}, it follows that 
	\[
		\frac{\varphi_k}{\phi_{\alpha}} \to \frac{\varphi}{\phi_{\alpha}} \quad 
		\text{strongly in} \ H^1(S^+_1,\sigma_{N+1}^{1-2s}dS), \quad 
		\frac{\varphi}{\phi_{\alpha}} \in H^1(S^+_1,\sigma_{N+1}^{1-2s}dS). 
	\]
Hence, it suffices to show \eqref{eq:4.23} for $\varphi \in C^1( \overline{ S^+_1 } )$. 
For $\varphi\in C^1( \overline{S^+_1} )$, notice that
	\[
		\begin{aligned}
		\nabla_{S^+_1} \phi_\alpha \cdot \nabla_{S^+_1} 
		\left( \frac{\varphi^2}{\phi_\alpha} \right)
		&= \nabla_{S^+_1} \phi_{\alpha} \cdot 
		\left[ \frac{2 \varphi \nabla_{S^+_1} \varphi}{\phi_{\alpha}} 
		- \frac{\varphi^2 \nabla_{S^+_1} \phi_\alpha}{\phi_{\alpha}^2} \right]
		= |\nabla_{S^+_1} \varphi |^2 
		- \left|\nabla_{S^+_1} \left(\frac{\varphi}{\phi_\alpha} \right) \right|^2 
		\phi_\alpha^2.
		\end{aligned}
	\]
Thus, multiplying \eqref{eq:4.28} by $\varphi^2/\phi_\alpha$, we see from \eqref{eq:4.27} that 
\eqref{eq:4.23} holds.

	Finally, for $0 \leq \alpha_1 \leq \alpha_2 < \frac{N-2s}{2}$, we infer from 
$0< \phi_{\alpha}$ that 
	\[
		\begin{aligned}
			- {\rm{div}}_{S^+_1} \left( \sigma_{N+1}^{1-2s} \nabla_{S^+_1} \phi_{\alpha_1} \right) 
			&= 
			- \left[ \left( \frac{N-2}{2} \right)^2  - \alpha_1^2 \right] \sigma_{N+1}^{1-2s}  \phi_{\alpha_1}
			\\
			&\leq - \left[ \left(\frac{N-2s}{2}\right)^2 - \alpha_2^2 \right] \sigma_{N+1}^{1-2s}\phi_{\alpha_1}
			\quad\mbox{on}\quad S^+_1,
		\end{aligned}
	\]
which yields 
	\[
		- \diver_{S^+_1} 
		\left( \sigma_{N+1}^{1-2s} \nabla_{ S^+_1 } \left( \phi_{\alpha_2} - \phi_{\alpha_1} \right) \right) 
		+ \left[ \left(\frac{N-2s}{2}\right)^2 - \alpha_2^2 \right] \sigma_{N+1}^{1-2s} 
		\left( \phi_{\alpha_2} - \phi_{\alpha_1} \right) \geq 0.
	\]
Multiplying this inequality by $( \phi_{\alpha_2} - \phi_{\alpha_1} )_- :=\max\{0,-(\phi_{\alpha_2} - \phi_{\alpha_1})\}\in H^1( S^+_1,\sigma_{N+1}^{1-2s}dS)$ 
and integrating it over $S^+_1$, by $\phi_{\alpha_1} = 1 = \phi_{\alpha_2}$ on $\partial S^+_1$ 
and $\left( \frac{N-2s}{2} \right)^2 - \alpha^2_2 > 0$, we deduce that 
$ ( \phi_{\alpha_2} - \phi_{\alpha_1} )_- \equiv 0$, hence, \eqref{eq:4.24} holds.
\end{proof}

\medskip

\noindent
\textbf{Step 4: } \textsl{Conclusion}

\medskip 

Now we are ready to prove the assertion of Lemma \ref{Lemma:4.3}. 
By $p_S(N,\ell)<p$ and $\ell>-2s$ thanks to \eqref{eq:1.2}, 
we set 
	\begin{equation*}
		\widetilde{\alpha}=\frac{N-2s}{2}-\frac{2s+\ell}{p-1} \in \left( 0 , \frac{N-2s}{2} \right).
	\end{equation*}
By this choice of $\widetilde{\alpha}$, we see that 
\begin{equation}
\label{eq:4.29}
\left(\frac{N-2s}{2}\right)^2-\widetilde{\alpha}^2=\frac{2s+\ell}{p-1}\left(N-2s-\frac{2s+\ell}{p-1}\right)=\beta
\end{equation}
where $\beta$ appears in Step 1. 
Let $(\psi_k)_k$ be the functions in \eqref{eq:4.8} and notice that 
$\phi_{0} / \phi_{\widetilde{\alpha}} \in C( \overline{ S^1_+ } ) \cap H^1(S^+_1,\sigma_{N+1}^{1-2s}dS)$, 
$\phi_0=\phi_{\widetilde{\alpha}}=1$ on $\partial S^+_1$ and 
$\psi_k \phi_0 / \phi_{\widetilde{\alpha}} \in H^1( S^+_1,\sigma_{N+1}^{1-2s}dS) \cap L^\infty(S^+_1) $. 
Hence, \eqref{eq:4.15} and \eqref{eq:4.23} with $\varphi = \psi_k \phi_0 / \phi_{\widetilde{\alpha}}$ 
and $\alpha = 0$ give 
	\begin{equation*}
		\begin{aligned}
			&
			\kappa_s p\int_{\partial S^+_1} |\psi|^{p-1} \psi_k^2 \,d \omega
			\\
			\leq \ & 
			\int_{S^+_1} 
			\sigma_{N+1}^{1-2s} 
			\left| \nabla_{S^+_1} \left( \frac{\psi_k \phi_0}{\phi_{\widetilde{\alpha}}} \right) \right|^2 
			\, dS 
			+ \left( \frac{N-2s}{2} \right)^2 
			\int_{S^+_1} \sigma_{N+1}^{1-2s} \left( \frac{\psi_k \phi_0}{\phi_{\widetilde{\alpha}}} \right)^2\,dS 
			\\
			= \ & \kappa_s \lambda(0) \int_{ \partial S^+_1} \psi_k^2 \, d \omega 
			+ \int_{ S^+_1} \sigma_{N+1}^{1-2s} \phi_{0}^2 
			\left| \nabla_{ S^+_1 } \left( \frac{\psi_k}{\phi_{\widetilde{\alpha}}} \right) \right|^2 \,dS
\end{aligned}
\end{equation*}
This together with \eqref{eq:4.24} implies
\begin{equation}
\label{eq:4.30}
\kappa_s p \int_{\partial S^+_1} |\psi|^{p-1} \psi_k^2 \,d \omega 
\le \kappa_s \lambda(0) \int_{\partial S^+_1} \psi^2_k\,d \omega
+\int_{ S^+_1 } \sigma_{N+1}^{1-2s} \phi_{\widetilde{\alpha}}^2 
\left|\nabla_{ S^+_1} \left(\frac{\psi_k }{\phi_{\widetilde{\alpha}}} \right) \right|^2\,dS.
\end{equation}
Substituting \eqref{eq:4.23} with $\varphi = \psi_k$ and $\alpha = \widetilde{\alpha}$ 
into \eqref{eq:4.30}, we observe from \eqref{eq:4.29} that 
\begin{equation}\label{eq:4.31}
\begin{aligned}
&
\kappa_sp\int_{\partial S^+_1} |\psi|^{p-1} \psi_k^2 \,d \omega
\\
\le \ & 
\kappa_s \lambda(0) \int_{\partial S^+_1} \psi^2_k \,d \omega  
+  \int_{ S^+_1 } \sigma_{N+1}^{1-2s}|\nabla_{ S^+_1 }\psi_k|^2\,dS
 \\
 &\hspace{1.5cm}
 +\left[ \left( \frac{N-2s}{2}\right)^2-\widetilde{\alpha}^2\right] 
 \int_{S^+_1} \sigma_{N+1}^{1-2s}\psi^2_k \,dS
 -\kappa_s\lambda( \widetilde{\alpha} )\int_{\partial S^+_1}\psi^2_k \,d \omega
\\
= \ & \kappa_s \left( \lambda(0) - \lambda(\widetilde{\alpha}) \right) 
\int_{\partial S^+_1} \psi^2_k \,d \omega 
+ \int_{S^+_1} \sigma_{N+1}^{1-2s}|\nabla_{S^+_1}\psi_k |^2\,dS
+\beta\int_{ S^+_1 } \sigma_{N+1}^{1-2s}\psi_k^2\,dS.
\end{aligned}
\end{equation}

	On the other hand, by \eqref{eq:4.23} with $\varphi = \psi_k$ and $\alpha = \widetilde{\alpha}$, 
we have 
	\[
		\begin{aligned}
			\int_{ S^+_1} \sigma_{N+1}^{1-2s} | \nabla_{ S^+_1 } \psi_k |^2 \, d S 
			+ \beta \int_{S^+_1} \sigma_{N+1}^{1-2s} \psi_k^2 \, dS 
			\geq 
			\kappa_s \lambda(\widetilde{\alpha}) 
			\int_{ \partial S^+_1} \psi_k^2 \, d \omega.
		\end{aligned}
	\]
From \eqref{eq:4.31} and the fact $\lambda (0) > \lambda(\widetilde{\alpha})$ due to \eqref{eq:1.8}, 
it follows that 
	\begin{equation*}
		\kappa_s p \int_{ \partial S^+_1} | \psi |^{p-1} \psi_k^2 \, d \omega 
		\leq \frac{\lambda (0)}{\lambda(\widetilde{\alpha})} 
		\left\{ \int_{ S^+_1} \sigma_{N+1}^{1-2s} | \nabla_{ S^+_1 } \psi_k |^2 \, d S 
		+ \beta \int_{S^+_1} \sigma_{N+1}^{1-2s} \psi_k^2 \, dS\right\}.
	\end{equation*}
Letting $k \to \infty$ and noting \eqref{eq:4.8} and \eqref{eq:4.11}, we obtain 
	\[
		\kappa_s p \int_{ \partial S^+_1} | \psi |^{p+1} \, d \omega 
				\leq \frac{\lambda (0)}{\lambda(\widetilde{\alpha})} 
			\kappa_s \int_{ \partial S^+_1} |\psi|^{p+1} \, d \omega. 
	\]
Thus, we obtain $\lambda(\widetilde{\alpha})p\le \lambda(0)$
unless $\psi\equiv 0$.
Therefore, if \eqref{eq:1.11} holds, namely
$\lambda(\widetilde{\alpha})p> \lambda(0)$,
then $\psi\equiv0$, and by $W=r^{ -\frac{2s+\ell}{p-1} } \psi$, we have $W\equiv 0$.
Hence, Lemma~\ref{Lemma:4.3} follows. 
	\end{proof}

Now we are ready to prove Theorem~\ref{Theorem:1.1} for $p_S(N,\ell)<p$.
Following \cite{DDW}, we use the blow-down analysis.

	\begin{proof}[Proof of Theorem~\ref{Theorem:1.1} for $p_S(N,\ell)<p$] 
Assume \eqref{eq:1.2} and \eqref{eq:1.11}.
Let $u \in \Hsloc (\RN) \cap L^\infty_{\rm loc} (\RN) \cap L^2(\RN,(1+|x|)^{-N-2s}dx) $ 
be a solution  of \eqref{eq:1.1} which is stable outside $B_{R_0}$ and 
let $U$ be the function given in \eqref{eq:2.3}. 
Recall $D$, $H$ and $E$ in \eqref{eq:4.1}, \eqref{eq:4.2} and \eqref{eq:4.3}, respectively.
Then, by Lemma~\ref{Lemma:2.9} we see that
\begin{equation}
\label{eq:4.32}
\lambda^{\frac{2(2s+\ell)}{p-1}}D(U;\lambda)
\le C\lambda^{\frac{2}{p-1}(s(p+1)+\ell)-N}\left(\int_{B^+_\lambda}t^{1-2s}|\nabla U|^2\, dX
+\int_{B_\lambda}|x|^\ell|u|^{p+1}\,dx\right)
\le C
\end{equation}
for $\lambda\ge 3 R_0$.
Since $E$ is nondecreasing with respect to $\lambda$ due to Lemma \ref{Lemma:4.2}, 
by \eqref{eq:4.32} and Lemma \ref{Lemma:2.8}, for $\lambda \geq 3R_0$, we have 
\begin{equation*}
\begin{aligned}
E(U;\lambda)
&
\le \lambda^{-1} \int_\lambda^{2\lambda}E(U;\xi)\,d\xi
\\
&
= \lambda^{-1} \int_\lambda^{2\lambda} 
\xi^{\frac{2(2s+\ell)}{p-1}} \left[D(U;\xi)+\frac{2s+\ell}{2(p-1)}H(U;\xi) \right]\,d\xi
\\
&\leq C + C \lambda^{-1} \int_{ \lambda}^{ 2\lambda} d \xi \, \xi^{ \frac{2(2s+\ell)}{p-1} - N - 1 + 2s } 
\int_{ S_\xi^+} t^{1-2s} U^2 \, dS
\\
&
\le C+C\lambda^{ \frac{2}{p-1}(2s+\ell) + 2s -N-2 }\int_{\lambda}^{2\lambda} d\xi 
\int_{S^+_\xi}t^{1-2s}U^2\,dS 
\\
&
\le C+C\lambda^{ \frac{2}{p-1}(2s+\ell) + 2s -N-2 }\int_{B_{2\lambda}^+ }t^{1-2s}U^2\,dX
\\
&
\le C+C\lambda^{ \frac{2}{p-1}(2s+\ell) + 2s -N-2 } 
\lambda^{N + 2 (1-s) - \frac{2(2s+\ell)}{p-1}  } 
\le C.
\end{aligned}
\end{equation*}
This implies that
\begin{equation}
\label{eq:4.33}
\lim_{\lambda\to\infty}E(U;\lambda)<+\infty.
\end{equation}

On the other hand, 
for $X\in\mathbb R^{N+1}_+$, let
\[
V_\lambda (X):=\lambda^{\frac{2s+\ell}{p-1}}U(\lambda X).
\]
Then it is easy to check that 
	\begin{equation}\label{eq:4.34}
		\begin{aligned}
			&V_\lambda(X) = 
			\left( P_s(\cdot, t) \ast \left( \lambda^{ \frac{2s+\ell}{p-1} } u \left( \lambda \cdot \right) \right) \right) (x), 
	 		\\
			&
	 		- \lim_{t \to +0} t^{1-2s} \partial_t V_\lambda(x,t) = \kappa_s |x|^\ell \left| V_\lambda(x,0) \right|^{p-1} 
			V_\lambda (x,0),
			\\
			&\lambda^{\frac{2(2s+\ell)}{p-1}}D(U;\lambda R)=D(V_\lambda ;R), 
			\,\,\, \lambda^{ \frac{2(2s+\ell)}{p-1} } H(U;\lambda R) = H(V_\lambda; R), \,\,\,
			E(U;\lambda R)=E(V_\lambda;R)
		\end{aligned}
	\end{equation}
for each $\lambda \geq 3R_0$. 
Since $u$ is stable outside $B_{R_0}$, 
as in the proof of Lemma \ref{Lemma:2.9} (see \eqref{eq:2.65}), by \eqref{eq:2.8}, 
$U$ is stable outside $B_{R_0}^+$. 
Therefore, for every $\psi \in C^1_c( \overline{ \HS} \setminus \overline{ B^+_{\lambda^{-1} R_0 } } )$, 
	\begin{equation}\label{eq:4.35}
		\begin{aligned}
			p \kappa_s \int_{  \RN} |x|^\ell | V_\lambda(x,0) |^{p-1} \psi(x,0)^2 \, dx 
			&= p \kappa_s  \lambda^{2s-N} \int_{  \RN} |x|^\ell |u(x)|^{p-1} \psi( \lambda^{-1} x,0 )^2 \, dx
			\\
			&\leq \lambda^{2s-N} \int_{  \HS} t^{1-2s} \left| \nabla ( \psi(\lambda^{-1} X) ) \right|^2 \, dX 
			\\
			&= \int_{  \HS}  t^{1-2s} \left| \nabla \psi \right|^2 \, dX,
		\end{aligned}
	\end{equation}
which implies that $V_\lambda$ is stable outside $B^+_{\lambda^{-1} R_0 }$. 
Furthermore, by \eqref{eq:2.52} and \eqref{eq:2.60}, 
$(V_\lambda)_{\lambda \geq 3 R_0}$ is bounded in $\Hloc(\overline{ \HS},t^{1-2s}dX)$ and 
$(V_\lambda(x, 0))_{\lambda \geq 3R_0}$ is bounded in $L^{p+1}_{\rm loc} (\RN,|x|^\ell dx)$.

	Now let $(\lambda_i)_{i=1}^\infty$ satisfy $\lambda_i \to \infty$ and 
$V_{\lambda_i} \rightharpoonup U_\infty$ weakly in $\Hloc ( \overline{\HS},t^{1-2s}dX)$. 
Thanks to the above fact, without loss of generality, we may also assume that 
	\begin{equation}\label{eq:4.36}
		U_\infty(x,0) \in L^{p+1}_{\rm loc} ( \RN,|x|^\ell dx).
	\end{equation}
We shall claim that 
		\begin{align}
			&V_{\lambda_i} (x,0) \to U_\infty(x,0) & &\text{strongly in } L^q_{\rm loc} (\RN) 
			\quad \text{for $1 \leq q < \frac{2N}{N-2s}$}, 
			\label{align:4.37}
			\\
			&V_{\lambda_i}(X) \to U_\infty(X) & &\text{strongly in } L^2_{\rm loc} ( \overline{ \HS},t^{1-2s}dX).
			\label{align:4.38}
		\end{align}
Due to the boundedness of the trace operator from $H^1( B_R \times (0,R),t^{1-2s}dX)$ 
to $H^s(B_R)$ for each $R$ (see \cite{DD-12}), 
we also have $V_{\lambda_i} (x,0) \rightharpoonup U_\infty (x,0)$ 
weakly in $H^s(B_R)$. 
By the compactness of embedding $H^s(B_R) \subset L^q(B_R)$ where $1 \leq q < 2N/(N-2s)$, 
we get \eqref{align:4.37}. 
For \eqref{align:4.38}, 
since 
$H^1( B_R \times (R^{-1},R), t^{1-2s}dX) = H^1( B_R \times (R^{-1},R) )$,
we first remark that
\begin{equation}
\label{eq:4.39}
V_{\lambda_i} \to U_\infty \qquad\mbox{strongly in $L^2_{\rm loc} (\HS,t^{1-2s}dX)$}.
\end{equation}
Around $t = 0$, we notice that for $\psi \in C^1( \overline{ \HS})$, 
	\[
		\begin{aligned}
			\left| \psi(x,t) \right| 
			&\leq 
			\left| \psi (x,0) \right| 
			+ \int_0^t \tau^{\frac{2s-1}{2}} \tau^{\frac{1-2s}{2}} \left| \partial_t \psi (x,\tau) \right| \, d \tau 
			\\
			&\leq |\psi(x,0)| + \left[ \frac{1}{2s} t^{2s} \right]^{1/2} 
			\left( \int_0^t \tau^{1-2s} \left| \partial_t \psi (x, \tau  ) \right|^2 \, d \tau \right)^{1/2}.
		\end{aligned}
	\]
Therefore, 
	\begin{equation}\label{eq:4.40}
		\begin{aligned}
			&
			\int_{B_R \times (0,T)} t^{1-2s} \left| \psi (X) \right|^2 \, dX 
			\\
			&\leq 2 \int_{B_R \times (0,T)} t^{1-2s} \left[ |\psi(x,0)|^2 
			+ \frac{t^{2s}}{s} \int_0^t \tau^{1-2s} |\partial_t \psi (x,\tau) |^2 \, d \tau   \right] dX
			\\
			& \leq C_s\left[ T^{2-2s} \| \psi(\cdot, 0) \|_{L^2(B_R)}^2 
			+T^2 \int_{B_R\times(0,T)} t^{1-2s} |\partial_t \psi (X)|^2 \, dX  \right].
		\end{aligned}
	\end{equation}
By the density argument, \eqref{eq:4.40} holds for every $W \in \Hloc (\overline{\HS},t^{1-2s}dX)$. 
Thus, by \eqref{eq:4.39},
	\[
		\begin{aligned}
			&
			\limsup_{i \to \infty} \int_{ B_R \times (0,R) } t^{1-2s} | V_{\lambda_i} - U_\infty |^2 \, dX
			\\ 
			= \ & \limsup_{i \to \infty} \left( \int_{B_R \times (0,T)} + \int_{B_R \times (T,R)} \right) 
			t^{1-2s} | V_{\lambda_i} - U_\infty |^2 \, dX 
			\leq C (T^{2-2s} + T^2 ).
		\end{aligned}
	\]
Since $T \in (0,R)$ is arbitrary and $C$ is independent of $T$, \eqref{align:4.38} holds.

	Next, we shall prove that $U_\infty$ satisfies \eqref{eq:4.7}. 
For the third property in \eqref{eq:4.7}, we observe from \eqref{eq:2.18} and \eqref{eq:4.34} that 
for each $\varphi \in C^1_c( \overline{ \HS} )$, 
	\[
		\int_{  \HS} t^{1-2s} \nabla V_{\lambda_i} \cdot \nabla \varphi \, dX 
		= \kappa_s \int_{  \RN} |x|^\ell |V_{\lambda_i} (x,0)|^{p-1} V_{\lambda_i} (x,0) \varphi (x,0) \, dx.
	\]
By \eqref{align:4.37}, we may also suppose that $V_{\lambda_i} (x,0) \to U_\infty(x,0)$ for a.a. $x \in \RN$. 
Since $(V_{\lambda_i}(x,0))_i$ is bounded in $L^{p+1}_{\rm loc} (\RN,|x|^\ell dx)$, 
a variant of Strauss' lemma (see Strauss \cite[Compactness Lemma 2]{Str-77} and 
Berestycki and Lions \cite[Theorem A.I]{BL}) 
and the fact $V_{\lambda_i} \rightharpoonup U_\infty$ weakly in $\Hloc( \overline{ \HS}, t^{1-2s}dX)$ 
give 
	\[
		\int_{  \HS} t^{1-2s} \nabla U_\infty \cdot \nabla \varphi \, dX 
		= \kappa_s \int_{  \RN} |x|^\ell |U_{\infty} (x,0)|^{p-1} U_{\infty} (x,0) \varphi (x,0) \, dx.
	\]
In a similar way, by \eqref{eq:4.35}, we also observe that $U_\infty$ is stable outside $B_\varepsilon^+$ 
for any $\varepsilon > 0$, that is, for each $\psi \in C^1_c( \overline{ \HS} \setminus \overline{ B^+_{\e}} )$, 
	\[
		\kappa_s p \int_{  \RN} |x|^\ell |U_\infty(x,0)|^{p-1} \psi(x,0)^2 \, dx 
		\leq \int_{  \HS} t^{1-2s} |\nabla \psi|^2 \, dX. 
	\]

Finally, we prove $U_\infty(X) = r^{- \frac{2s+\ell}{p-1}} U_\infty( r^{-1} X )$. 
If this is true, then \eqref{eq:4.36} gives $\psi(\omega, 0) = U_\infty(x/r,0) \in L^{p+1}(\partial S_1^+)$ and 
Lemma \ref{Lemma:4.3} is applicable for $U_\infty$. Remark that \eqref{eq:4.34} implies 
	\begin{equation}\label{eq:4.41}
		\left( - \Delta \right)^s V_\lambda ( x,0) = |x|^\ell \left| V_\lambda (x,0) \right|^{p-1} V_\lambda (x,0) 
		\quad \text{in} \ \RN
	\end{equation}
and Proposition \ref{Proposition:2.2} and Lemma \ref{Lemma:4.2} hold for $V_\lambda$. Hence, 
for $R_2>R_1>0$, by \eqref{eq:4.33}, \eqref{eq:4.34} and Lemma \ref{Lemma:4.2}, we have
\begin{equation}\label{eq:4.42}
\begin{aligned}
0
=\lim_{i\to\infty}\bigg\{E(U;\lambda_iR_2)-E(U;\lambda_iR_1)\bigg\}
&=\lim_{i\to\infty}\bigg\{E(V_{\lambda_i};R_2)-E(V_{\lambda_i};R_1)\bigg\}
\\
&
\ge\liminf_{i\to\infty}\int_{R_1}^{R_2}\frac{\partial}{\partial r}E(V_{\lambda_i};r)\,dr.
\end{aligned}
\end{equation}
This together with \eqref{eq:4.4}, \eqref{align:4.38} and the weak lower semicontinuity of norms yield 
\begin{equation}
\label{eq:4.43}
\begin{aligned}
0
&
\ge
\liminf_{i\to\infty}\int_{R_1}^{R_2}r^{\frac{2}{p-1}(s(p+1)+\ell)-N}
\left( \int_{S^+_r}t^{1-2s}\left(\frac{2s+\ell}{p-1}\frac{V_{\lambda_i}}{r} 
+\frac{\partial V_{\lambda_i}}{\partial r}\right)^2\, dS
\right)dr
\\
&
=
\liminf_{i\to\infty}\int_{B^+_{R_2}\setminus B^+_{R_1}}t^{1-2s}r^{\frac{2}{p-1}(s(p+1)+\ell)-N}
\left(\frac{2s+\ell}{p-1}\frac{V_{\lambda_i}}{r}+\frac{\partial V_{\lambda_i}}{\partial r}\right)^2\,dX
\\
&
\ge
\int_{B^+_{R_2}\setminus B^+_{R_1}}t^{1-2s}r^{\frac{2}{p-1}(s(p+1)+\ell)-N}
\left(\frac{2s+\ell}{p-1}\frac{U_\infty}{r}+\frac{\partial U_\infty}{\partial r}\right)^2\,dX.
\end{aligned}
\end{equation}
Noting that $U_\infty \in C^\infty(\HS)$ thanks to 
$\mathrm{div} (t^{1-2s} \nabla U_\infty) = 0$ and elliptic regularity, 
by the arbitrariness of $R_1$ and $R_2$, we have 
	\[
	0 = 
	\frac{\partial U_\infty}{\partial r}+\frac{2s+\ell}{p-1}\frac{U_\infty}{r} 
	= r^{ - \frac{2s+\ell}{p-1} } \frac{\partial}{\partial r}\left( r^{ \frac{2s+\ell}{p-1} } U_\infty \right)
	\quad \mbox{in}\quad \mathbb R^{N+1}_+.
	\]
Integrating this equality with respect to $r$, we obtain
	\[
		U_\infty(X)=r^{-\frac{2s+\ell}{p-1}}U_\infty(r^{-1}X)
	\]
and hence, $U_\infty$ satisfies \eqref{eq:4.7}.

	It follows from Lemma~\ref{Lemma:4.3} that $U_\infty\equiv0$. 
Since the weak limit does not depend on choices of subsequences, we infer that 
$V_\lambda \rightharpoonup 0$ weakly in $\Hloc  (  \overline{ \HS},t^{1-2s}dX)$ and 
from \eqref{align:4.38} that 
	\begin{equation}\label{eq:4.44}
		\int_{ B_{2R}^+ } t^{1-2s} |V_\lambda|^2 \, dX \to 0 \quad \mbox{as}\quad \lambda \to \infty.
	\end{equation}
Recalling \eqref{eq:4.34}, \eqref{eq:4.35}, \eqref{eq:4.41} and 
the proof of \eqref{eq:2.66} in Lemma \ref{Lemma:2.9}, we see
	\begin{equation}\label{eq:4.45}
	\int_{\mathbb R^{N+1}_+}t^{1-2s}|\nabla(V_\lambda \zeta)|^2\,dX 
	\leq 
	\frac{p}{p-1}\int_{\mathbb R^{N+1}_+}t^{1-2s}|V_{\lambda}|^2|\nabla\zeta|^2\,dX
	\end{equation}
where $\zeta \in C^1_c( \overline{ \HS} ) $ satisfying 
	\[
		\zeta \equiv 1\ \ \text{in} \ B_{R}^+ \setminus B_{r}^+, \quad 
		\zeta \equiv 0 \quad \text{in} \ B_{r/2}^+ \cup ( \HS \setminus B_{2R}^+ )
		\quad \mbox{for}\quad \lambda^{-1} R_0 < r < R.
	\]
From \eqref{eq:4.44}, \eqref{eq:4.45} and the property of $\zeta$, we observe that 
for any $0 < r < R$, 
	\begin{equation}\label{eq:4.46}
		\lim_{ \lambda \to \infty} 
		\int_{ B_R^+ \setminus B_r^+ } t^{1-2s} \left| \nabla V_{\lambda} \right|^2 \, dX  = 0.
	\end{equation}
Furthermore, by \eqref{eq:2.64} with $V_\lambda$, \eqref{eq:4.44} and \eqref{eq:4.45}, for each $0 < r < R$, 
	\begin{equation}\label{eq:4.47}
		\lim_{ \lambda \to \infty}\int_{B_R \setminus B_r} |x|^\ell |V_{\lambda}(x,0)|^{p+1} \, dx = 0.
	\end{equation}

	Next, we shall prove $E(U;\lambda) \to 0$ as $\lambda \to \infty$. 
In view of \eqref{eq:4.34}, for each $\e \in (0,1)$, we have 
\begin{equation}\label{eq:4.48}
\begin{aligned}
\lambda^{\frac{2(2s+\ell)}{p-1}}D(U;\lambda)
= \ &D(V_\lambda ;1)
\\
= \ &\frac{1}{2}\int_{B^+_\varepsilon}t^{1-2s}|\nabla V_\lambda|^2\,dX
-\frac{\kappa_s}{p+1}\int_{B_\varepsilon}|x|^\ell|V_\lambda (x,0)|^{p+1}\,dx
\\
&
\quad
+\frac{1}{2}\int_{B^+_1\setminus B^+_\varepsilon}t^{1-2s}|\nabla V_\lambda|^2\,dX
-\frac{\kappa_s}{p+1}\int_{B_1\setminus B_\varepsilon}|x|^\ell|V_\lambda (x,0) |^{p+1}\,dx
\\
= \ & \e^{N-2s} D(V_\lambda; \e) 
+\frac{1}{2}\int_{B^+_1\setminus B^+_\varepsilon}t^{1-2s}|\nabla V_\lambda|^2\,dX
\\
&\quad 
-\frac{\kappa_s}{p+1}\int_{B_1\setminus B_\varepsilon}|x|^\ell|V_\lambda (x,0) |^{p+1}\,dx
\\
= \ &\varepsilon^{N-\frac{2}{p-1}(s(p+1)+\ell)}
\left[ (\lambda\varepsilon)^{\frac{2(2s+\ell)}{p-1}} D(U;\lambda\varepsilon)\right] 
\\
& 
\quad 
+\frac{1}{2}\int_{B^+_1\setminus B^+_\varepsilon}t^{1-2s}|\nabla V_\lambda|^2\,dX
-\frac{\kappa_s}{p+1}\int_{B_1\setminus B_\varepsilon}|x|^\ell|V_\lambda (x,0) |^{p+1}\,dx.
\end{aligned}
\end{equation}
By \eqref{eq:4.32} and \eqref{eq:4.46}--\eqref{eq:4.48}, we see that 
	\[
		\begin{aligned}
			&\limsup_{ \lambda \to \infty} 
			\left| \lambda^{\frac{2(2s+\ell)}{p-1}}D(U;\lambda) \right| 
			\\
			= \ &
			\limsup_{ \lambda \to \infty} 
			\left|  \varepsilon^{N-\frac{2}{p-1}(s(p+1)+\ell)}
			\left[ (\lambda\varepsilon)^{\frac{2(2s+\ell)}{p-1}} D(U;\lambda\varepsilon)\right] 
			\right.
			\\
			& \hspace{3cm} 
			\left. +\frac{1}{2}\int_{B^+_1\setminus B^+_\varepsilon}t^{1-2s}|\nabla V_\lambda|^2\,dX
			-\frac{\kappa_s}{p+1}\int_{B_1\setminus B_\varepsilon}|x|^\ell|V_\lambda (x,0) |^{p+1}\,dx
			 \right|
			 \\
			 \leq \ & C_0 \varepsilon^{N-\frac{2}{p-1}(s(p+1)+\ell)}
		\end{aligned}
	\]
for some $C_0>0$. Since $\e \in (0,1)$ is arbitrary and $N - \frac{2}{p-1} ( (p+1) s + \ell ) > 0$, we obtain 
\begin{equation}
\label{eq:4.49}
\lim_{\lambda\to\infty}\lambda^{\frac{2(2s+\ell)}{p-1}}D(U;\lambda) = 0.
\end{equation}
On the other hand, from \eqref{eq:4.44} it follows that 
	\[
		0 = \lim_{ \lambda \to \infty} \int_{B_2^+} t^{1-2s} |V_\lambda|^2 \,dX 
		= \lim_{\lambda\to\infty} \int_0^2 dr \int_{S_r^+} t^{1-2s} |V_\lambda|^2 \, dS. 
	\]
By choosing a subsequence $(\lambda_i)$, 
	\[
		\int_{S_r^+} t^{1-2s} |V_{\lambda_i}|^2 dS \to 0 \quad \text{a.a. $r \in (0,2)$}.
	\]
Therefore, there exists an $r_0 \in (0,2)$ such that \eqref{eq:4.34} gives 
	\[
		\lambda_i^{ \frac{2(2s+\ell)}{p-1} } H(U; r_0 \lambda_i)  
		= H(V_{\lambda_i} ; r_0 ) = r_0^{ -(N+1-2s) } \int_{ S_{r_0}^+} t^{1-2s} |V_{\lambda_i}|^2 \,dS \to 0.
	\]
With \eqref{eq:4.34}, \eqref{eq:4.49} and the monotonicity of $E(U;\lambda)$, we have 
	\begin{equation}\label{eq:4.50}
		\lim_{ \lambda \to \infty} E(U;\lambda) 
		= \lim_{ i \to \infty} E(U;\lambda_i r_0)= 0.
	\end{equation}

	We shall prove $U \equiv 0$. To this end, we shall prove 
$E(U;\lambda) \to 0$ as $\lambda \to 0$. 
Since $U \in C ( \overline{\HS} )$ holds by Lemma \ref{Lemma:2.3}, as $\lambda \to 0$, we have 
	\begin{equation}\label{eq:4.51}
		0 \leq \lambda^{\frac{2(2s+\ell)}{p-1}} H(U;\lambda) 
		= \lambda^{ \frac{2(2s+\ell)}{p-1} } 
		\int_{ S^+_1} \sigma_{N+1}^{1-2s} U(\lambda \sigma)^2 \, dS 
		\leq C \| U \|_{ L^\infty( B_1^+ ) }^2 \lambda^{ \frac{2(2s+\ell)}{p-1} } \to 0
	\end{equation}
and by \eqref{eq:1.2}, 
	\begin{equation}\label{eq:4.52}
		\lambda^{ \frac{2(2s+\ell)}{p-1} - (N-2s)} \int_{ B_\lambda} |x|^\ell |u|^{p+1} \, dx 
		\leq C \| u \|_{ L^\infty(B_1) }^{p+1} \lambda^{ \frac{2(2s+\ell)}{p-1} + 2s + \ell } \to 0.
	\end{equation}
By \eqref{eq:4.50} and the monotonicity of $E(U;\lambda)$, we have 
$E(U;\lambda) \leq 0$ for all $\lambda \in (0,\infty)$. 
In view of this fact with \eqref{eq:4.51} and \eqref{eq:4.52}, it follows that 
	\[
		\begin{aligned}
			&\limsup_{\lambda \to + 0} \frac{\lambda^{ \frac{2(2s+\ell)}{p-1} - (N-2s) } }{2} 
			\int_{B^+_\lambda} t^{1-2s} |\nabla U|^2 \, dX 
			\\
			= \ & \limsup_{\lambda \to + 0} 
			\left[ E(U;\lambda) + \frac{\lambda^{ \frac{2(2s+\ell)}{p-1} - (N-2s)  } \kappa_s }{p+1} 
			\int_{B_\lambda} |x|^\ell |u|^{p+1} \, dx - 
			\lambda^{ \frac{2(2s+\ell)}{p-1} } \frac{2s+\ell}{2(p-1)} H(U;\lambda) 
			 \right] 
			 \\
			 \leq \ & 0,
		\end{aligned}
	\]
which yields  
	\begin{equation}\label{eq:4.53}
		\lim_{\lambda \to +0} 
		\lambda^{ \frac{2(2s+\ell)}{p-1} - (N-2s)} \int_{ B_\lambda^+} t^{1-2s} |\nabla U|^2 \, dX 
		= 0.
	\end{equation}
Now, \eqref{eq:4.51}--\eqref{eq:4.53} yields $E(U;\lambda) \to 0$ $(\lambda \to 0)$. 
Therefore, $E(U; \lambda) \equiv 0$ thanks to \eqref{eq:4.50} and the monotonicity of $E$. 
Applying the argument similar to \eqref{eq:4.42} and \eqref{eq:4.43},
we see that $U = r^{ - \frac{2s+\ell}{p-1} } U( r^{-1} X ) $.  In addition, 
since $U\in C ( \overline{\HS}) $, $\psi(\omega,0) = U(\omega,0) \in L^\infty(\partial S^+_1)$ and 
$u$ (respectively $U$) is stable outside $B_{R_0}$ (respectively $B_{R_0}^+$), 
\eqref{eq:4.7} is satisfied and 
it follows from Lemma~\ref{Lemma:4.3} that $U\equiv 0$. 
This completes the proof. 
\end{proof}
\noindent
\textbf{Acknowledgment.} 
The authors would like to express their sincere gratitude to the anonymous referee 
for his/her careful reading and useful comments. 
They also would like to thank Alexander Quaas for letting them know \cite{BQ-20}. 
The first author (S.H.) was supported by JSPS KAKENHI Grant Numbers JP 
20J01191.
The second author (N.I.) was supported by JSPS KAKENHI Grant Numbers JP 
17H02851, 19H01797 and 19K03590. 
The third author (T.K.) was supported by JSPS KAKENHI Grant Numbers JP 
19H05599 and 16K17629. 


\end{document}